\documentclass[reqno,12pt]{amsart}
 \usepackage{amsmath}
 \usepackage{amsthm}
 \usepackage{amssymb}
 \usepackage{latexsym,longtable}
 \usepackage{graphicx}
 \usepackage{multicol}
 \usepackage{mathrsfs}
 \usepackage{young}
 \usepackage[vcentermath,enableskew,stdtext]{youngtab}
 \usepackage[left=1.1in,right=1.1in,top=1.1in,bottom=1.14in]{geometry}
 \usepackage[all]{xy}
    \SelectTips{cm}{10}     
    \everyxy={<2.5em,0em>:} 
 \usepackage{fancyhdr}      
 \usepackage{multicol}
 \usepackage{multirow}
 \usepackage{enumitem}
 \usepackage{stmaryrd}
 \usepackage{etex}
 \usepackage{tikz}
 \usepackage{float}
 \usepackage{array}
 \usepackage{colortbl}
 \usepackage{mathtools}
 \restylefloat{figure}
 \usetikzlibrary{shapes}
 \usetikzlibrary{positioning}
 \usetikzlibrary{snakes}
 \usetikzlibrary{decorations.pathmorphing}
 \usetikzlibrary{decorations.markings}

\usepackage{amsfonts,epsfig,color}
\makeatletter
\newcommand*{\centerfloat}{%
  \parindent \z@
  \leftskip \z@ \@plus 1fil \@minus \textwidth
  \rightskip\leftskip
  \parfillskip \z@skip}
\makeatother
\newcounter{ctr}
\theoremstyle{plain}
\newtheorem{theorem}{Theorem}[section]
\newtheorem{lemma}[theorem]{Lemma}
\newtheorem{corollary}[theorem]{Corollary}
\newtheorem{proposition}[theorem]{Proposition}

\newtheorem{problem}[theorem]{Problem}

\theoremstyle{definition}
\newtheorem{definition}[theorem]{Definition}

\newtheorem{remark}[theorem]{Remark}
\newtheorem{example}[theorem]{Example}

\newcommand{\ignore}[1]{}

\renewcommand{\a}{\ensuremath{\mathfrak{a}}}

\newcommand{\CC}{\ensuremath{\mathbb{C}}}

\newcommand{\ZZ}{\ensuremath{\mathbb{Z}}}

\newcommand{\be}{\begin{equation}}
\newcommand{\ee}{\end{equation}}

\renewcommand{\S}{\ensuremath{\mathcal{S}}}
\newcommand{\tsr}{\ensuremath{\otimes}}

\newcommand{\tto}{\ensuremath{\rightsquigarrow}}

\newcommand{\reading}{\text{\rm rowword}}

\newcommand{\ke}{\ensuremath{\sim}}
\newcommand{\sh}{\text{\rm sh}}
\newcommand{\evac}{\text{\rm ev}}
\newcommand{\transpose}{\ensuremath{\text{\rm t}}}

\newcommand{\ngcell}{\text{-}}

\newcommand{\tc}{\ensuremath{\text{tc}}}
\newcommand{\crc}[1]{\ensuremath{\overline{#1}}\vphantom{\underline{\overline{#1}}}}
\newcommand{\crcempty}{{\crc{\phantom{{}_\circ}}}}
\newcommand{\sub}{\ensuremath{\text{\rm sub}}}
\newcommand{\convert}[2]{\ensuremath{{#2}(#1)}}
\newcommand{\mix}{\ensuremath{\text{\rm m}}}
\newcommand{\dualmix}{\ensuremath{\text{\rm dm}}}
\newcommand{\lr}{\ensuremath{\text{\rm lr}}} 
\newcommand{\stand}{\ensuremath{\text{\rm st}}} 

\newcommand{\bneg}{\ensuremath{{\text{\rm neg}}}}
\newcommand{\lessdoteq}{\ensuremath{\underline{\lessdot}}}
\newcommand{\sw}{\ensuremath{\text{SW}}}
\newcommand{\mboxtwo}[1]{\makebox[4mm][r]{$#1$}}
\newcommand{\plain}{\ensuremath{{\text{\rm blft}}}}
\newcommand{\plainr}{\ensuremath{{\text{\rm brgt}}}}
\newcommand{\rev}{\ensuremath{{\text{\rm rev}}}}
\newcommand{\gcoef}{\ensuremath{g_{\lambda\,\mu(d)\,\nu}}}
\newcommand{\ud}{\ensuremath{{\text{\rm ud}}}} 
\newcommand{\inv}{\ensuremath{{\text{\rm inv}}}}
\newcommand{\barrev}{\ensuremath{{\text{\rm rev}_\crcempty}}}
\newcommand{\barud}{\ensuremath{{{\text{\rm ud}_\crcempty}}}} 

\newcommand{\nobarrev}{\ensuremath{{{\text{\rm rev}_\varnothing}}}}
\newcommand{\nobarud}{\ensuremath{{{\text{\rm ud}_\varnothing}}}} 

\newcommand{\nmix}{_\mix} 
\newcommand\mybox[1]{
\vcenter{
\let\\=\cr
\baselineskip=-16000pt \lineskiplimit=16000pt \lineskip=0pt
\halign{&\boxcell{##}\cr\vline#1\vline\crcr}}}

\newcommand{\boxcell}[1]{{%
\unitlength=\cellsizeCol
\begin{picture}(1,1)
\put(0,0){\makebox(1,1){$#1$}}
\put(0,0){\line(1,0){1}}
\put(0,1){\line(1,0){1}}
\end{picture}%
}}

\newcommand\pad[1]{
\vtop{
\let\\=\cr
\baselineskip=-16000pt
\lineskiplimit=16000pt
\lineskip=0pt
\halign{& \inviscell{##} \cr #1 \crcr} }
\hspace{-.73ex}}

\newcommand{\inviscell}[1]{{%
\unitlength=\cellsizeCol
\begin{picture}(1,1)
\put(0,0){\makebox(1,1){$#1$}}
\end{picture}%
}}

\newlength{\cellsize}
\newcommand\tableau[1]{
\vcenter{
\let\\=\cr
\baselineskip=-16000pt \lineskiplimit=16000pt \lineskip=0pt
\halign{&\tableaucell{##}\cr#1\crcr}}}

\newcommand{\tableaucell}[1]{{%
\def \arg{#1}\def \void{}%
\ifx \void \arg
\vbox to \cellsize{\vfil \hrule width \cellsize height 0pt}%
\else \unitlength=\cellsize
\begin{picture}(1,1)
\put(0,0){\makebox(1,1){$#1\vphantom{\crc{#1}}$}}
\put(0,0){\line(1,0){1}}
\put(0,1){\line(1,0){1}}
\put(0,0){\line(0,1){1}}
\put(1,0){\line(0,1){1}}
\end{picture}%
\fi}}

\newcommand\boldtableau[1]{
\vcenter{
\let\\=\cr
\baselineskip=-16000pt \lineskiplimit=16000pt \lineskip=0pt
\halign{&\boldtableaucell{##}\cr#1\crcr}}}

\newcommand{\boldtableaucell}[1]{{%
\def \arg{#1}\def \void{}%
\ifx \void \arg
\vbox to \cellsize{\vfil \hrule width \cellsize height 0pt}%
\else \unitlength=\cellsize
\begin{picture}(1,1)
\put(0,0){\makebox(1,1){$\mathbf{#1\vphantom{\crc{#1}}}$}}
\put(0,0){\line(1,0){1}}
\put(0,1){\line(1,0){1}}
\put(0,0){\line(0,1){1}}
\put(1,0){\line(0,1){1}}
\end{picture}%
\fi}}

\newlength{\colskip}
\newlength{\dwidth}

\newcommand{\lms}{\mathbf{\{}} 
\newcommand{\rms}{\mathbf{\}}}

\title{Kronecker coefficients for one hook shape}


\begin{document}

\author{Jonah Blasiak}
\email{jblasiak@gmail.com}
\address{University of Michigan, Department of Mathematics, 2074 East Hall, 530 Church Street, Ann Arbor, MI 48109-1043}
\thanks{This work was partially supported by NSF Grant DMS-1161280.}

\begin{abstract}
We give a positive combinatorial formula for the Kronecker coefficient $\gcoef$ for any partitions $\lambda$, $\nu$ of  $n$ and hook shape $\mu(d) := (n-d,1^d)$.
Our main tool is Haiman's \emph{mixed insertion}.  This is a generalization of Schensted insertion to \emph{colored words}, words in the alphabet of barred letters  $\crc{1},\crc{2}, \ldots$ and unbarred letters $1,2, \ldots$.
We define the set of \emph{colored Yamanouchi tableaux of content $\lambda$ and total color $d$}  ($\text{CYT}_{\lambda, d}$) to be the set of mixed insertion tableaux of colored words $w$ with exactly $d$ barred letters and such that $w^\plain$ is a Yamanouchi word of content $\lambda$, where $w^\plain$ is the ordinary word formed from $w$ by shuffling its barred letters to the left and then removing their bars.
We prove that $\gcoef$ is equal to the number of $\text{CYT}_{\lambda, d}$ of shape $\nu$ with unbarred southwest corner.
\end{abstract}
\maketitle
\setlength{\cellsize}{2.2ex}
\section{Introduction}
Let  $\S_n$ be the symmetric group on $n$ letters and $M_\nu$ be the irreducible  $\CC \S_n$-module corresponding to the partition $\nu$.
Given three partitions $\lambda,\mu,\nu$ of  $n$,
the \emph{Kronecker coefficient $g_{\lambda \mu \nu}$} is the multiplicity of
$M_\nu$ in the tensor product $M_\lambda \otimes M_\mu$.  A fundamental open problem in algebraic combinatorics, called the \emph{Kronecker problem}, is to find a positive combinatorial formula for these coefficients.

This work began with the following computer experiment, first investigated by Lascoux in \cite{Lascoux}:
let $Z_\lambda$ be the superstandard tableau of shape and content $\lambda$ and $Z_{\lambda}^\stand$ its standardization.
Let $\Gamma_\lambda$ denote the set of permutations with insertion tableau $Z_\lambda^\stand$.
Form the multiset of permutations
\be \label{e lambda circ mu}
\Gamma_\lambda \circ \Gamma_\mu := \lms u \circ v : u \in \Gamma_\lambda, v \in \Gamma_\mu \rms,
\ee
where $\circ$ denotes multiplication in $\S_n$, i.e. composition of permutations.  Then form the multisets of insertion and recording tableaux:
\begin{align*}
P(\Gamma_\lambda \circ \Gamma_\mu) := \lms P(w) :  w \in \Gamma_\lambda \circ \Gamma_\mu \rms,\\
Q(\Gamma_\lambda \circ \Gamma_\mu) := \lms Q(w) :  w \in \Gamma_\lambda \circ \Gamma_\mu \rms.
\end{align*}
 The set $\Gamma_\lambda$  naturally labels a basis of $M_\lambda$. For instance,  $\Gamma_\lambda$ can be identified with a right cell of the $W$-graph  $\Gamma_W$  as defined by Kazhdan-Lusztig in \cite{KL}, for  $W = \S_n$.
A nice solution to the Kronecker problem might assign labels to a basis of $M_\lambda \tsr M_\mu$ so that the decomposition of $M_\lambda \tsr M_\mu$  into irreducibles is apparent from these labels.
The following two properties, if true for every partition $\nu$ of $n$, would make
$\Gamma_\lambda \circ \Gamma_\mu$ a beautifully simple candidate for such labels.
\begin{list}{(\Alph{ctr})}{\usecounter{ctr} \setlength{\itemsep}{2pt} \setlength{\topsep}{3pt}}
\item For every $T \in \text{SYT}(\nu)$, the multiplicity of $T$ in $P(\Gamma_{\lambda} \circ \Gamma_{\mu})$ is $g_{\lambda \mu \nu} f^\nu$ or 0.
\item For every $B_\nu \in \text{SYT}(\nu)$, the multiplicity of $B_\nu$ in  $Q(\Gamma_\lambda \circ \Gamma_\mu)$ is $g_{\lambda \mu \nu}$.
\end{list}
Here, $\text{SYT}(\nu)$ denotes the set of  standard Young tableau  of shape $\nu$ and $f^\nu := |\text{SYT}(\nu)|$.

\begin{theorem}[Lascoux's Kronecker Rule \cite{Lascoux}]
If $\lambda$ and  $\mu$ are hook shapes, then (A) and (B) hold for all $\nu$.
\end{theorem}
Lascoux \cite{Lascoux} and Garsia-Remmel \cite[\textsection6--7]{GarsiaRemmel} both investigate the extent to which this rule generalizes to other shapes. They give examples
showing that it does not extend beyond the hook hook case.
As far as we know, this approach to the Kronecker problem has not been pursued any further in the literature. 

Our computations indicate, however, that (B) is amazingly close to being true in general,  and we therefore believe that
there is much more to be gained from this experiment. To give an idea of how close (B) comes to holding for general shapes, let $m_{\lambda \, \mu \, B_\nu}$ denote the multiplicity in (B) and define the fractions
\[
\alpha_{\lambda \mu \nu} := \big|\big\{B_\nu \in \text{SYT}(\nu) : g_{\lambda \mu \nu} = m_{\lambda \, \mu \, B_\nu}\big\}\big| / f^\nu.
\]
Of the 42376 triples of partitions $\lambda, \mu, \nu$ of 10 for which either $g_{\lambda \mu \nu}$ or some $m_{\lambda \, \mu \, B_\nu}$ is nonzero, 11112 of them satisfy $\alpha_{\lambda \mu \nu} = 1$, 3703 of them satisfy $\alpha_{\lambda \mu \nu} \in [\frac{9}{10},1)$, etc., as indicated below. Note that the maximum size of a Kronecker coefficient for $n = 10$ is 117.
\[
\vspace{.7mm}
{\footnotesize
\begin{array}{rcccccccccccccc}
\{0\}\!&\!(0,\frac{1}{10})\!&\![\frac{1}{10},\frac{2}{10})\!&\![\frac{2}{10},\frac{3}{10})\!&\![\frac{3}{10},\frac{4}{10})\!&\![\frac{4}{10},\frac{5}{10})\!&\![\frac{5}{10},\frac{6}{10})\!&\![\frac{6}{10},\frac{7}{10})\!&\![\frac{7}{10},\frac{8}{10})\!&\![\frac{8}{10},\frac{9}{10})\!&\![\frac{9}{10},1)\!&\!\{1\}\\[2mm]
231\!&\!1558\!&\!3801\!&\!3413\!&\!2997\!&\!2792\!&\!2838\!&\!3216\!&\!3129\!&\!3586\!&\!3703\!&\!11112
\end{array}}
\]

This ``approximate rule'' does even better when  $\mu$ is a hook shape and, in fact, we conjecture that (B) holds for any $\nu$ when  $\lambda_2 \leq 2$ and $\mu$ is a hook shape.
While this procedure only sometimes produces a multiset of permutations whose number is $g_{\lambda \mu \nu}$, when it does, it somehow miraculously avoids the difficulty encountered in many positivity problems in algebraic combinatorics: \emph{a quantity that is known to be nonnegative is easily expressed as the difference in cardinality of two natural sets of combinatorial objects but finding an injection from the smaller of these sets to the larger is extremely difficult}.

This paper gives a way of modifying $\Gamma_\lambda \circ \Gamma_\mu$ in the case $\mu$ is a hook shape, using colored words and mixed insertion, to obtain a positive combinatorial formula for Kronecker coefficients for one hook shape and two arbitrary shapes.
We now outline this rule.


A \emph{colored word} is a word in the alphabet of barred letters  $\{\crc{1},\crc{2}, \cdots\}$ and unbarred letters $\{1,2, \cdots\}$.
Let  $w$ be a colored word. The \emph{total color} of $w$ is the number of barred letters in  $w$.  Define $w^\plain$ to be the ordinary word formed from $w$ by shuffling the barred letters to the left and then removing their bars.
We say that $w$ is \emph{Yamanouchi} of content $\lambda$ if $w^\plain$ is Yamanouchi of content $\lambda$.   For example, if $w = 1\, \crc{3}\, \crc{1}\, 1\, \crc{2}\, \crc{2}\, 2\, 1$, then  $w^\plain = 3\, 1 \, 2\, 2\, 1\, 1\, 2\, 1$,  and these are Yamanouchi of content  $(4,3,1)$.

Set $\mu(d) := (n-d,1^d)$.
We define $\text{CYW}_{\lambda, d}$ to be the set of colored Yamanouchi words of content $\lambda$ and total color $d$;
Figure \ref{f cw} depicts the case  $\lambda = (3,1,1)$,  $d=2$.
This replaces the multiset of permutations $\Gamma_{\lambda} \circ \big(\Gamma_{\mu(d)} \sqcup \Gamma_{\mu(d-1)}\big)$ in the experiment above.   This will be  fully
explained in \textsection\ref{ss comparison with Lascoux},  but for now we remark that if  $P(v) = Z_\mu^\stand$ has hook shape, then we can color $u \circ v$ in such a way that allows us to recover  $u$ and  $v$ from  $u \circ v$.


Mixed insertion is a generalization of Schensted insertion to colored words, developed by Haiman in \cite{Hmixed}. Its chief advantage for this work is that it is simultaneously compatible with any ordering of colored letters in which $1 < 2 < \cdots$ and $\crc{1} < \crc{2} < \cdots$ (see Proposition \ref{p conversion} for a precise statement).
Let $\text{CYT}_{\lambda, d}$ (resp. $\text{CYT}^\prec_{\lambda, d}$) denote the set of mixed insertion tableaux of the words in $\text{CYW}_{\lambda, d}$ using the \emph{natural order} $\crc{1}<1<\crc{2}<2\cdots$  (resp.  the \emph{small bar order} $\crc{1}\prec\crc{2}\prec\cdots\prec1\prec2\cdots$); see Figure \ref{f ct 311}.
For any set of tableaux $\text{ST}$, let $\text{ST}(\nu)$ denote the subset of $\text{ST}$ consisting of tableaux of shape $\nu$.

It is easy to show that $\text{CYT}^\prec_{\lambda, d}(\nu)$ has size  $\gcoef + g_{\lambda\, \mu(d-1)\, \nu}$  (Proposition \ref{p arm leg}).  This is in some sense not new. For example, the $\text{CYT}^\prec_{\lambda,d}$ are closely related to the $(k,l)$ tableaux and hook Schur functions of Berele-Regev \cite{BereleRegev} (see Remark \ref{r hook Schur functions}).  What is genuinely new here is the use of mixed insertion  for both the orders  $<$ and  $\prec$.
The miracle in this setup is that it is easy to identify a subset of $ \text{CYT}_{\lambda, d}(\nu)$ having cardinality $\gcoef$: it is the subset consisting of those tableaux with unbarred southwest corner (see Figure \ref{f CYT 321}).  We call this combinatorial formula for $\gcoef$ Hook Kronecker Rule I.

We define the \emph{color lowering operator} to be the operation that removes the bar on the southwest entry of a colored tableau (if it is barred). One of the main tasks in this paper is to understand the corresponding operator on colored words. This operator is more subtle and involves rotating a certain subword once to the right. Once this operator is understood, the proof of Hook Kronecker Rule I is not difficult; it also allows us to prove two somewhat more versatile versions of this rule (Hook Kronecker Rules II and III). We also show that Hook Kronecker Rule I easily generalizes to $\nu$ a skew shape (Hook Kronecker Rule IV).

This paper is organized as follows: Section \ref{s colored tableaux and insertion algorithms} gives the necessary background on colored tableaux and mixed insertion and also establishes (\textsection\ref{ss the operators bneg and plain}) some important facts about the operator ${}^\plain$ and a related operator ${}^\bneg$. In Section \ref{s kronecker coefficients for one hook shape}, we show that $|\text{CYT}_{\lambda, d}(\nu)| = \gcoef + g_{\lambda\, \mu(d-1)\, \nu}$ and officially state Hook Kronecker Rule I. In Section \ref{s color raising and lowering operators on words}, we define a color lowering operator on words and relate it to the color lowering operator; we then use this to complete the proof of our rule. Finally, Section \ref{s more versions and symmetries of the Hook Kronecker Rule} gives three more versions of our rule, discusses symmetries of these rules, and explains how they are related to Lascoux's Kronecker Rule.

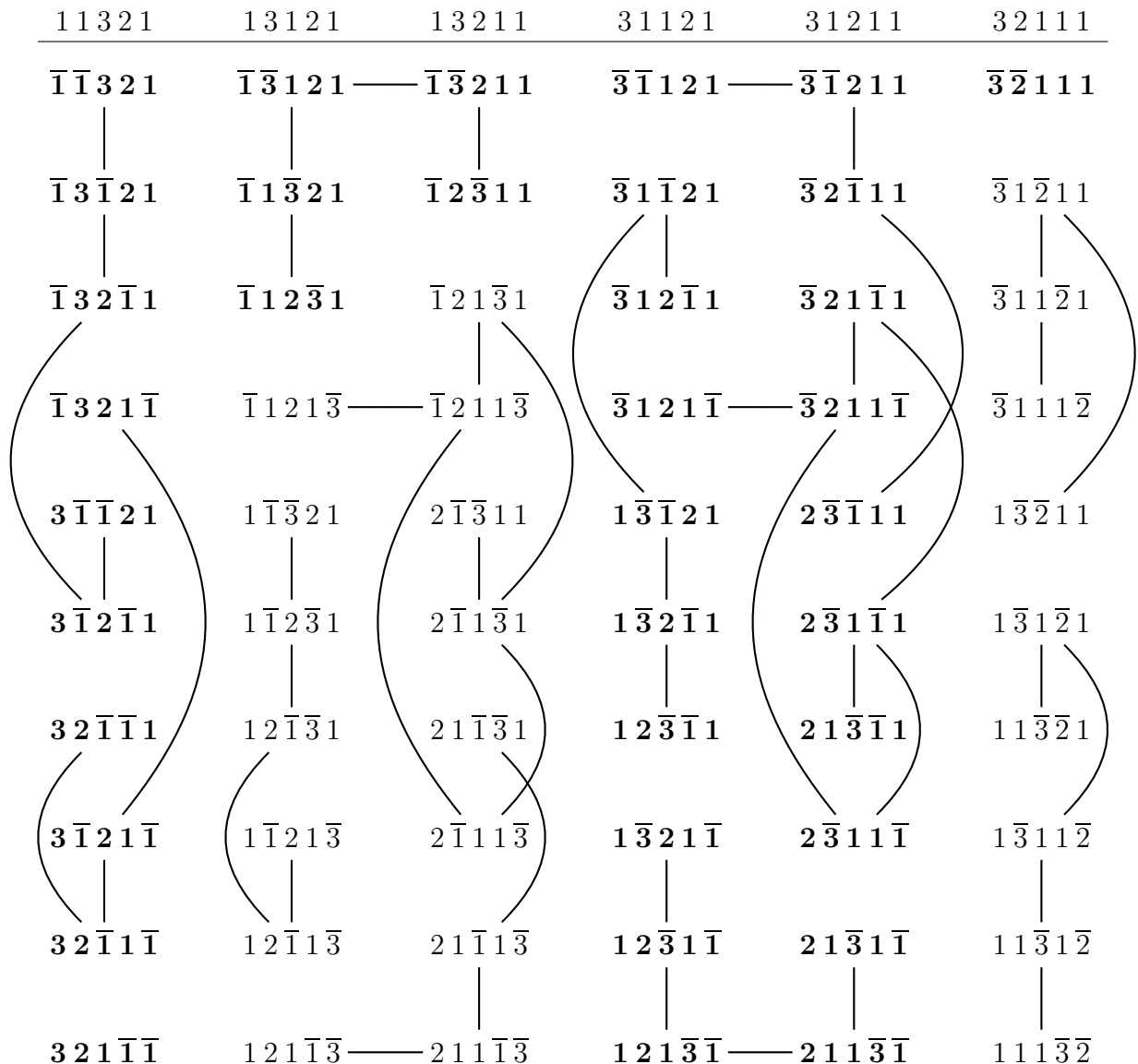
\begin{figure}
\centerfloat
\begin{tikzpicture}[xscale = 2.7,yscale = 1.55]
\tikzstyle{vertex}=[inner sep=0pt, outer sep=3pt]
\tikzstyle{aedge} = [draw, thin, ->,black]
\tikzstyle{bedge} = [draw, -, black]
\tikzstyle{edge} = [draw, thick, -,black]
\tikzstyle{LabelStyleH} = [text=black, anchor=south, near start]
\tikzstyle{LabelStyleV} = [text=black, anchor=east, near start]

\node[vertex] (hlineleft) at (-0.4,9.4){};
\node[vertex] (hlineright) at (5.4,9.4){};
\draw[bedge] (hlineleft) to (hlineright);
\node[vertex] (v1) at (5,9.59){$3\:2\:1\:1\:1$};
\node[vertex] (v2) at (4,9.59){$3\:1\:2\:1\:1$};
\node[vertex] (v3) at (3,9.59){$3\:1\:1\:2\:1$};
\node[vertex] (v4) at (2,9.59){$1\:3\:2\:1\:1$};
\node[vertex] (v5) at (1,9.59){$1\:3\:1\:2\:1$};
\node[vertex] (v6) at (0,9.59){$1\:1\:3\:2\:1$};
\node[vertex] (v1) at (0,6){$\mathbf{\crc{1}\:3\:2\:1\:\crc{1}}$};
\node[vertex] (v2) at (0,2){$\mathbf{3\:\crc{1}\:2\:1\:\crc{1}}$};
\node[vertex] (v3) at (0,1){$\mathbf{3\:2\:\crc{1}\:1\:\crc{1}}$};
\node[vertex] (v4) at (0,0){$\mathbf{3\:2\:1\:\crc{1}\:\crc{1}}$};
\node[vertex] (v5) at (1,6){$\crc{1}\:1\:2\:1\:\crc{3}$};
\node[vertex] (v6) at (1,2){$1\:\crc{1}\:2\:1\:\crc{3}$};
\node[vertex] (v7) at (1,1){$1\:2\:\crc{1}\:1\:\crc{3}$};
\node[vertex] (v8) at (1,0){$1\:2\:1\:\crc{1}\:\crc{3}$};
\node[vertex] (v9) at (2,6){$\crc{1}\:2\:1\:1\:\crc{3}$};
\node[vertex] (v10) at (2,2){$2\:\crc{1}\:1\:1\:\crc{3}$};
\node[vertex] (v11) at (2,1){$2\:1\:\crc{1}\:1\:\crc{3}$};
\node[vertex] (v12) at (2,0){$2\:1\:1\:\crc{1}\:\crc{3}$};
\node[vertex] (v13) at (3,6){$\mathbf{\crc{3}\:1\:2\:1\:\crc{1}}$};
\node[vertex] (v14) at (3,2){$\mathbf{1\:\crc{3}\:2\:1\:\crc{1}}$};
\node[vertex] (v15) at (3,1){$\mathbf{1\:2\:\crc{3}\:1\:\crc{1}}$};
\node[vertex] (v16) at (3,0){$\mathbf{1\:2\:1\:\crc{3}\:\crc{1}}$};
\node[vertex] (v17) at (4,6){$\mathbf{\crc{3}\:2\:1\:1\:\crc{1}}$};
\node[vertex] (v18) at (4,2){$\mathbf{2\:\crc{3}\:1\:1\:\crc{1}}$};
\node[vertex] (v19) at (4,1){$\mathbf{2\:1\:\crc{3}\:1\:\crc{1}}$};
\node[vertex] (v20) at (4,0){$\mathbf{2\:1\:1\:\crc{3}\:\crc{1}}$};
\node[vertex] (v21) at (5,6){$\crc{3}\:1\:1\:1\:\crc{2}$};
\node[vertex] (v22) at (5,2){$1\:\crc{3}\:1\:1\:\crc{2}$};
\node[vertex] (v23) at (5,1){$1\:1\:\crc{3}\:1\:\crc{2}$};
\node[vertex] (v24) at (5,0){$1\:1\:1\:\crc{3}\:\crc{2}$};
\node[vertex] (v25) at (0,9){$\mathbf{\crc{1}\:\crc{1}\:3\:2\:1}$};
\node[vertex] (v26) at (0,8){$\mathbf{\crc{1}\:3\:\crc{1}\:2\:1}$};
\node[vertex] (v27) at (0,7){$\mathbf{\crc{1}\:3\:2\:\crc{1}\:1}$};
\node[vertex] (v28) at (0,5){$\mathbf{3\:\crc{1}\:\crc{1}\:2\:1}$};
\node[vertex] (v29) at (0,4){$\mathbf{3\:\crc{1}\:2\:\crc{1}\:1}$};
\node[vertex] (v30) at (0,3){$\mathbf{3\:2\:\crc{1}\:\crc{1}\:1}$};
\node[vertex] (v31) at (1,9){$\mathbf{\crc{1}\:\crc{3}\:1\:2\:1}$};
\node[vertex] (v32) at (1,8){$\mathbf{\crc{1}\:1\:\crc{3}\:2\:1}$};
\node[vertex] (v33) at (1,7){$\mathbf{\crc{1}\:1\:2\:\crc{3}\:1}$};
\node[vertex] (v34) at (1,5){$1\:\crc{1}\:\crc{3}\:2\:1$};
\node[vertex] (v35) at (1,4){$1\:\crc{1}\:2\:\crc{3}\:1$};
\node[vertex] (v36) at (1,3){$1\:2\:\crc{1}\:\crc{3}\:1$};
\node[vertex] (v37) at (2,9){$\mathbf{\crc{1}\:\crc{3}\:2\:1\:1}$};
\node[vertex] (v38) at (2,8){$\mathbf{\crc{1}\:2\:\crc{3}\:1\:1}$};
\node[vertex] (v39) at (2,7){$\crc{1}\:2\:1\:\crc{3}\:1$};
\node[vertex] (v40) at (2,5){$2\:\crc{1}\:\crc{3}\:1\:1$};
\node[vertex] (v41) at (2,4){$2\:\crc{1}\:1\:\crc{3}\:1$};
\node[vertex] (v42) at (2,3){$2\:1\:\crc{1}\:\crc{3}\:1$};
\node[vertex] (v43) at (3,9){$\mathbf{\crc{3}\:\crc{1}\:1\:2\:1}$};
\node[vertex] (v44) at (3,8){$\mathbf{\crc{3}\:1\:\crc{1}\:2\:1}$};
\node[vertex] (v45) at (3,7){$\mathbf{\crc{3}\:1\:2\:\crc{1}\:1}$};
\node[vertex] (v46) at (3,5){$\mathbf{1\:\crc{3}\:\crc{1}\:2\:1}$};
\node[vertex] (v47) at (3,4){$\mathbf{1\:\crc{3}\:2\:\crc{1}\:1}$};
\node[vertex] (v48) at (3,3){$\mathbf{1\:2\:\crc{3}\:\crc{1}\:1}$};
\node[vertex] (v49) at (4,9){$\mathbf{\crc{3}\:\crc{1}\:2\:1\:1}$};
\node[vertex] (v50) at (4,8){$\mathbf{\crc{3}\:2\:\crc{1}\:1\:1}$};
\node[vertex] (v51) at (4,7){$\mathbf{\crc{3}\:2\:1\:\crc{1}\:1}$};
\node[vertex] (v52) at (4,5){$\mathbf{2\:\crc{3}\:\crc{1}\:1\:1}$};
\node[vertex] (v53) at (4,4){$\mathbf{2\:\crc{3}\:1\:\crc{1}\:1}$};
\node[vertex] (v54) at (4,3){$\mathbf{2\:1\:\crc{3}\:\crc{1}\:1}$};
\node[vertex] (v55) at (5,9){$\mathbf{\crc{3}\:\crc{2}\:1\:1\:1}$};
\node[vertex] (v56) at (5,8){$\crc{3}\:1\:\crc{2}\:1\:1$};
\node[vertex] (v57) at (5,7){$\crc{3}\:1\:1\:\crc{2}\:1$};
\node[vertex] (v58) at (5,5){$1\:\crc{3}\:\crc{2}\:1\:1$};
\node[vertex] (v59) at (5,4){$1\:\crc{3}\:1\:\crc{2}\:1$};
\node[vertex] (v60) at (5,3){$1\:1\:\crc{3}\:\crc{2}\:1$};
\draw[edge, bend left=25] (v1) to (v2);
\draw[edge] (v2) to (v3);
\draw[edge, bend left=30] (v3) to (v30);
\draw[edge] (v5) to (v9);
\draw[edge] (v6) to (v7);
\draw[edge, bend left=30] (v7) to (v36);
\draw[edge] (v8) to (v12);
\draw[edge, bend right=25] (v9) to (v10);
\draw[edge] (v9) to (v39);
\draw[edge, bend right=30] (v10) to (v41);
\draw[edge] (v11) to (v12);
\draw[edge, bend right=30] (v11) to (v42);
\draw[edge] (v13) to (v17);
\draw[edge] (v14) to (v15);
\draw[edge] (v15) to (v16);
\draw[edge] (v16) to (v20);
\draw[edge, bend right=25] (v17) to (v18);
\draw[edge] (v17) to (v51);
\draw[edge, bend right=30] (v18) to (v53);
\draw[edge] (v19) to (v20);
\draw[edge] (v21) to (v57);
\draw[edge, bend right=30] (v22) to (v59);
\draw[edge] (v22) to (v23);
\draw[edge] (v23) to (v24);
\draw[edge] (v25) to (v26);
\draw[edge] (v26) to (v27);
\draw[edge, bend right=30] (v27) to (v29);
\draw[edge] (v28) to (v29);
\draw[edge] (v31) to (v32);
\draw[edge] (v31) to (v37);
\draw[edge] (v32) to (v33);
\draw[edge] (v34) to (v35);
\draw[edge] (v35) to (v36);
\draw[edge] (v37) to (v38);
\draw[edge, bend left=30] (v39) to (v41);
\draw[edge] (v40) to (v41);
\draw[edge] (v43) to (v49);
\draw[edge, bend right=30] (v44) to (v46);
\draw[edge] (v44) to (v45);
\draw[edge] (v46) to (v47);
\draw[edge] (v47) to (v48);
\draw[edge] (v49) to (v50);
\draw[edge, bend left=35] (v50) to (v52);
\draw[edge, bend left=35] (v51) to (v53);
\draw[edge] (v53) to (v54);
\draw[edge, bend left=30] (v56) to (v58);
\draw[edge] (v56) to (v57);
\draw[edge] (v59) to (v60);
\end{tikzpicture}
\caption{The set $\text{CYW}_{(3,1,1), 2}$. Edges are Knuth transformations of the words obtained by applying ${}^\bneg$. Column labels correspond to applying ${}^\plain$ and the positions of the barred letters are constant along rows. The color raisable words are shown in bold.}
\label{f cw}
\end{figure}

\begin{figure}
\centerfloat
\begin{tikzpicture}[xscale = 2.8,yscale = 1.5]
\tikzstyle{vertex}=[inner sep=0pt, outer sep=3pt]
\tikzstyle{aedge} = [draw, thin, ->,black]
\tikzstyle{bedge} = [draw, -, black]
\tikzstyle{edge} = [draw, thick, -,black]
\tikzstyle{LabelStyleH} = [text=black, anchor=south, near start]
\tikzstyle{LabelStyleV} = [text=black, anchor=east, near start]

\node[vertex] (v1) at (5,9.8){$3\:2\:1\:1\:1$};
\node[vertex] (v2) at (4,9.8){$3\:1\:2\:1\:1$};
\node[vertex] (v3) at (3,9.8){$3\:1\:1\:2\:1$};
\node[vertex] (v4) at (2,9.8){$1\:3\:2\:1\:1$};
\node[vertex] (v5) at (1,9.8){$1\:3\:1\:2\:1$};
\node[vertex] (v6) at (0,9.8){$1\:1\:3\:2\:1$};
\node[vertex] (v1) at (0,6){${ \tiny \boldtableau{
\crc{1}&1\\\crc{1}\\2\\3\\}} $};
\node[vertex] (v2) at (0,2){$\bullet$};
\node[vertex] (v3) at (0,1){$\bullet$};
\node[vertex] (v4) at (0,0){${ \tiny \boldtableau{
\crc{1}\\\crc{1}\\1\\2\\3\\}} $};
\node[vertex] (v5) at (1,6){$\bullet$};
\node[vertex] (v6) at (1,2){$\bullet$};
\node[vertex] (v7) at (1,1){$\bullet$};
\node[vertex] (v8) at (1,0){$\bullet$};
\node[vertex] (v9) at (2,6){$\bullet$};
\node[vertex] (v10) at (2,2){$\bullet$};
\node[vertex] (v11) at (2,1){$\bullet$};
\node[vertex] (v12) at (2,0){$\bullet$};
\node[vertex] (v13) at (3,6){$\bullet$};
\node[vertex] (v14) at (3,2){${ \tiny \boldtableau{
\crc{1}&1\\1&\crc{3}\\2\\}} $};
\node[vertex] (v15) at (3,1){$\bullet$};
\node[vertex] (v16) at (3,0){$\bullet$};
\node[vertex] (v17) at (4,6){$\bullet$};
\node[vertex] (v18) at (4,2){$\bullet$};
\node[vertex] (v19) at (4,1){$\bullet$};
\node[vertex] (v20) at (4,0){$\bullet$};
\node[vertex] (v21) at (5,6){$\bullet$};
\node[vertex] (v22) at (5,2){$\bullet$};
\node[vertex] (v23) at (5,1){$\bullet$};
\node[vertex] (v24) at (5,0){$\bullet$};
\node[vertex] (v25) at (0,9){${ \tiny \boldtableau{
\crc{1}&1\\\crc{1}&2\\3\\}} $};
\node[vertex] (v26) at (0,8){$\bullet$};
\node[vertex] (v27) at (0,7){$\bullet$};
\node[vertex] (v28) at (0,5){$\bullet$};
\node[vertex] (v29) at (0,4){$\bullet$};
\node[vertex] (v30) at (0,3){$\bullet$};
\node[vertex] (v31) at (1,9){${ \tiny \boldtableau{
\crc{1}&1&1\\2&\crc{3}\\}} $};
\node[vertex] (v32) at (1,8){$\bullet$};
\node[vertex] (v33) at (1,7){$\bullet$};
\node[vertex] (v34) at (1,5){${ \tiny \tableau{
\crc{1}&1\\1&2\\\crc{3}\\}} $};
\node[vertex] (v35) at (1,4){$\bullet$};
\node[vertex] (v36) at (1,3){$\bullet$};
\node[vertex] (v37) at (2,9){$\bullet$};
\node[vertex] (v38) at (2,8){$\bullet$};
\node[vertex] (v39) at (2,7){${ \tiny \tableau{
\crc{1}&1&1\\2\\\crc{3}\\}} $};
\node[vertex] (v40) at (2,5){$\bullet$};
\node[vertex] (v41) at (2,4){$\bullet$};
\node[vertex] (v42) at (2,3){${ \tiny \tableau{
\crc{1}&1\\1\\2\\\crc{3}\\}} $};
\node[vertex] (v43) at (3,9){${ \tiny \boldtableau{
\crc{1}&1&1&\crc{3}\\2\\}} $};
\node[vertex] (v44) at (3,8){${ \tiny \boldtableau{
\crc{1}&1&\crc{3}\\1&2\\}} $};
\node[vertex] (v45) at (3,7){$\bullet$};
\node[vertex] (v46) at (3,5){$\bullet$};
\node[vertex] (v47) at (3,4){$\bullet$};
\node[vertex] (v48) at (3,3){$\bullet$};
\node[vertex] (v49) at (4,9){$\bullet$};
\node[vertex] (v50) at (4,8){$\bullet$};
\node[vertex] (v51) at (4,7){${ \tiny \boldtableau{
\crc{1}&1&\crc{3}\\1\\2\\}} $};
\node[vertex] (v52) at (4,5){$\bullet$};
\node[vertex] (v53) at (4,4){$\bullet$};
\node[vertex] (v54) at (4,4){$\bullet$};
\node[vertex] (v55) at (5,9){${ \tiny \boldtableau{
1&1&1&\crc{2}&\crc{3}\\}} $};
\node[vertex] (v56) at (5,8){${ \tiny \tableau{
1&1&1&\crc{3}\\\crc{2}\\}} $};
\node[vertex] (v57) at (5,7){$\bullet$};
\node[vertex] (v58) at (5,5){$\bullet$};
\node[vertex] (v59) at (5,4){${ \tiny \tableau{
1&1&1\\\crc{2}&\crc{3}\\}} $};
\node[vertex] (v60) at (5,3){$\bullet$};
\draw[edge, bend left=25] (v1) to (v2);
\draw[edge] (v2) to (v3);
\draw[edge, bend left=30] (v3) to (v30);
\draw[edge] (v5) to (v9);
\draw[edge] (v6) to (v7);
\draw[edge, bend left=30] (v7) to (v36);
\draw[edge] (v8) to (v12);
\draw[edge, bend right=25] (v9) to (v10);
\draw[edge] (v9) to (v39);
\draw[edge, bend right=30] (v10) to (v41);
\draw[edge] (v11) to (v12);
\draw[edge, bend right=25] (v11) to (v42);
\draw[edge] (v13) to (v17);
\draw[edge] (v14) to (v15);
\draw[edge] (v15) to (v16);
\draw[edge] (v16) to (v20);
\draw[edge, bend right=20] (v17) to (v18);
\draw[edge] (v17) to (v51);
\draw[edge, bend right=30] (v18) to (v53);
\draw[edge] (v19) to (v20);
\draw[edge] (v21) to (v57);
\draw[edge, bend right=30] (v22) to (v59);
\draw[edge] (v22) to (v23);
\draw[edge] (v23) to (v24);
\draw[edge] (v25) to (v26);
\draw[edge] (v26) to (v27);
\draw[edge, bend right=30] (v27) to (v29);
\draw[edge] (v28) to (v29);
\draw[edge] (v31) to (v32);
\draw[edge] (v31) to (v37);
\draw[edge] (v32) to (v33);
\draw[edge] (v34) to (v35);
\draw[edge] (v35) to (v36);
\draw[edge] (v37) to (v38);
\draw[edge, bend left=30] (v39) to (v41);
\draw[edge] (v40) to (v41);
\draw[edge] (v43) to (v49);
\draw[edge, bend right=25] (v44) to (v46);
\draw[edge] (v44) to (v45);
\draw[edge] (v46) to (v47);
\draw[edge] (v47) to (v48);
\draw[edge] (v49) to (v50);
\draw[edge, bend left=35] (v50) to (v52);
\draw[edge, bend left=35] (v51) to (v53);
\draw[edge] (v53) to (v54);
\draw[edge, bend left=25] (v56) to (v58);
\draw[edge] (v56) to (v57);
\draw[edge] (v59) to (v60);
\end{tikzpicture}
\caption{The mixed insertion tableaux of the words in the previous figure (which are constant on connected components). This set of tableaux is $\text{CYT}_{(3,1,1), 2}$ and the tableaux in bold are those with unbarred southwest corner ($\text{CYT}^-_{(3,1,1), 2}$).}
\label{f ct 311}
\end{figure}
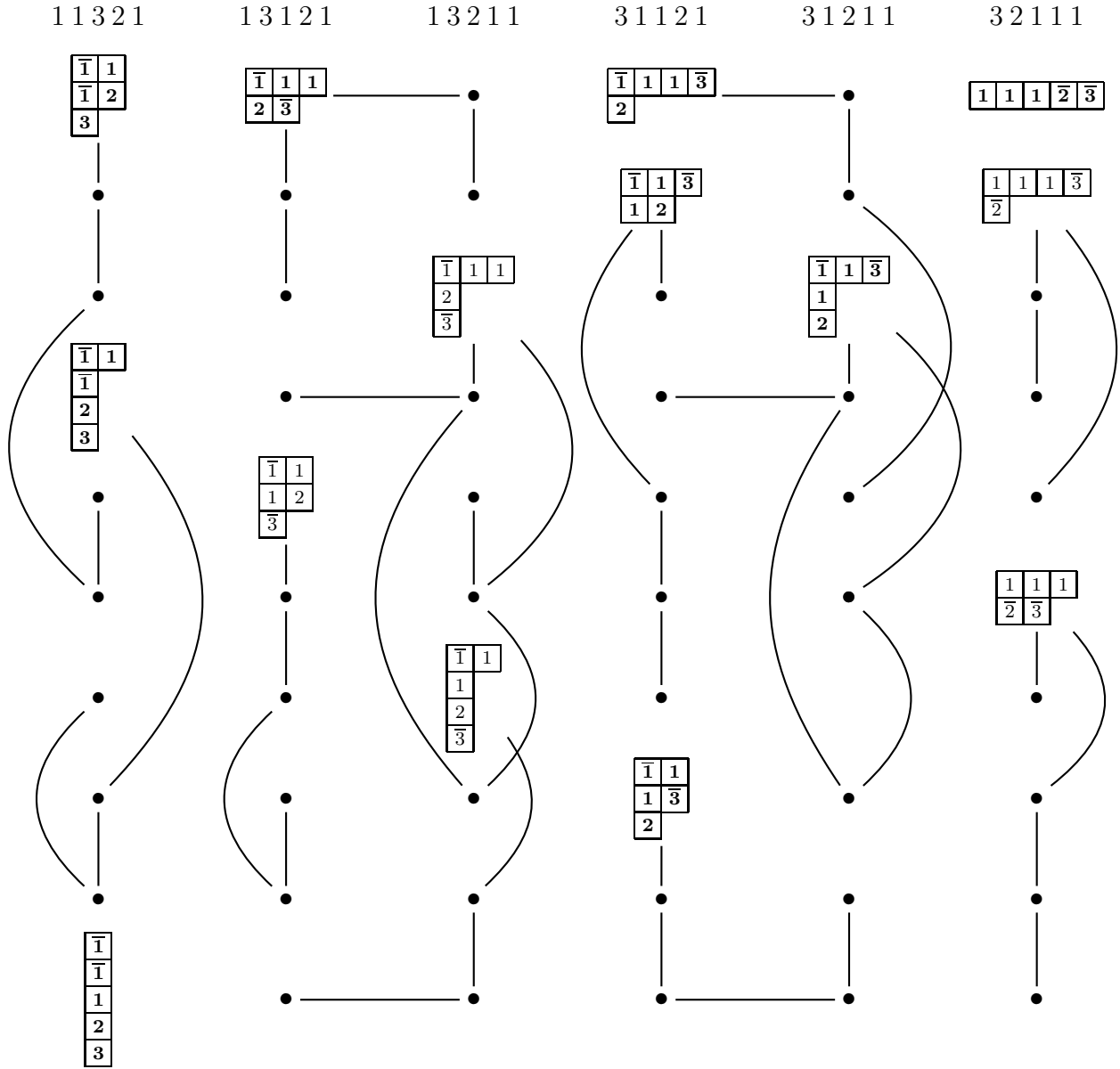
\section{Colored tableaux and Haiman's insertion algorithms} \label{s colored tableaux and insertion algorithms}
We begin this section with basic definitions of colored words and tableaux, and operators on these objects (\textsection\ref{ss words}--\ref{ss tableaux}).  Then, after fixing some notation for Schensted insertion (\textsection\ref{ss plactic equivalence}), we review Haiman's insertion algorithms and conversion \cite{Hmixed} (\textsection\ref{ss Haiman's insertion algorithms}--\ref{ss conversion}).  Finally, we establish some important facts about the operator ${}^\plain$ and a related operator ${}^\bneg$ (\textsection\ref{ss the operators bneg and plain}).
Almost all of the results in this section are restatements or easy consequences of results from \cite{Hmixed}.
Shimozono and White  \cite{SWspin} also give a nice exposition of this background, and we follow much of their notation.

\subsection{Words}
\label{ss words}
A \emph{word} is a sequence of (not necessarily distinct) letters from some totally ordered alphabet.
A \emph{subword} of a word  $w_1 w_2 \cdots w_n$ is a word of the form  $w_{k_1} w_{k_2} \cdots w_{k_l}$,  $k_1 < k_2< \cdots < k_l$. We say that $i$ is the \emph{place} of  $w_{i}$ and  $\mathbf{k} = k_1 k_2 \cdots k_l$ is the \emph{place word} of $w_{k_1} w_{k_2} \cdots w_{k_l}$; we also set $w_\mathbf{k} = w_{k_1} w_{k_2} \cdots w_{k_l}$.

The set $\{1,2,\ldots\}$ is the \emph{alphabet of unbarred letters} or \emph{ordinary letters} and the set $\{\crc{1},\crc{2},\ldots\}$ is the \emph{alphabet of barred letters}.
An \emph{ordinary word} is a word in the alphabet of ordinary letters.
A \emph{colored word} is a word in the alphabet $\mathcal{A} = \{\crc{1},\crc{2},\ldots \} \cup \{ 1,2,\ldots\}$  of barred and unbarred letters.
We typically write $w = w_1 w_2 \cdots w_n$ to denote a colored word of length  $n$, where each  $w_i$ denotes a colored letter which could be either barred or unbarred.
Also, we often use the symbol $x$ for an unbarred letter, while $\alpha, \beta$, and  $\eta$ are used for a colored letter which could be either barred or unbarred. For a colored letter $\alpha$, define $\alpha^* := \crc{x}$ if $\alpha = x$ and $\alpha^* := x$ if $\alpha = \crc{x}$.

Let $w = w_1 w_2 \cdots w_n$ be a colored word.  The \emph{total color} $\tc(w)$ of $w$ is the number of barred letters in $w$.
We write $\sub_{\crcempty}(w)$ for the subword of barred letters of $w$ and  $\sub_\varnothing(w)$ for the subword of unbarred letters.
We let $w^*$ denote the colored word $(w_1)^* (w_2)^* \cdots (w_n)^*$.
The ordinary word $w^\plain$ is formed from $w$ by shuffling the barred letters to the left and then removing their bars; precisely, $w^\plain = \sub_{\crcempty}(w)^*\sub_\varnothing(w)$. This operator will be studied further in \textsection\ref{ss the operators bneg and plain}.

The \emph{content} of an ordinary word $y$ is the sequence $(c_1, c_2, \ldots, c_m)$, where $c_i$ is the number of occurrences of $i$ in $y$ and $m$ is the largest letter of $y$.
The \emph{content} of a colored word $w$ is the content of $w^\plain$.
A \emph{colored permutation} is a colored word with content $(1^n)$.

The \emph{reverse} of a word $w = w_1 w_2 \cdots w_n$, denoted $w^{\rev}$, is the word $w_n w_{n-1} \cdots w_1$.
The \emph{upside-down word} of a colored permutation $v = v_1 v_2 \cdots v_n$, denoted $v^{\ud}$, is the colored permutation obtained by replacing each barred letter $\crc{x}$ by $\crc{n+1-x}$ and each unbarred letter $x$ by $n+1-x$.
The \emph{inverse} of a colored permutation $v$ is the colored permutation $v^\inv$ for which $(v^\inv)_i = j$ if  $v_j = i$ and $(v^\inv)_i = \crc{j}$ if $v_j = \crc{i}$.  If colored permutations are identified with signed permutation matrices (with barred letters corresponding to matrix entries equal to  $-1$), then the matrix for $v^\inv$ is just the transpose of the matrix for $v$.

\begin{example}
\label{ex colored word}
The colored word $w$ below has content $(4,4,1)$ and total color $5$.
The word $v$ below is a colored permutation, equal to the standardization $w^\stand$ of $w$ (defined below).
\[
\begin{array}{@{\hspace{-16mm}}rcr}
w & =\, & \crc{3}\ \crc{1}\ 2 \ 1 \ \crc{2}\ \crc{2}\ \crc{1}\ 2\ 1 \\[.6mm]
\sub_\varnothing(w) & =\, & 2\ 1\ 2\ 1 \\[.6mm]
\sub_{\crcempty}(w) & =\, & 3\ 1\ 2\ 2\ 1 \ \phantom{2\ 1\ 2\ 1}\\[.6mm]
w^\plain & =\, & 3\ 1\ 2\ 2\ 1\ 2\ 1\ 2\ 1 \\[.6mm]
w^* & =\, & 3\ 1\ \crc{2} \ \crc{1} \ 2\ 2\ 1\ \crc{2}\ \crc{1}\\[2.4mm]
v & =\, & \crc{9}\ \crc{1}\ 7\ 3\ \crc{5}\ \crc{6}\ \crc{2}\ 8\ 4 \\[.6mm]
v^{\rev} & =\, & 4\ 8\ \crc{2}\ \crc{6}\ \crc{5}\ 3\ 7\ \crc{1}\ \crc{9}\\[.6mm]
v^{\ud} & =\, & \crc{1}\ \crc{9}\ 3\ 7\ \crc{5}\ \crc{4}\ \crc{8}\ 2\ 6\\[.6mm]
v^\inv & =\, & \crc{2}\ \crc{7}\ 4\ 9\ \crc{5}\ \crc{6}\ 3\ 8\ \crc{1} \\[.6mm]
\end{array}
\]
\end{example}

We will work mostly with the following two orders on  $\mathcal{A}$:
\begin{align*}
\text{the \emph{natural order}} \ \ & \crc{1}<1<\crc{2}<2\cdots\\
\text{the \emph{small bar order}}\ \ & \crc{1}\prec\crc{2}\prec\crc{3}\prec\cdots\prec1\prec2\cdots.
\end{align*}
We reserve the symbol $\lessdot$ for an arbitrary total order on $\mathcal{A}$. Certain objects and operations in this paper are defined for any order $\lessdot$ and we indicate this by a superscript, i.e. $P^\lessdot_\mix$ will denote mixed insertion with respect to the order $\lessdot$; if no order is specified, then we mean the natural order $<$.

For any order  $\lessdot$ on  $\mathcal{A}$ and colored word $w$, the \emph{standardization} of  $w$ with respect to $\lessdot$, denoted $w^{\stand^\lessdot}$,
is the colored permutation obtained from $w$ by first relabeling, from left to right, the occurrences of the smallest letter in $w$ by  $1,\ldots,k$ (resp. $\crc{1},\ldots,\crc{k}$) if this letter is unbarred (resp. barred), then relabeling the occurrences of the next smallest letter of $w$ by $k+1,\ldots,k+k'$ (resp. $\crc{k+1},\ldots,\crc{k+k'}$) if this letter is unbarred (resp. barred), and so on.
For a colored word  $w$ and letter $\alpha$, $\sub_{\lessdoteq \alpha}(w)$ denotes the subword of $w$ consisting of the letters $\lessdoteq \, \alpha$.



\subsection{Tableaux} \label{ss tableaux}
A partition $\lambda$ of $n$ is a weakly decreasing sequence $ (\lambda_1, \ldots, \lambda_l)$ of nonnegative integers that sum to $n$.  We also write $\lambda \vdash n$  to mean that $\lambda$ is a partition of $n$.

The \emph{Ferrers diagram} or \emph{shape} of a partition $\lambda$ is the array of square cells, left-justified, with $\lambda_i$ cells in row $i$. Ferrers diagrams are drawn with the English (matrix-style) convention so that row (resp. column) labels start with 1 and increase from north to south (resp. west to east).
Write $\mu \subseteq \lambda$ if the shape of $\mu$ is contained in the shape of $\lambda$. If $\mu \subseteq \lambda$, then $\lambda/\mu$ denotes the \emph{skew shape} obtained by removing the cells of $\mu$ from the shape of $\lambda$.
The notation $\lambda \oplus \mu$ denotes the skew shape constructed by placing translates of shapes $\lambda$ and $\mu$ so that all cells of $\mu$ are above and to the right of all cells of $\lambda$. The conjugate partition $\lambda'$ of a partition $\lambda$ is the partition whose shape is the transpose of the shape of $\lambda$.

A \emph{tableau} $T$ of shape $\lambda/ \mu$ is the Ferrers diagram of $\lambda/\mu$ together with a letter occupying each of its cells.  The \emph{size} of $T$ is the number of cells of $T$, and $\sh(T)$ denotes the shape of $T$. The notation ${T}^{\transpose}$ denotes the transpose of $T$, so that $\sh({T}^{\transpose}) = \sh(T)'$.

Just as for shapes, $T \oplus U$ denotes the tableau constructed by placing translates of tableaux $T$ and $U$ so that all cells of $U$ are above and to the right of all cells of $T$. Given a cell $z$ and (skew) shape $\theta$, say that $z$ is \emph{addable to  $\theta$} if  $\theta \cap z = \varnothing$ and  $\theta \sqcup z$ is a skew shape. If $T$ is a tableau,  $\alpha$ a letter, and the cell  $z$ at position $(r,c)$ is addable to $\sh(T)$, then $T \sqcup {\tiny \tableau{\alpha}}_{(r,c)}$ denotes the result of adding the cell $z$ to $T$  and filling it with $\alpha$.



A \emph{semistandard tableau} or \emph{ordinary tableau} is a tableau in the alphabet of ordinary letters in which entries strictly increase from north to south in each column and weakly increase from west to east in each row. The \emph{content} of a semistandard tableau $T$ is the sequence $(c_1, c_2, \ldots, c_m)$, where $c_i$ is the number of occurrences of $i$ in $T$ and $m$ is the largest letter of $T$. A \emph{standard tableau} is a semistandard tableau of content $1^n$. The set of standard Young tableaux is denoted SYT and the subset of SYT of shape $\lambda$ is denoted SYT$(\lambda)$.
The \emph{row reading word} of a semistandard tableau $T$, denoted $\reading(T)$, is the word obtained by concatenating the rows of $T$ from bottom to top.

Let $Z_\lambda$ be the superstandard tableau of shape and content $\lambda$---the tableau whose $i$-th row is filled with $i$'s.
For an SYT  $Q$, $Q^\evac$ denotes the \emph{Sch\"utzenberger involution} or \emph{evacuation} of $Q$ (see, e.g., \cite[A1.2]{F}).

A \emph{semistandard colored tableau}, or \emph{colored tableau} for short, for the order $\lessdot$ is a tableau with entries in $\mathcal{A}$ such that unbarred letters strictly increase from north to south in each column and weakly increase from west to east in each row, and  barred letters weakly increase from north to south in each column and strictly increase from west to east in each row.
The set of colored tableaux for the order $\lessdot$ is denoted $\text{CT}^\lessdot$ (and $\text{CT} := \text{CT}^<$).
The \emph{content} of a colored tableau $T$ is the content of the ordinary tableau obtained by removing the bars on all the entries of $T$. A \emph{standard colored tableau} is a colored tableau of content $1^n$.
The standardization of a colored tableau $T$ for the order $\lessdot$, denoted $T^{\stand^\lessdot}$, is defined as for colored words, except that barred letters are relabeled from bottom to top and unbarred letters from left to right.

\begin{remark}
\label{r standardization}
Many of the algorithms used in this paper, like insertion and conversion, depend on knowing when one letter in a word or tableau is less than or greater than another.  For semistandard objects, when the two letters being compared are equal, the tie is resolved by checking which letter is larger than the other after standardizing.
\end{remark}

\begin{example} \label{ex colored tableau}
The tableau $T={\tiny\tableau{
\crc{1}&1&2&2&\crc{3}\\\crc{1}&\crc{2}\\\crc{2}\\}}$ is a colored tableau for the order $<$ of content $(3,4,1)$, shape $(5,2,1)$, and total color $5$.\vspace{1mm}
The standardization of $T$ is
$T^\stand = {\tiny\tableau{\crc{1}&3&6&7&\crc{8}\\\crc{2}&\crc{4}\\\crc{5}\\}}$.
The cell at position $(2,3)$ is an addable cell of $T$ and $T\, \sqcup\, {\tiny\tableau{3}}_{(2,3)}={\tiny\tableau{
\crc{1}&1&2&2&\crc{3}\\\crc{1}&\crc{2}&3\\\crc{2}\\}}$.
\end{example}


Just as for words, we write $\sub_{\lessdoteq \alpha}(T)$ for the subtableau of $T \in \text{CT}^{\lessdot}$ consisting of the letters $\lessdoteq \, \alpha$.
Let $T^*$ denote the colored tableau obtained from $T$ by applying  ${}^*$ to all the letters and then transposing the result. This is always a colored tableau, but not for the same order as $T$, in general.  We will avoid this issue by only applying ${}^*$ to standard colored tableaux or colored tableaux having only barred letters (also see Remark \ref{r bad rev and star}).
Let $T$ be a colored tableau for the order $\prec$.
Just as for words, define $\sub_{\crcempty}(T)$ to be the subtableau consisting of the barred letters of $T$ and $\sub_\varnothing(T)$ to be the skew subtableau consisting of the unbarred letters of $T$ (see Example \ref{ex neg plain}).

\subsection{Schensted insertion and the plactic monoid}
\label{ss plactic equivalence}

The insertion algorithms in this paper use the notion of \emph{inserting} a letter $\alpha$ into a row or column $R$ of a  $\text{CT}^\lessdot$.
By Remark \ref{r standardization}, it suffices to give this definition in the case that the letters of $R$ are distinct and distinct from $\alpha$.  In this case, inserting $ \alpha$ into  $R$ means that $\alpha$ replaces the least letter $\beta \gtrdot \alpha$ in $R$ or, if no such $\beta$ exists, adds a new cell containing $\alpha$ to the end of $R$. In the former case, we say that $\alpha$ \emph{bumps}  $\beta$.

For a colored word $w$, \emph{the insertion tableau and recording tableau of  $w$}, $P(w)$ and $Q(w)$, are defined using the usual Schensted insertion algorithm
using the order $<$ and breaking ties by Remark \ref{r standardization}.

For ordinary words $u$ and $v$, we write $u \ke v$ to indicate that $u$ and $v$ are Knuth equivalent or \emph{plactic equivalent}. Knuth equivalence classes, under the operation of concatenation, form a free associative monoid called the \emph{plactic monoid}. The Knuth equivalence class containing $u$ may be identified with the semistandard tableau $P(u)$, and any (skew) semistandard tableau $T$ may be identified with the Knuth equivalence class containing $\reading(T)$. Therefore, for ordinary words $u$ and $u'$, we allow such expressions as $u u' \ke P(u) \oplus P(u') \ke \reading(P(u)) \, \reading(P(u'))$ in the plactic monoid.

\subsection{Mixed insertion}\label{ss Haiman's insertion algorithms}
Here we review mixed insertion, as developed by Haiman in \cite{Hmixed}.
Mixed insertion was actually first defined by Berele and Regev in \cite{BereleRegev} and also studied by Remmel in \cite{Remmel2}. Haiman's treatment goes somewhat deeper and relates mixed insertion to an operation called conversion.  This relationship is of fundamental importance for this work and roughly means that mixed insertion is simultaneously compatible with any ordering of colored letters in which $1 < 2 < \cdots$ and $\crc{1} < \crc{2} < \cdots$.

\begin{definition}[Mixed insertion \cite{Hmixed}] Let $w = w_1 \dots w_n $ be a colored word and $T_0$ a colored tableau for the order $\lessdot$. Construct a sequence $T_0,T_1, \ldots, T_n = T$ of $\text{CT}^\lessdot$: for each $i = 1, \ldots, n$ form $T_i$ from $T_{i-1}$ by \emph{mixed inserting} $w_i$ as follows:

If $w_i$ is unbarred, insert $w_i$ (using the order $\lessdot$) into the first row of $T_{i-1}$; if it is barred, into the first column. As each subsequent element $\alpha$ of $T_{i-1}$ is bumped by an insertion, insert $\alpha$ into the row immediately below if it is unbarred, or into the column immediately to its right if it is barred. Continue until an insertion takes place at the end of a row or column, bumping no new element.

We say that $T = T_0 \xleftarrow{\mix} w$ is the \emph{mixed insertion of $w$ into $T_0$}. If $T_0 = \varnothing$, then $T$ is the \emph{mixed insertion tableau of $w$ for the order $\lessdot$} and is denoted $P^\lessdot_\mix(w)$; the \emph{mixed recording tableau of $w$ for the order $\lessdot$}, denoted $Q^\lessdot_\mix(w)$, is the SYT with the letter $i$ in the cell $\sh(T_i)/\sh(T_{i-1})$.

For the mixed insertion of a single letter $\alpha$, the \emph{insertion path} of  $T_0 \xleftarrow{\mix} \alpha$ is the sequence of cells containing the letters bumped during the mixed insertion, followed by the cell added at the end.
\end{definition}
See Example \ref{ex mixed insertion} for an example of mixed insertion.

\begin{definition}[Dual mixed insertion]
Following \cite[\textsection3.4]{SWspin}\ (see also \cite[Remark 8.5]{Hmixed}), define the \emph{dual mixed insertion} of the colored word $w$ into the colored tableau $T_0$, denoted $T_0 \xleftarrow{\dualmix} w$, to be the same as mixed insertion except with barred letters treated as if they are unbarred and vice versa. As for mixed insertion, this may be done with respect to any order $\lessdot$ on $\mathcal{A}$.
\end{definition}

We now assemble some basic facts about mixed and dual mixed insertion for later use.
\begin{proposition}[{\cite[Proposition 3.3]{Hmixed}}]
\label{p mix commutes with subwords}
Let $\alpha$ be a colored letter in $w$. Then
\[P^\lessdot_\mix(\sub_{\lessdoteq \alpha}(w)) = \sub_{\lessdoteq \alpha}(P^\lessdot_\mix(w)).\]
\end{proposition}

\begin{proposition}[{\cite[Remark 8.5]{Hmixed}}]
\label{p dual mixed}
For a colored word $w = w_1 \cdots w_n$
\[
P^\lessdot_\mix(w_2 w_3 \cdots w_n) \xleftarrow{\dualmix} w_1 = P^\lessdot_\mix(w).
\]
\end{proposition}

The next proposition follows easily from the definitions.
\begin{proposition}
\label{p standardization commutes}
Standardization commutes with many of the operations in this paper:
\begin{align*}
{P(w)}^{\stand} &= P({w}^{\stand}), \\
Q(w) &= Q({w}^{\stand}), \\
{P^\lessdot_\mix(w)}^{\stand^\lessdot} &= P_\mix({w}^{\stand^\lessdot}), \\
Q^\lessdot_\mix(w) &= Q_\mix({w}^{\stand^\lessdot}), \\
{w}^{\plain\;\stand} &= {w}^{\stand\;\plain},
\end{align*}
for any colored word $w$ and total order $\lessdot$ on  $\mathcal{A}$.
\end{proposition}

\begin{remark}\label{r bad rev and star}
The operators ${}^*$ and ${}^{\rev}$ do not commute with standardization. For example, $(\crc{1}\,\crc{1}\,1)^{*\;\stand} = 2\,3\,\crc{1}$, whereas $(\crc{1}\,\crc{1}\,1)^{\stand\;*} = 1\,2\,\crc{3}$; $(1\,1\,1)^{\rev\;\stand} = 1\,2\,3$, whereas $(1\,1\,1)^{\stand\;\rev} = 3\,2\,1$. We therefore only apply these operators to colored permutations.  Similarly, as commented in \textsection\ref{ss tableaux}, the operator ${}^*$ on colored tableaux will only be applied to standard colored tableaux and colored tableaux having only barred letters.
Left-right insertion also does not commute with the version of standardization used in this paper.
\end{remark}

\begin{remark}
Shimozono and White \cite{SWspin} use the convention that barred letters standardize from right to left in words and from left to right in colored tableaux, whereas we use the convention, in agreement with the introduction of \cite{Hmixed}, that barred letters standardize  from left to right in words and from bottom to top in  colored tableaux.
With either of   these conventions, standardization commutes with mixed insertion.
\end{remark}

The Schensted insertions of $u$, $u^\rev$, $u^\ud$, and $u^{\ud\;\rev}$, for an ordinary permutation $u$, are related by
\begin{align}
P(u^\rev) = P(u)^\transpose \ \ \ &\text{and}\ \ \ Q(u^\rev) = Q(u)^{\evac\;\transpose} \label{e rev}\\
P(u^\ud) = P(u)^{\evac\;\transpose} \ \ \ &\text{and}\ \ \ Q(u^\ud) = Q(u)^{\transpose} \label{e ud} \\
P(u^{\ud\;\rev}) = P(u)^{\evac} \ \ \ &\text{and}\ \ \ Q(u^{\ud\;\rev}) = Q(u)^{\evac} \label{e ud rev}
\end{align}
These well-known facts are nicely explained in  \cite[A1.2]{F}.
Some similar results hold for mixed insertion as well (though be warned that ${}^{\ud}$ is not compatible with mixed insertion in a simple way); these are proved in Propositions 3.4 and 8.3 and Corollary 8.4 of \cite{Hmixed}.
\begin{proposition}\label{p rev star mixed}
The operators ${}^*$ and ${}^{\rev}$ have the following effect on mixed insertion:
\begin{list}{\emph{(\roman{ctr})}}{\usecounter{ctr} \setlength{\itemsep}{2pt} \setlength{\topsep}{3pt}}
\item $P\nmix(w^*) = P\nmix(w)^*$,
\item $Q\nmix(w^*) = {Q\nmix(w)}^{\transpose}$,
\item $P\nmix({w}^{\rev}) = {P\nmix(w)}^{\transpose}$,
\item $Q\nmix({w}^{\rev}) = {Q\nmix(w)}^{\evac\;\transpose}$,
\end{list}
where $w$ is any colored permutation\footnote{This proposition holds more generally for any colored word with content consisting of 1's and 0's.}.
\end{proposition}

\subsection{Left-right insertion}
\label{ss left right insertion}
The algorithm which is dual to mixed insertion under inverses is left-right insertion.
Schensted insertion of an ordinary letter into a semistandard tableau is also called row insertion or \emph{right insertion}.
The transposed version of Schensted which bumps letters by columns is called column insertion or \emph{left insertion}.
\begin{definition}[Left-right insertion \cite{Hmixed}]
Let $w =w_1 \cdots w_n$ be a colored word. Construct a sequence $T_0, T_1, \ldots, T_n = T$ of semistandard tableaux: put  $T_0 = \varnothing$; for each $i = 1,\ldots,n$ form $T_i$ from  $T_{i-1}$ by left inserting $w_i^*$ if  $w_i$ is barred and right inserting $w_i$ if $w_i$ is unbarred.

We say that $T = T_0 \xleftarrow{\lr} w$ is the \emph{left-right insertion of $w$ into $T_0$}. If $T_0 = \varnothing$, then $T$ is the \emph{left-right insertion tableau of $w$}, denoted  $T=P_\lr(w)$. Let $Q$ be the recording tableau for the sequence $\varnothing \subset \sh(T_1) \subset \cdots \subset \sh(T_n) = \sh(T)$. The \emph{left-right recording tableau of $w$}, denoted by $Q_\lr(w)$, is obtained from $Q$ by barring those letters of $Q$ in cells added by left insertions; that is, $j$ is barred in $Q_\lr(w)$ if and only if $w_j$ is barred in $w$.  The \emph{insertion path} of $T_0 \xleftarrow{\lr} \alpha$ is defined just as for mixed insertion.
\end{definition}

\begin{proposition}
\label{p remove largest letter}
Let  $w = w_1 \cdots w_n$ be a colored permutation with largest letter $w_k$ (for the order  $<$), and set $w' = w_1 \cdots w_{k-1} w_{k+1} \cdots w_n$.  Let $Q'$ be the tableau obtained from $Q\nmix(w')$ by replacing  $n-1$ with  $n$,  $n-2$ with  $n-1$,  $\ldots$,   $k$ with  $k+1$ (this leaves the standardization of this recording tableau unchanged). Then
\[
\text{$P\nmix(w) = P\nmix(w') \sqcup  {\tiny \tableau{w_k}}_{(r,c)}$\ \ and\ \ $Q\nmix(w) = Q' \xleftarrow{\lr} (w^\inv)_n$,}
\]
where  $(r,c)$ is the position of the cell $\sh(Q\nmix(w))/\sh(Q')$.
\end{proposition}
Note that $(w^\inv)_n$ is $k$ if $w_k$ is unbarred and $\crc{k}$ if $w_k$ is barred, so the left-right insertion of $(w^\inv)_n$ is simply the row (resp. column) insertion of $k$ if $w_k$ is unbarred (resp.  barred).
\begin{proof}
Set  $(w^\inv)_L = (w^\inv)_1 (w^\inv)_2 \cdots (w^\inv)_{n-1}$. By Theorem 4.3 of \cite{Hmixed},
\begin{align}
Q\nmix(w) &=  P_\lr(w^\inv),  \label{e1 mix lr}\\
Q\nmix(w') &= P_\lr((w')^\inv), \label{e2 mix lr} \\
P\nmix(w)  &= Q_\lr(w^\inv), \label{e3 mix lr} \\
P\nmix(w')  &= Q_\lr((w')^\inv).  \label{e4 mix lr}
\end{align}
Since  $(w^\inv)_L$ is obtained from $(w')^\inv$ the same way $Q'$ is obtained from  $Q\nmix(w')$, \eqref{e2 mix lr} gives $Q' = P_\lr((w^\inv)_L)$.  Combining this with \eqref{e1 mix lr}, we obtain
\[Q\nmix(w)=  P_\lr(w^\inv) = P_\lr((w^\inv)_L) \xleftarrow{\lr} (w^\inv)_n = Q' \xleftarrow{\lr} (w^\inv)_n.\]
Similarly, \eqref{e3 mix lr}, \eqref{e4 mix lr}, and the relation between $(w^\inv)_L$ and $(w')^\inv$ just mentioned give
\[P\nmix(w)=  Q_\lr(w^\inv) = Q_\lr((w^\inv)_L) \sqcup {\tiny\tableau{w_k}}_{(r,c)} = Q_\lr((w')^\inv) \sqcup {\tiny\tableau{w_k}}_{(r,c)} = P\nmix(w') \sqcup {\tiny\tableau{w_k}}_{(r,c)}.\]
\end{proof}

\begin{remark}
Left-right insertion and Proposition \ref{p remove largest letter} are better understood using biwords.  In fact,
left-right insertion and mixed insertion can both be viewed as special cases of doubly mixed insertion of doubly colored biwords, as is explained in  \cite{SWspin}.
However, for this paper we have decided that this cleaner setup is not worth the notational overhead. 
\end{remark}

\subsection{Conversion}
\label{ss conversion}

For any total order $\lessdot$ on $\mathcal{A}$ and permutation $\sigma$ of $\mathcal{A}$, let $\lessdot^\sigma$ denote the total order on $\mathcal{A}$ in which $\sigma^{-1}(\alpha) \lessdot^\sigma \sigma^{-1}(\beta)$ if and only if $\alpha \lessdot \beta$.
For $k \in \ZZ_{\geq 1} \cup \{\infty\}$, let $<^k$ denote the order
\[\crc{1} \ <^k \ \crc{2} \ <^k \cdots <^k \ \crc{k} \ <^k \ 1 \ <^k \ 2 \ <^k \cdots <^k \ k \ <^k \ \crc{k+1} \ <^k \ k+1\ <^k \ \crc{k+2} \ <^k  \ k+2 \ \cdots.\]
Hence $<^1 = <$, $<^\infty = \prec$, and $(<^k)^\sigma = <^{k+1}$, where $\sigma$ is the cycle  $(1 \ 2\, \cdots \, k \ \crc{k+1})$.

\begin{definition}[Conversion \cite{Hmixed}]
\label{d conversion}
We first define conversion for a colored tableau $T$ with no repeated letter.  Let $\alpha$ be any letter in $T$ and $\beta$ be a letter not in $T$. The operation of \emph{converting $\alpha$ to $\beta$ in $T$} is as follows:

First, replace $\alpha$ with $\beta$. What results is not in general a colored tableau since $\beta$ may be too large or too small, relative to neighboring letters. As long as that is the case, repeatedly perform \emph{exchanges}: if $\beta$ is greater than its neighbor below or to the right, exchange $\beta$ with the lesser (or only) one of these neighbors; if instead $\beta$ is less than its neighbor above or to the left, exchange $\beta$ with the greater (or only) one of these.

The resulting tableau is denoted $\convert{\alpha \to \beta}{T}$. 

We have found it convenient to sometimes think of conversion in a slightly different way, a perspective which is also adopted in \cite[Algorithm 2.4]{BSSswitching}.
Instead of changing the letter in a cell, we keep the letters the same and change the order on the alphabet. Then replacing one letter with another can be accomplished by converting the current order $\lessdot$ to $\lessdot^\sigma$ for some cycle $\sigma$.  Hence, for a colored tableau  $T$ for the order  $<^k$ with no repeated letter, we define $\convert{<^k \to <^{k+1}}{T}$ to be the result of repeatedly performing exchanges between $\crc{k+1}$ and letters in  $\{1,\ldots,k\}$ until $T$ is semistandard for the order $<^{k+1}$. Similarly, the inverse of this procedure is denoted $\convert{<^{k+1} \to <^{k}}{U}$, which converts a colored tableau  $U$ for the order $<^{k+1}$ to a colored tableau for the order $<^k$. Finally, we define
\begin{align*}
\convert{<^k \to <^{l}}{T} := \convert{<^{l-1} \to <^l}{\convert{<^{k+1} \to <^{k +2}}{\convert{<^k \to <^{k+1}}{T}}\cdots \, } \text{ if  $k <l$}, \\
\convert{<^k \to <^{l}}{T} := \convert{<^{l+1} \to <^l}{\convert{<^{k-1} \to <^{k -2}}{\convert{<^k \to <^{k-1}}{T}}\cdots \, } \text{ if  $k > l$.}
\end{align*}
\end{definition}

\begin{remark}
Benkart, Sottile, and Stroomer \cite{BSSswitching} explain conversion as a special case of \emph{switching}, an operation which takes two tableaux with a common border and  moves them through each other using a sequence of exchanges.
They show that many different sequences of exchanges can be used to compute a given switch.  Hence, for instance, the particular sequence of exchanges prescribed above to convert from $<$ to $\prec$ is just a convenient choice---many other sequences would work as well.
\end{remark}

For a general semistandard colored tableau $T$ for the order $\lessdot$, conversion is defined from the above definition using Remark \ref{r standardization}. This means that $\convert{<^k \to <^{k+1}}{T}$ is accomplished by performing exchanges between the bottommost $\crc{k+1}$ and $\{1,\ldots,k\}$ until no more exchanges can be performed, then performing exchanges between the second bottommost $\crc{k+1}$ and $\{1,\ldots,k\}$ until no more exchanges can be performed, etc.
To be careful, there is something to check here, which is that the result of this procedure is a semistandard colored tableau for the order $<^{k+1}$. This is true because this conversion, in the language of \cite{BSSswitching}, is obtained by switching the subtableaux $T|_{\{\crc{k+1}\}}$ and $T|_{[k]}$ of $T$ (and leaving the remainder of $T$ fixed). Here, $T|_S$, $S \subseteq \mathcal{A}$, denotes the subtableau of $T$ consisting of the letters in $S$.

\begin{example}
The colored tableau on the left is converted from the small bar order to the natural order by converting each barred letter, from largest to smallest (keeping in mind Remark \ref{r standardization}). As indicated below, the conversions $<^{3} \to <^{2}$ and $<^{2} \to <$ each take two steps, where the occurrences of $\crc{3}$ and $\crc{2}$ are converted from topmost to bottommost.
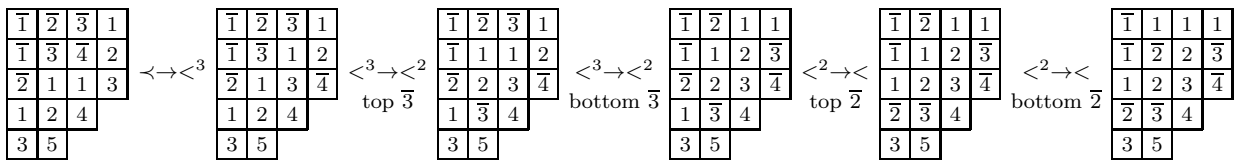
\[{\tiny\tableau{
\crc{1}&\crc{2}&\crc{3}&1\\\crc{1}&\crc{3}&\crc{4}&2\\\crc{2}&1&1&3\\1&2&4\\3&5\\}
\ \text{{\tiny${{\prec \to <^{3}} \atop {\text{}}}$}}\
\tableau{
\crc{1}&\crc{2}&\crc{3}&1\\\crc{1}&\crc{3}&1&2\\\crc{2}&1&3&\crc{4}\\1&2&4\\3&5\\}
\ \text{{\tiny${{<^{3} \to <^{2}} \atop {\text{top $\crc{3}$}}}$}} \
\tableau{
\crc{1}&\crc{2}&\crc{3}&1\\\crc{1}&1&1&2\\\crc{2}&2&3&\crc{4}\\1&\crc{3}&4\\3&5\\}
\ \text{{\tiny${{<^{3} \to <^{2}} \atop {\text{bottom $\crc{3}$}}}$}} \
\tableau{
\crc{1}&\crc{2}&1&1\\\crc{1}&1&2&\crc{3}\\\crc{2}&2&3&\crc{4}\\1&\crc{3}&4\\3&5\\}
\ \text{{\tiny${{<^{2} \to <}  \atop {\text{top $\crc{2}$}}}$}} \
\tableau{
\crc{1}&\crc{2}&1&1\\\crc{1}&1&2&\crc{3}\\1&2&3&\crc{4}\\\crc{2}&\crc{3}&4\\3&5\\}
\ \text{{\tiny${{<^{2} \to <} \atop {\text{bottom $\crc{2}$}}}$}} \
\tableau{
\crc{1}&1&1&1\\\crc{1}&\crc{2}&2&\crc{3}\\1&2&3&\crc{4}\\\crc{2}&\crc{3}&4\\3&5\\}}\]
\end{example}

Given the discussion above, it is not hard to verify that
\begin{proposition}
 \label{p standardize conversion}
Conversion commutes with standardization in the following sense: if  $T \in \text{CT}^{\,<^k}$ and the bottommost (resp.   topmost) $\crc{k+1}$ in  $T$ is relabeled by $\crc{l}$ (resp.  $\crc{m}$) in $T^{\stand_k}$, then
\[
 \convert{<^k \to < ^{k+1}}{T}^{\stand_{k+1}} = \convert{<^{l-1} \to < ^{m}}{T^{\stand_{k}} }.
\]
(We have abbreviated $\stand^{<^k}$ by $\stand_k$.)
A similar statement holds for any conversion  $<^k \to <^{k'}$.
\end{proposition}

\begin{proposition} \label{p conversion}
Converting between the small bar order and natural order commutes with mixed insertion in the following sense:
\begin{align}
\label{e1 conversion}
P\nmix(w) &= \convert{\prec \to <}{P^\prec_\mix(w)} \\
\label{e2 conversion}
P^\prec_\mix(w) &= \convert{< \to \prec}{P\nmix(w)}  \\
\label{e3 conversion}
Q\nmix(w) &= Q^\prec_\mix(w).
\end{align}
\end{proposition}
\begin{proof}
By Propositions \ref{p standardization commutes} and \ref{p standardize conversion}, we can assume that $w$ is a colored permutation. Corollary 3.16 of \cite{Hmixed}, adjusted to the notation at the end of Definition \ref{d conversion}, states that $P^{<^{k+1}}_\mix(w) =$ $ \convert{<^k \to <^{k+1}}{P^{<^{k}}_\mix(w)}$.  Repeated application then yields \eqref{e1 conversion}. The proof of \eqref{e2 conversion} is similar. Finally, applying \eqref{e1 conversion} to every initial subword $w_1 \cdots w_k$ of $w = w_1 w_2 \cdots w_n$ yields \eqref{e3 conversion}.
\end{proof}

\begin{example} \label{ex mixed insertion}
Let $w = \crc{3}\,\crc{1}\,2 \,1 \,\crc{2}\,\crc{2}\,\crc{1}\,2\,1$. The sequence of tableaux produced in computing $P\nmix(w)$ is shown on the next line, and below that the sequence for $P^\prec_\mix(w)$.
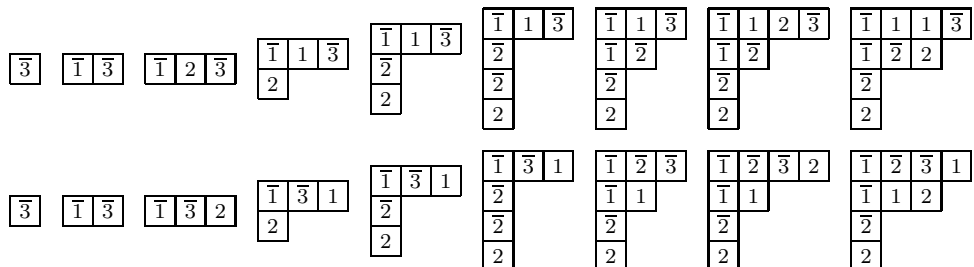
\[{\tiny \tableau{
\crc{3}\\}\quad\tableau{
\crc{1}&\crc{3}\\}\quad\tableau{
\crc{1}&2&\crc{3}\\}\quad\tableau{
\crc{1}&1&\crc{3}\\2\\}\quad\tableau{
\crc{1}&1&\crc{3}\\\crc{2}\\2\\}\quad\tableau{
\crc{1}&1&\crc{3}\\\crc{2}\\\crc{2}\\2\\}\quad\tableau{
\crc{1}&1&\crc{3}\\\crc{1}&\crc{2}\\\crc{2}\\2\\}\quad\tableau{
\crc{1}&1&2&\crc{3}\\\crc{1}&\crc{2}\\\crc{2}\\2\\}\quad\tableau{
\crc{1}&1&1&\crc{3}\\\crc{1}&\crc{2}&2\\\crc{2}\\2\\}\quad}\]
\[{\tiny \tableau{
\crc{3}\\}\quad\tableau{
\crc{1}&\crc{3}\\}\quad\tableau{
\crc{1}&\crc{3}&2\\}\quad\tableau{
\crc{1}&\crc{3}&1\\2\\}\quad\tableau{
\crc{1}&\crc{3}&1\\\crc{2}\\2\\}\quad\tableau{
\crc{1}&\crc{3}&1\\\crc{2}\\\crc{2}\\2\\}\quad\tableau{
\crc{1}&\crc{2}&\crc{3}\\\crc{1}&1\\\crc{2}\\2\\}\quad\tableau{
\crc{1}&\crc{2}&\crc{3}&2\\\crc{1}&1\\\crc{2}\\2\\}\quad\tableau{
\crc{1}&\crc{2}&\crc{3}&1\\\crc{1}&1&2\\\crc{2}\\2\\}\quad}\]
By Proposition \ref{p conversion}, each tableau $T$ on the top line is related to the tableau $U$ below it by $U = \convert{< \to \prec}{T}$. The mixed recording tableaux for the orders $<$ and $\prec$ encode the sequence of shapes above:
\[Q\nmix(w) = Q^\prec_\mix(w) = {\tiny \tableau{
1&2&3&8\\4&7&9\\5\\6\\}}. \]
\end{example}

\subsection{The operators ${}^\bneg$ and  ${}^\plain$} \label{ss the operators bneg and plain}
Here we study two operators,  ${}^\bneg$ and  ${}^\plain$, which take colored words to ordinary words.
We will see that the insertion tableaux of  $w^\bneg$ and  $w^\plain$ can both be computed from the mixed insertion tableau of $w$.

Let $w$ be a colored word. The ordinary word $w^{\bneg}$ is formed from  $w$ by replacing each barred letter $\crc{x}$ with the unbarred
letter\footnote{We must add $-1 > -2 > \cdots$ to our alphabet of ordinary letters.}
$-x$. Unfortunately, in order to make ${}^\bneg$ commute with standardization, we must adopt the convention that negative numbers standardize from right to left. To avoid this confusion, we will standardize before applying ${}^\bneg$.

The operator ${}^\bneg$ was defined and studied in \cite{SWspin}.
The next result is \cite[Proposition 14]{SWspin} (which is an easy consequence of Haimain's Theorem 3.12 \cite{Hmixed}).
\begin{proposition} \label{p neg and mixed insertion}
Let  $\crc{x}_1 < \crc{x}_2 < \cdots < \crc{x}_k$ be the barred letters of a colored permutation $v$. Then
\begin{list}{\emph{(\roman{ctr})}}{\usecounter{ctr} \setlength{\itemsep}{2pt} \setlength{\topsep}{3pt}}
\item $\convert{\crc{x}_k \to -x_k}{\convert{\crc{x}_2 \to -x_2}{\convert{\crc{x}_1 \to -x_1}{P\nmix(v)}}\cdots\,} = P(v^{\bneg})$,
\item $Q\nmix(v) = Q(v^{\bneg})$.
\end{list}
\end{proposition}

Recall that the ordinary word $w^\plain$ is formed from $w$ by shuffling the barred letters to the left and then removing their bars.
Given a colored tableau $T$ for the order $<^k$, let $T' = \convert{<^k \to \prec}{T}$. We define  $T^\plain$ to be the ordinary straight-shape tableau $P$ such that $P \ke \sub_{\crcempty}(T')^* \oplus \sub_\varnothing(T')$.

For a colored word $w$, let  $w^\barrev$ denote the colored word obtained from $w$ by reversing its subword of barred letters (keeping the unbarred letters  fixed).  For example, $(1 \, \crc{4}\, \crc{3}\, 2\, \crc{6}\, 5)^\barrev = 1\, \crc{6}\, \crc{3}\, 2\, \crc{4}\, 5$.
\begin{proposition}
\label{p plain basics}
For any colored word $w$ and colored permutation $v$,
\begin{list}{\emph{(\roman{ctr})}}{\usecounter{ctr} \setlength{\itemsep}{2pt} \setlength{\topsep}{3pt}}
\item $P\nmix(w)^\plain = P(w^\plain)$,
\item $Q\nmix(v^{\barrev\; \inv}) = P_{\lr}({v}^{\barrev}) = P(v^\plain)$,
\item $Q\nmix(v^{\inv \; \barrev}) = P_{\lr}({v}^{\barud}) = P(v^{\barud \; \barrev\; \plain})$,
\item The tableau $P := P(v^{\barud \; \barrev\; \plain})$ can be computed from $U := P\nmix(v)$ as follows: let $U' = \convert{< \to \prec}{U}$; then  $P$ is the ordinary straight-shape tableau $P$ such that $P \ke \sub_{\crcempty}(U')^{* \; \evac} \oplus \sub_\varnothing(U')$.
\end{list}
\end{proposition}
\begin{proof}
By Propositions \ref{p standardize conversion} and \ref{p standardization commutes}, ${T}^{\plain\;\stand} = {T}^{\stand\;\plain}$ for any $CT$  $T$. Together with Proposition \ref{p standardization commutes}, this implies that we can assume $w$ is a colored permutation.
Let $T' = \convert{< \to \prec}{P\nmix(w)}$.  By Proposition \ref{p conversion},  $T' = P^\prec_\mix(w)$.
Then by Proposition \ref{p mix commutes with subwords}  with the order $\prec$, $\sub_\crcempty(T') = P^\prec_\mix(\sub_\crcempty(w))$.
Since  $\sub_\crcempty(w)$ consists of only barred letters, this implies
\be \label{e mix plain1}
\sub_\crcempty(T')^* = P(\sub_\crcempty(w)^*).
\ee
Let $\prec^\prime$ denote the order $1\prec^\prime2\prec^\prime\cdots\crc{1}\prec^\prime\crc{2}\cdots$. Then by Proposition \ref{p mix commutes with subwords} with this order,  \be \label{e mix plain2}
\sub_\varnothing(P^{\prec^\prime}_\mix(w)) =  P^{\prec^\prime}_\mix(\sub_\varnothing(w)) = P(\sub_\varnothing(w)).
\ee
Since the conversion  $\convert{\prec \to \prec^\prime}{T'}$, ignoring barred letters, amounts to performing jeu de taquin slides to compute the straight-shape tableau that is plactic equivalent to $\sub_\varnothing(T')$, there holds $\sub_\varnothing(T') \ke \sub_\varnothing(P^{\prec^\prime}_\mix(w))$. Combining this with \eqref{e mix plain1} and \eqref{e mix plain2} gives
\[
P\nmix(w)^\plain \ke \sub_{\crcempty}(T')^* \oplus \sub_\varnothing(T') \ke P(\sub_\crcempty(w)^*) \oplus P(\sub_\varnothing(w)) \ke P(w^\plain),
\]
which proves (i).

Statement (ii) is an application of \cite[Theorem 4.3]{Hmixed} followed by \cite[Remark 4.4]{Hmixed}.  As $v^{\inv \; \barrev \; \inv} = v^{\barud}$, (iii) is just another way of writing (ii).  The proof of (iv) is the same as that of (i), using the additional fact that  $P(u^{\ud\; \rev}) = P(u)^\evac$ for any ordinary word  $u$.
\end{proof}

\begin{example} \label{ex neg plain}
Continuing Examples \ref{ex colored word} and \ref{ex mixed insertion}, recall $w = \crc{3}\,\crc{1}\,2 \,1 \,\crc{2}\,\crc{2}\,\crc{1}\,2\,1$ and  $v := w^\stand$.
To illustrate Proposition \ref{p neg and mixed insertion} (i), we compute
\[
{\small
\begin{array}{rlllllllllll}
v &=& \mboxtwo{\crc{9}} & \mboxtwo{\crc{1}}&\mboxtwo{7}&\mboxtwo{3}&\mboxtwo{\crc{5}}&\mboxtwo{\crc{6}}&\mboxtwo{\crc{2}}&\mboxtwo{8}& \mboxtwo{4}\\
v^\bneg &=& \mboxtwo{-9}&\mboxtwo{-1}&\mboxtwo{7}&\mboxtwo{3}&\mboxtwo{-5}&\mboxtwo{-6}&\mboxtwo{-2}&\mboxtwo{8}&\mboxtwo{4}
\end{array}
}
\]
\[ {{\tiny\tableau{\crc{1}&3&4&\crc{9}\\\crc{2}&\crc{5}&8\\\crc{6}\\7}} \atop {\small \text{$P\nmix(v)$}}}
\qquad {{\tiny\tableau{\ngcell 5&\ngcell 2&4&\crc{9}\\\ngcell 1&3\\\crc{6}\\7}} \atop {\small \text{$\convert{\crc{5} \to -5}{\convert{\crc{2} \to -2}{\convert{\crc{1} \to -1}{P\nmix(v)}}}$}}}
\qquad {{\tiny\tableau{\ngcell 9&\ngcell 6&\ngcell 2&4\\\ngcell 5&3&8\\\ngcell 1\\7}} \atop {\small \text{$P(v)$}}}
\]

To illustrate Proposition \ref{p plain basics} (i), we have
$w^\plain = 3\,1\,2\,2\,1\,2\,1\,2\,1$
and $P\nmix(w)^\plain = P(w^\plain)$ is computed from $ P^\prec_\mix(w)$ as follows
\[
\sub_\crcempty(P^\prec_\mix(w))^* \oplus \sub_\varnothing(P^\prec_\mix(w)) =  {\tiny \tableau{1&1&2\\2\\3\\}} \oplus  {\tiny \tableau{
&&&1\\&1&2\\\\2\\}} \ke {\tiny \tableau{1&1&1&1&2\\2&2&2\\3\\}} = P(w^\plain).
\]
\end{example}

The next result will be useful for better understanding Hook Kronecker Rule III, which expresses  $\gcoef$ as the cardinality of a set of colored words.
The operators $w \mapsto w^\plain$ and  $w \mapsto w^{\bneg\;\stand}$ both lose information and are related the same way ${}^\rev$ and ${}^\ud$ are related, except with a ``twist'' by  ${}^\barrev$.  The following proposition makes this precise and gives some related results.  Its proof is straightforward from the definitions.
\begin{proposition}\label{p plain and neg}
Let $w$ be a colored permutation. Then
\begin{list}{\emph{(\roman{ctr})}}{\usecounter{ctr} \setlength{\itemsep}{2pt} \setlength{\topsep}{3pt}}
\item $ w^{\plain\;\inv} = w^{\barrev\;\inv\;\bneg\;\stand} $
\item $ w^{\rev\;\barrev\;\nobarrev\;\plain}= w^\plain$
\item $ w^{\ud\;\barud\;\nobarud\;\bneg\;\stand}= w^{\bneg\;\stand}$
\item $ w^{*\;\plain}= w^{\rev\;\plain\;\rev}$
\item $ w^{*\;\bneg\;\stand}= w^{\bneg\;\stand\;\ud}.$
\end{list}
\end{proposition}


\section{Kronecker coefficients for one hook shape} \label{s kronecker coefficients for one hook shape}
Here we introduce the fundamental combinatorial objects of this work, colored Yamanouchi tableaux (CYT) and color raisable Yamanouchi tableaux ($\text{CYT}^-$). We then explain their relationship to Kronecker coefficients.

\subsection{Colored Yamanouchi tableaux}
An ordinary word $y=y_1 \cdots y_n$ is \emph{Yamanouchi} if every terminal subword $y_k y_{k+1} \cdots y_n$ has partition content. This is equivalent to $P(y) = Z_\lambda$, where $\lambda$ is the content of $y$ and $Z_\lambda$ is the superstandard tableau of shape and content $\lambda$.

We say that a colored word $w$ is \emph{Yamanouchi} if any of the following equivalent conditions is satisfied
\begin{list}{(\arabic{ctr})}{\usecounter{ctr} \setlength{\itemsep}{2pt} \setlength{\topsep}{3pt}}
\item $w^\plain$ is Yamanouchi,
\item $P(w^\plain)$ is superstandard,
\item $P_\mix(w)^\plain$ is superstandard.
\end{list}
Conditions (2) and (3) are equivalent by Proposition \ref{p plain basics} (i). We say that a colored tableau $T$ is \emph{Yamanouchi} if $T^\plain$ is superstandard, or equivalently,  if $T$ is the mixed insertion tableau of some Yamanouchi word.
The $w$ of Example \ref{ex neg plain} is not Yamanouchi because $w^\plain$ ends in $2\,2\,1\,2\,1\,2\,1$, which has content $(3,4)$.
An example of a colored Yamanouchi word is $\crc{3}\ \crc{1}\ 2 \ 1\ \crc{2}\ \crc{1}\ 2\ 1$. See Figure \ref{f CYT 321} for examples of colored Yamanouchi tableaux.

Define the following subsets of colored Yamanouchi tableaux (CYT):
\begin{align*}
\text{CYT}_{\lambda} &:= \{T \in \text{CT}: T^\plain = Z_\lambda\}\ \ (\text{the set of colored Yamanouchi tableaux of content $\lambda$}),\\
\text{CYT}_{\lambda, d} &:= \{T \in \text{CT}: T^\plain = Z_\lambda,\ \tc(T) = d\},\\
\text{CYT}_{\lambda, d}(\nu) &:= \{T \in \text{CT}: T^\plain = Z_\lambda,\ \tc(T) = d,\ \sh(T) = \nu\}.
\end{align*}
In the introduction, $\text{CYT}_{\lambda, d}$  was defined to be the set of mixed insertion tableaux of the colored Yamanouchi words of content $\lambda$ and total color $d$.   This is equivalent to  the present definition by Proposition \ref{p plain basics} (i).

\subsection{Counting colored Yamanouchi tableaux}
\label{ss counting colored Yamanouchi tableaux}
Recall that $\mu(d)$ denotes the hook shape $(n-d, 1^d)$ for $d \in \{0, 1, \ldots ,n-1\}$.
For a (skew) shape $\theta$, let $s_\theta = s_\theta(\mathbf{x})$ denote the Schur function corresponding to $\theta$ in the infinite set of variables $\mathbf{x} = x_1, x_2, \ldots$.
Let $c_{\lambda \, \mu}^\nu = \langle s_\lambda s_\mu, s_\nu \rangle$ be the Littlewood-Richardson coefficient. It is also convenient to set $c^{\nu/\mu}_\lambda = c_{\lambda \, \mu}^\nu$ (defined to be 0 if $\mu \not \subseteq \nu$).
Let $*$ denote the internal product of symmetric functions, which may be defined by $s_\lambda * s_\mu = \sum_{\nu} g_{\lambda\,\mu\,\nu} s_\nu$.

The following proposition relates colored Yamanouchi tableaux to Kronecker coefficients and is in some sense well known (see Remark \ref{r hook Schur functions}, below).
\begin{proposition} \label{p arm leg}
The following nonnegative integers are equal
\begin{list}{\emph{(\Alph{ctr})}}{\usecounter{ctr} \setlength{\itemsep}{2pt} \setlength{\topsep}{3pt}}
\item $\gcoef + g_{\lambda \, \mu(d-1) \, \nu}$,
\item $\langle s_\lambda * (s_{(1^d)} s_{(n-d)}), s_\nu \rangle$,
\item $\sum_{\alpha \vdash d,\ \beta \vdash n-d}\, c_{\alpha \, \beta}^\lambda \, c_{\alpha' \, \beta}^\nu$,
\item $|\text{CYT}_{\lambda, d}(\nu)|$,
\end{list}
for any $\lambda, \nu \vdash n$ and $d \in \{0,1,\ldots,n\}$ (interpreting the undefined expressions $g_{\lambda \, \mu(n) \, \nu}$ and  $g_{\lambda \, \mu(-1) \, \nu}$ to be 0).
\end{proposition}
\begin{proof}
The quantities (A) and (B) are the same since  $s_{(1^d)} s_{(n-d)} = s_{\mu(d)} + s_{\mu(d-1)}$.

The following general result of Littlewood \cite{Littlewood} relates the internal and ordinary products of the symmetric group:
\[
s_\lambda * (s_\theta s_\kappa) = \displaystyle\sum_{\alpha \vdash d,\ \beta \vdash n-d} c_{\alpha \, \beta}^\lambda (s_\alpha * s_\theta) (s_\beta * s_\kappa)
\]
for any partitions $\theta, \kappa$.
Setting $\theta = (1^d),$ $\kappa = (n-d)$, we obtain
\[
s_\lambda * (s_{(1^d)} s_{(n-d)}) = \displaystyle\sum_{\alpha \vdash d,\ \beta \vdash n-d} c_{\alpha \, \beta}^\lambda \, s_{\alpha'} s_\beta.
\]
Taking the inner product with $s_\nu$ on both sides then shows that (B) and (C) are equal.

Finally, we consider (D). After converting the tableaux $\text{CYT}_{\lambda, d}(\nu)$ to the order $\prec$ and unraveling the definition of $T^\plain$, we see that this set of tableaux is in bijection with the union of the Littlewood-Richardson tableaux of content $\lambda$ and shape $\alpha \oplus (\nu/\alpha')$, over all $\alpha \vdash d$ such that $\alpha' \subseteq \nu$. Hence
\[
|\text{CYT}_{\lambda, d}(\nu)| = \displaystyle\sum_{\alpha \vdash d} c_{\lambda}^{\alpha \oplus (\nu/\alpha')}.
\]
Multiplying this quantity by $s_\lambda$ and summing over $\lambda$ yields
\begin{align*}
\displaystyle\sum_{\alpha \vdash d,\ \lambda \vdash n} c_{\lambda}^{\alpha \oplus (\nu/\alpha')} s_\lambda &=
\displaystyle\sum_{\alpha \vdash d} s_{\alpha \oplus (\nu/\alpha')} \label{e arm leg comb2} \\
&=
\displaystyle\sum_{\alpha \vdash d} s_{\alpha} s_{\nu/\alpha'} =
\displaystyle\sum_{\alpha \vdash d,\ \beta \vdash n-d}  c_{\alpha' \, \beta}^\nu \, s_{\alpha} s_\beta =
\displaystyle\sum_{\substack{\alpha \vdash d,\ \beta \vdash n-d, \\ \lambda \vdash n}}  c_{\alpha' \, \beta}^\nu \, c_{\alpha \, \beta}^\lambda \, s_{\lambda}. \notag
\end{align*}
Taking the coefficient of $s_\lambda$ on the left- and right-hand sides proves that (D) equals (C).
\end{proof}

\begin{remark} \label{r hook Schur functions}
Proposition \ref{p arm leg} is closely related to hook Schur functions and the combinatorial objects used to describe them, $(k,l)$ tableaux.
The \emph{hook Schur function} or super Schur function  $HS_\nu(\mathbf{x}, \mathbf{y})$ of Berele and Regev \cite{BereleRegev}
is the character of a certain irreducible representation of the general linear Lie superalgebra.  It can be given the following two descriptions:
the first description (\cite[Definition 6.3]{BereleRegev}) is
\be \label{e h schur2}
HS_\nu(\mathbf{x}; \mathbf{y}) = \sum_{\beta \subseteq \nu} s_\beta(\mathbf{x}) s_{\nu' / \beta'}(\mathbf{y}) = \displaystyle \sum_{\alpha', \beta \subseteq \nu} c_{\alpha' \, \beta}^\nu \, s_\beta(\mathbf{x}) s_{\alpha}(\mathbf{y}).
\ee
For the second, let $\prec^\prime$ denote the order $1\prec^\prime2\prec^\prime\cdots\crc{1}\prec^\prime\crc{2}\cdots$.  Then $\text{CT}^{\prec'}$
is the same as the set of \emph{$(k,l)$ tableaux} defined in \cite{BereleRegev}, as  $k$ and  $l$ go to infinity.
For  $T \in \text{CT}^{\prec'}$, let  $T(\mathbf{x}; \mathbf{y}) = x_1^{c_1} x_2^{c_2} \cdots \, y_1^{d_1} y_2^{d_2} \cdots$, where $(c_1,c_2,\ldots)$ is the content of $\sub_\varnothing(T)$ and  $(d_1,d_2,\ldots)$ is the content of  $\sub_\crcempty(T)$.  Then
\be \label{e h schur1}
HS_\nu(\mathbf{x}; \mathbf{y}) = \displaystyle \sum_{T \in \text{CT}^{\prec'},\ \sh(T) = \nu} T(\mathbf{x}; \mathbf{y}).
\ee

We now claim that the coefficient of $t^d s_\lambda$ in the specialization $HS_\nu(\mathbf{x};t \, \mathbf{x})$ is equal to the quantities in Proposition \ref{p arm leg}.
A direct computation using \eqref{e h schur2} shows that this coefficient is the same as (C):
\[
HS_\nu(x_1, x_2, \ldots;t x_1, t x_2,\ldots) = \displaystyle \sum_{\alpha', \beta \subseteq \nu} c_{\alpha' \, \beta}^\nu \, s_\beta(\mathbf{x}) s_{\alpha}(t\,\mathbf{x}) =
\sum_{d = 0}^n \ \sum_{\alpha' \vdash d, \ \beta \vdash n-d} c_{\alpha' \, \beta}^\nu \, c_{\alpha \, \beta}^\lambda \, t^d \, s_\lambda(\mathbf{x}).
\]
We can also specialize $ \mathbf{y} = t \, \mathbf{x}$ in \eqref{e h schur1};  with a little thought, using the beginning of the proof above that (D) equals (C) and the combinatorial definition of $s_\lambda(\mathbf{x})$, it can be shown that the coefficient of $t^d$ in this specialization is equal to $\sum_{\lambda \vdash n} |\text{CYT}_{\lambda,d}(\nu)| s_\lambda$. Hence the descriptions (C) and (D) of Proposition \ref{p arm leg} are somewhat analogous to the descriptions \eqref{e h schur2} and \eqref{e h schur1} of hook Schur functions.
\end{remark}

\subsection{Color raisable and lowerable tableaux} \label{ss color raisable and lowerable tableaux}
By Proposition \ref{p arm leg}, the Kronecker coefficient $\gcoef$ can be written as the difference
\[\gcoef = \bigg| \bigcup_{i \in \{0, 2, 4,\ldots\}} \text{CYT}_{\lambda, d-i}(\nu) \bigg| - \bigg| \bigcup_{i=\{1, 3,5 \ldots\}} \text{CYT}_{\lambda, d-i}(\nu) \bigg|.\]
This is typical for positivity problems in algebraic combinatorics:
a nonnegative coefficient is easily written as the difference in cardinality of two natural sets of combinatorial objects.  The difficulty in producing a positive combinatorial formula lies in finding an injection from the smaller of the sets to the larger.  For many sets of combinatorial objects in bijection with $\text{CYT}_{\lambda, d}(\nu)$ ($\{\convert{< \to \prec}{T} : T \in \text{CYT}_{\lambda, d}(\nu) \}$, for instance), describing such an injection seems to be extremely difficult.
The miracle in this setup is that $\text{CYT}_{\lambda, d}(\nu)$ can naturally be partitioned into two subsets with cardinalities $\gcoef$ and $g_{\lambda\,\mu(d-1)\,\nu}$.

A colored tableau for the order  $<$ is \emph{color lowerable} if its southwest entry is barred, and is \emph{color raisable} if its southwest entry is unbarred.
Hence unbarring the southwest entry of any color lowerable tableaux is a bijection between color lowerable tableaux and color raisable tableaux, which we call the \emph{color lowering operator} $C_-$. Similarly, the \emph{color raising operator} $C_+$ is the inverse of $C_-$ which acts by barring the southwest entry of any color raisable tableau.

For example,
\[C_- \left({\tiny \tableau{\crc{1}&1&\crc{2}\\\crc{1}&\crc{2}&2\\ \crc{2}&2&3\\}} \right) = {\tiny \tableau{\crc{1}&1&\crc{2}\\\crc{1}&\crc{2}&2\\ 2&2&3\\}}\ , \ \ \  C_+ \left({\tiny \tableau{\crc{1}&1&1&1\\\crc{1}&\crc{2}&2\\ 2&2 \\}} \right) = {\tiny \tableau{\crc{1}&1&1&1\\\crc{1}&\crc{2}&2\\ \crc{2}&2 \\}}.\]


Let $\text{CYT}^{-}_{\lambda}$, $\text{CYT}^{-}_{\lambda, d}$, $\text{CYT}^{-}_{\lambda, d}(\nu)$ denote the subsets of $\text{CYT}_{\lambda}$, $\text{CYT}_{\lambda, d}$, and $\text{CYT}_{\lambda, d}(\nu)$, respectively, consisting of color raisable tableaux. Similarly, let $\text{CYT}^{+}_{\lambda}$, etc. denote the corresponding sets of color lowerable tableaux.

\begin{figure}
\[
{\tiny
\definecolor{light-gray}{gray}{0.8}
\arrayrulecolor{light-gray}
\Yboxdim8pt
\begin{array}{l|l}
\text{{\normalsize $\mu(d)$}} & \hspace{3.4cm} \text{{\normalsize $\text{CYT}^-_{(3,2,1), d}$}} \\[2mm]\hline\\[0mm]
\yng(6)&
\tableau{1&1&1\\2&2\\3\\}\quad \\[6mm]\hline\\[0mm]
\yng(5,1) &
\tableau{
1&1&1&\crc{3}\\2&2\\}\quad\tableau{
1&1&1&\crc{2}\\2\\3\\}\quad\tableau{
1&1&1\\2&2&\crc{3}\\}\quad\tableau{
1&1&1\\\crc{2}&2\\3\\}\quad\tableau{
\crc{1}&1&1\\2&2\\3\\}\quad\tableau{
1&1&1\\\crc{2}\\2\\3\\}\quad\tableau{
\crc{1}&1\\1&2\\2&3\\}\quad\tableau{
\crc{1}&1\\1&2\\2\\3\\}\quad \\[10mm]\hline\\[0mm]
\yng(4,1,1) & \parbox{12cm}{$
\tableau{
1&1&1&\crc{2}&\crc{3}\\2\\}\quad\tableau{
\crc{1}&1&1&\crc{3}\\2&2\\}\quad\tableau{
1&1&1&\crc{2}\\2&\crc{3}\\}\quad\tableau{
1&1&1&\crc{3}\\\crc{2}\\2\\}\quad\tableau{
\crc{1}&1&1&\crc{2}\\2\\3\\}\quad\tableau{
\crc{1}&1&1\\2&2&\crc{3}\\}\quad\tableau{
\crc{1}&1&\crc{3}\\1&2\\2\\}\quad\\[5mm]
\tableau{
\crc{1}&1&\crc{2}\\1&2\\3\\}\quad\tableau{
\crc{1}&1&1\\\crc{2}&2\\3\\}\quad\tableau{
1&1&1\\\crc{2}&\crc{3}\\2\\}\quad\tableau{
\crc{1}&1&1\\\crc{2}\\2\\3\\}\quad\tableau{
\crc{1}&1&\crc{2}\\1\\2\\3\\}\quad\tableau{
\crc{1}&1\\1&2\\2&\crc{3}\\}\quad\tableau{
\crc{1}&1\\1&\crc{2}\\2\\3\\}\quad\tableau{
\crc{1}&1\\1&2\\\crc{2}\\3\\}\quad\tableau{
\crc{1}&1\\1\\\crc{2}\\2\\3\\}$} \\[20mm]\hline\\[0mm]
\yng(3,1,1,1) & \parbox{12cm}{$
\tableau{
\crc{1}&1&1&\crc{2}&\crc{3}\\2\\}\quad\tableau{
\crc{1}&1&1&\crc{2}\\2&\crc{3}\\}\quad\tableau{
\crc{1}&1&\crc{2}&\crc{3}\\1&2\\}\quad\tableau{
\crc{1}&1&\crc{2}\\1&2&\crc{3}\\}\quad\tableau{
\crc{1}&1&\crc{2}&\crc{3}\\1\\2\\}\quad\tableau{
\crc{1}&1&1&\crc{3}\\\crc{2}\\2\\}\quad\tableau{
\crc{1}&1&\crc{3}\\1&\crc{2}\\2\\}\quad \\[5mm]
\tableau{
\crc{1}&1&\crc{2}\\\crc{1}&2\\3\\}\quad\tableau{
\crc{1}&1&\crc{2}\\1&\crc{3}\\2\\}\quad\tableau{
\crc{1}&1&1\\\crc{2}&\crc{3}\\2\\}\quad\tableau{
\crc{1}&1\\1&\crc{2}\\2&\crc{3}\\}\quad\tableau{
\crc{1}&1&\crc{3}\\1\\\crc{2}\\2\\}\quad\tableau{
\crc{1}&1&\crc{2}\\\crc{1}\\2\\3\\}\quad\tableau{
\crc{1}&1\\\crc{1}&\crc{2}\\2\\3\\}\quad\tableau{
\crc{1}&1\\1&\crc{3}\\\crc{2}\\2\\}\quad\tableau{
\crc{1}&\crc{2}\\\crc{1}\\1\\2\\3\\}\quad$} \\[18mm]\hline\\[0mm]
\yng(2,1,1,1,1) & \tableau{
\crc{1}&1&\crc{2}&\crc{3}\\1&\crc{2}\\}\quad\tableau{
\crc{1}&1&\crc{2}&\crc{3}\\\crc{1}\\2\\}\quad\tableau{
\crc{1}&1&\crc{2}\\1&\crc{2}&\crc{3}\\}\quad\tableau{
\crc{1}&1&\crc{2}\\\crc{1}&\crc{3}\\2\\}\quad\tableau{
\crc{1}&1&\crc{3}\\\crc{1}&\crc{2}\\2\\}\quad\tableau{
\crc{1}&\crc{2}&\crc{3}\\\crc{1}\\1\\2\\}\quad\tableau{
\crc{1}&1\\\crc{1}&\crc{2}\\2&\crc{3}\\}\quad\tableau{
\crc{1}&\crc{2}\\\crc{1}&\crc{3}\\1\\2\\}\quad \\[8mm]\hline\\[-2mm]
\yng(1,1,1,1,1,1) & \tableau{
\crc{1}&\crc{2}&\crc{3}\\\crc{1}&\crc{2}\\3\\}\quad
\end{array}
}
\]
\caption{The set of color raisable Yamanouchi tableaux of content $\lambda =(3,2,1)$; the number of such tableaux of shape  $\nu$ and total color $d$ is the Kronecker coefficient  $g_{\lambda \, (6-d,1^d) \, \nu}$.}
\label{f CYT 321}
\end{figure}


Our main result, from which all else follows, is
\begin{theorem}\label{t plain commutes with color lowering}
For any color lowerable tableau $T$, $T^\plain = C_-(T)^\plain$.
\end{theorem}
This will be proved in \textsection\ref{s color raising and lowering operators on words}.

\begin{corollary}
\label{c color lowering Yamanouchi bijection}
The color lowering operator restricts to a bijection from color lowerable Yamanouchi tableaux of content  $\lambda$ and total color $d+1$ to color raisable Yamanouchi tableaux of content $\lambda$ and total color  $d$, i.e. $C_- : \text{CYT}^{\, +}_{\lambda, d+1} \xrightarrow{\cong} \text{CYT}^{\, -}_{\lambda, d}$.
\end{corollary}

\begin{theorem}[Hook Kronecker Rule I]\label{t main thm tableaux}
The Kronecker coefficient  $\gcoef$ (where $\mu(d) = (n-d,1^d)$) is equal to the number of color raisable Yamanouchi tableaux of content $\lambda$, total color $d$, and shape $\nu$. This is, by definition, the number of  colored tableaux  $T$ of shape  $\nu$, having $d$ barred entries and unbarred southwest corner, and such that $T^\plain$ is the superstandard tableau of shape and content $\lambda$.
\end{theorem}
\begin{proof}
We compute
\[
\begin{array}{rlll}
& &\hspace{-7.7mm} (1+t) \sum_{d=0}^{n-1} \gcoef \, t^d  \\[2mm]
&=& \sum_{d=0}^{n}\big( \gcoef + g_{\lambda \,\mu(d-1)\, \nu} \big) \, t^d  \\[2mm]
&=& \sum_{d=0}^n |\text{CYT}_{\lambda, d}(\nu)| \, t^d & \text{by Proposition \ref{p arm leg}}, \\[2mm]
&=& \sum_{d=0}^{n} \big( |\text{CYT}^{-}_{\lambda, d}(\nu)| + |\text{CYT}^{+}_{\lambda, d}(\nu)|\big) \, t^d & \\[2mm]
&=& \sum_{d=0}^{n} \big( |\text{CYT}^{-}_{\lambda, d}(\nu)| + |\text{CYT}^{-}_{\lambda, d-1}(\nu)|\big) \, t^d  & \text{by Corollary \ref{c color lowering Yamanouchi bijection}}, \\[2mm]
&=& (1+t) \sum_{d=0}^{n-1} |\text{CYT}^-_{\lambda,d}(\nu)| \, t^d.
\end{array}
\]
Dividing by $1+t$ and taking the coefficient of $t^d$ yields $\gcoef = |\text{CYT}^{-}_{\lambda, d}(\nu)|$, as desired.
\end{proof}

\begin{remark}\label{r natural and small bar order}
The ability to convert between the orders $<$ and $\prec$ seems to be a powerful combinatorial tool since properties easily seen in one order may be difficult to see in the other and vice versa. Here are two specific examples of this phenomenon.

The two main conditions that need to be checked to test whether $T \in \text{CYT}^-_{\lambda, d}(\nu)$ are whether $T^\plain = Z_\lambda$ and whether the southwest corner of $T$ is unbarred. Interestingly, these are difficult to check ``at the same time:'' the former is easy to check for  $\convert{< \to \prec}{T}$, but not for $T$, and the latter is immediate to check for $T$, but difficult to check for $\convert{< \to \prec}{T}$.

The Kronecker coefficient $\gcoef$  is also equal to $|\text{CYT}^+_{\lambda, d+1}(\nu)|$.  While $\text{CYT}^+_{\lambda, d+1}(\nu)$ and $\text{CYT}^-_{\lambda, d}(\nu)$ are clearly in bijection, there does not seem to be an easy bijection between $\{\convert{< \to \prec}{T} : T \in \text{CYT}^+_{\lambda, d+1}(\nu) \}$ and $\{\convert{< \to \prec}{T} : T \in \text{CYT}^-_{\lambda, d}(\nu) \}$.
\end{remark}


\section{Color raising and lowering operators on words} \label{s color raising and lowering operators on words}
Here we determine the operator $\pi_-$ such that $P\nmix(\pi_-(w)) = C_-(P\nmix(w))$ and $Q\nmix(\pi_-(w)) = Q\nmix(w)$. While the color lowering operator $C_-$ is simple, $\pi_-$ is more subtle and involves rotating a certain subword of  $w$, which we call the rightmost special subword of  $w$, once to the right.  In  \textsection\ref{ss completing the proof}, this will be used to prove Theorem \ref{t plain commutes with color lowering}, thereby completing the proof of Hook Kronecker Rule I.  Throughout this section all words, tableaux, mixed insertions, etc. are with respect to the natural order $<$.

\subsection{Decreasing hook subwords}
\begin{definition}
\label{d hook subwords}
A \emph{decreasing hook word} is a colored word  $v$ such that ${v}^{\stand \; \bneg}$ is decreasing, i.e. $v = x_1 x_2 \cdots x_k \crc{x}_{k+1} \cdots \crc{x}_n$  and $x_1 > x_2 > \cdots > x_k$ and $\crc{x}_{k+1} \leq \ldots \leq \crc{x}_n$.  A \emph{decreasing hook subword} of a colored word  $w$ is a subword of  $w$ that is a decreasing hook word.
For a colored word  $w$, let $\tau(w)$ be the maximum possible length of a decreasing hook subword of $w$.

Given a colored word $w$, set $t := \tau(w)$ and let $\eta$ be smallest letter of $w$ (for  $<$) such that $\sub_{\leq \eta}(w)$ has a decreasing hook subword of length  $t$ (see Proposition \ref{p special subwords basics}, below, for a way to compute these values).
We say that a  decreasing hook subword of  $w$ is a \emph{special subword} if it has length  $t$ and uses letters $\leq \eta$.  See Example \ref{ex pi-}.
\end{definition}

For a finite poset $\mathcal{P}$, the \emph{set of Sperner 1-families}, denoted $\mathscr{S}_1(\mathcal{P})$, is the set of antichains of $\mathcal{P}$ of maximum size. The set $\mathscr{S}_1(\mathcal{P})$ is partially ordered as follows: if $A, B \in \mathscr{S}_1(\mathcal{P})$, then $A \leq B$ if, for each $a \in A$, there exists some $b \in B$ such that $a \leq b$.
Dilworth proved (see, e.g., \cite{GK}) that $\mathscr{S}_1(\mathcal{P})$ is a distributive lattice.  In particular, $\mathscr{S}_1(\mathcal{P})$ has a unique minimum and maximum.

\begin{definition}
For an ordinary word $y$ of length $n$, let $\text{Pos}(y)$ be the poset on $[n]$ in which $i$ is less than $j$ if and only if $i < j$ and $y_i \leq y_j$.
Thus a decreasing subword of  $y$ of length  $\tau(y)$ is the same as an element of $\mathscr{S}_1(\text{Pos}(y))$.
Given  $y_\mathbf{j}, y_\mathbf{k} \in \mathscr{S}_1(\text{Pos}(y))$, we say that $y_\mathbf{j}$ is \emph{further left} (resp. \emph{further right}) than $y_\mathbf{k}$ if  $y_\mathbf{j}$ is less than (resp.   greater than) $y_\mathbf{k}$ in the partial order on Sperner 1-families defined above. We refer to the minimum (resp.  maximum) element of  $\mathscr{S}_1(\text{Pos}(y))$ as the \emph{leftmost} (resp. \emph{rightmost}) longest decreasing subword of $y$.

For a colored word $w$, define $\text{Pos}(w)$ to be the poset $\text{Pos}(y)$ just defined, with $y = w^{\stand \; \bneg}$.
Thus a decreasing hook subword of  $w$ of length  $\tau(w)$ is the same as an element of $\mathscr{S}_1(\text{Pos}(w))$, and a special subword of  $w$ is the same as an element of $\mathscr{S}_1(\text{Pos}(\sub_{\leq \eta}(w)))$, where $\eta$ is as defined above.
\end{definition}

It turns out that the leftmost and rightmost longest decreasing subwords have a more direct description than their definition above.
If $\mathbf{k}$ and  $\mathbf{j}$ are place words of length  $t$, then  $\mathbf{k}$ is \emph{componentwise less than or equal to}  $\mathbf{j}$  if  $k_i \leq j_i$ for all  $i \in [t]$.
\begin{proposition}
\label{p lexicographic eq componentwise}
The leftmost longest decreasing subword of an ordinary word is the unique minimum for the componentwise order.
Precisely, let  $y$ be an ordinary word and let  $\mathbf{k} = k_1 k_2 \cdots k_t$ be the place word of the leftmost longest decreasing subword of $y$.  Let $\mathbf{j}= j_1 j_2\cdots j_t$ be a place word of $y$ such that $y_\mathbf{j}$ is decreasing.
Then $\mathbf{k}$ is componentwise less than or equal to $\mathbf{j}$.

Similarly, the rightmost longest decreasing subword of an ordinary word is the unique maximum for the componentwise order.
\end{proposition}
\begin{proof}
Let $i \in [t]$. Suppose for a contradiction that $k_i > j_i$.  If  $y_{k_i} < y_{j_i}$, then  $j_1 \cdots j_i k_i k_{i+1}\cdots k_t$ is the place word of a decreasing subword of length $t+1$, which is impossible.  If  $y_{k_i} \geq y_{j_i}$, then  $y_{k_{i-1}} > y_{k_i} \geq y_{j_i}$, hence $k_1 \cdots k_{i-1} j_i j_{i+1} \cdots j_t$ is the place word of a decreasing subword that is not further right than $y_\mathbf{k}$, contradiction.  The proof of the second statement is similar.
\end{proof}

For a colored word $w$, let $\sw(w)$ denote the southwest entry of $P_\mix(w)$.
The next corollary relates $\tau(w)$ and $\eta$ 
defined above to  $P_\mix(w)$.
We point out that Remmel also defines and studies decreasing hook subwords in \cite{Remmel2} (called decreasing subsequences of type 1 there); he also states the first part of the following proposition.
\begin{proposition}
\label{p special subwords basics}
Let  $w$ be a colored word and let $\eta$ be as in Definition \ref{d hook subwords}.
\begin{list}{\emph{(\roman{ctr})}}{\usecounter{ctr} \setlength{\itemsep}{2pt} \setlength{\topsep}{3pt}}
\item The length  $\tau = \tau(w)$ of the longest decreasing hook subword of $w$ is equal to the length of the first column of $P_\mix(w)$.
\item The letter $\eta$ is equal to $\sw(w)$.
\item If  $\eta$ is barred, then any special subword of  $w$ contains the rightmost occurrence of the letter $\eta$ in $w$.
\item If  $\eta$ is unbarred, then the leftmost special subword of  $w$ contains the leftmost occurrence of the letter $\eta$ in $w$.
\item If  $\eta$ is barred  and  $w_\mathbf{k}$ is the rightmost special subword of  $w$, then all occurrences of $\eta^*$ in  $w$ have place $> k_1$.
\item If  $\eta$ is unbarred and $w_\mathbf{k}$ is the leftmost special subword of  $w$, then all occurrences of $\eta^*$ in  $w$ have place $\leq k_\tau$.
\end{list}
\end{proposition}
\begin{proof}
The analog of (i) for ordinary words is the classical Greene's Theorem \cite{Greene}.  Statement (i) is immediate from this and Proposition \ref{p neg and mixed insertion}. Statement (ii) follows from (i) and Proposition \ref{p mix commutes with subwords}.

Let  $w_\mathbf{j}$ be a special subword of  $w$.
By definition,  $w$ contains letters $\leq \eta$, so if  $\eta$ is barred and $w_\mathbf{j}$ does not contain the rightmost occurrence of  $\eta$, then this can be appended to $w_\mathbf{j}$ to obtain a longer decreasing hook subword, which is impossible. This proves (iii).
For (iv), observe that if  $w_\mathbf{j}$ does not contain the leftmost occurrence of  $\eta$, then replacing $w_{j_1} = \eta$ with the leftmost occurrence of $\eta$ yields a special subword of  $w$ further left than  $w_\mathbf{j}$.

To prove (vi), observe that any occurrence of  $\eta^*$ with place $> k_\tau$ can be appended to  $w_\mathbf{k}$ to obtain a decreasing hook subword of  $w$ of length  $\tau+1$, which is impossible.  The proof of (v) is similar.
\end{proof}

We are now ready to define the color lowering and raising operators on words.
For a colored word $w$ and place word $\mathbf{k}$ of length $t$ such that $w_{k_t}$ is barred, let $\pi_\mathbf{k}(w)$ be the colored word obtained from $w$ by rotating its subword $w_{\mathbf{k}}$ once to the right and then unbarring $w_{k_t}$, i.e.,
\[\pi_\mathbf{k}(w) := w_1 \cdots w_{k_1-1} \mathbf{w_{k_t}^*}  w_{k_1+1} \cdots w_{k_2-1} \mathbf{w_{k_1}} w_{k_2+1} \cdots w_{k_t-1}  \mathbf{w_{k_{t-1}}} w_{k_t+1} \cdots w_n, \]
where the bold letters indicate the rotated subword.
It is clear that $\pi_\mathbf{k}$ is invertible and defines a bijection from colored words with a barred letter in position $k_t$ to colored words with an unbarred letter in position $k_1$. Let $\pi^{-1}_\mathbf{k}$ denote the inverse of $\pi_\mathbf{k}$.

We say that a colored word $w$ is \emph{color lowerable} (resp. \emph{color raisable}) if $\sw(w)$ is barred (resp. unbarred).
For a color lowerable word $w$, define the \emph{color lowering operator on words}, $\pi_-$, by
\[\pi_-(w) := \pi_\mathbf{k}(w) ,\quad \text{where $\mathbf{k}$ is the rightmost special subword of $w$}.\]
For a color raisable word $v$, define the \emph{color raising operator on words}, $\pi_+$, by
\[\pi_+(v) := \pi^{-1}_\mathbf{k}(v) ,\quad \text{where $\mathbf{k}$ is the leftmost special subword of $v$}.\]
Note that these operators are well defined by Proposition \ref{p special subwords basics}.

\begin{example}\label{ex pi-}
Let $w$ and $v$ be the colored words below. The rightmost special subword of $w$ and the leftmost special subword of $v$ are shown in bold and their place words are $1\ 3\ 8\ 11\ 12$. From this we can see that $w$ is color lowerable and $v = \pi_-(w) =\pi_{1\ 3\ 8\ 11\ 12}(w)$ and $v$ is color raisable and $w = \pi_+(v) =\pi^{-1}_{1\ 3\ 8\ 11\ 12}(v)$.
\vspace{1mm}
\[
{\tiny
\begin{array}{rlcccccccccccccc}
    \text{{\small $w$}}                  &  =  &  \mboxtwo{\mathbf{1}} & \mboxtwo{\crc{2}} & \mboxtwo{\mathbf{\crc{1}}} & \mboxtwo{\crc{2}} & \mboxtwo{2} & \mboxtwo{1} & \mboxtwo{\crc{2}} & \mboxtwo{\mathbf{\crc{1}}} & \mboxtwo{1} & \mboxtwo{2} & \mboxtwo{\mathbf{\crc{1}}} & \mboxtwo{\mathbf{\crc{2}}} & \mboxtwo{1} \\[0.6mm]
    \text{{\small $v$}}             &  =  & \mboxtwo{ \mathbf{2}} & \mboxtwo{\crc{2}} & \mboxtwo{\mathbf{1}} & \mboxtwo{\crc{2}} & \mboxtwo{2} & \mboxtwo{1} & \mboxtwo{\crc{2}} & \mboxtwo{\mathbf{\crc{1}}} & \mboxtwo{1} & \mboxtwo{2} & \mboxtwo{\mathbf{\crc{1}}} & \mboxtwo{\mathbf{\crc{1}}} & \mboxtwo{1} \\[2mm]
    \text{{\normalsize ${w}^{\stand}$}}   &  =  & \mboxtwo{\mathbf{4}} & \mboxtwo{ \crc{8}} & \mboxtwo{\mathbf{\crc{1}}} & \mboxtwo{\crc{9}} & \mboxtwo{12} & \mboxtwo{5} & \mboxtwo{\crc{10}} & \mboxtwo{\mathbf{\crc{2}}} & \mboxtwo{6} & \mboxtwo{13} & \mboxtwo{\mathbf{\crc{3}}} & \mboxtwo{\mathbf{\crc{11}}} & \mboxtwo{7} \\[0.6mm]
    \text{{\small ${v}^{\stand}$}}   &  =  & \mboxtwo{\mathbf{11}} & \mboxtwo{ \crc{8}} & \mboxtwo{\mathbf{4}} & \mboxtwo{\crc{9}} & \mboxtwo{12} & \mboxtwo{5} & \mboxtwo{\crc{10}} & \mboxtwo{\mathbf{\crc{1}}} & \mboxtwo{6} & \mboxtwo{13} & \mboxtwo{\mathbf{\crc{2}}} & \mboxtwo{\mathbf{\crc{3}}} & \mboxtwo{7} \\[2mm]
    \text{{\small ${w}^{\stand\;\bneg}$}}           &  =  & \mboxtwo{\mathbf{4}} & \mboxtwo{ -{8}} & \mboxtwo{\mathbf{-{1}}} & \mboxtwo{-{9}} & \mboxtwo{12} & \mboxtwo{5} & \mboxtwo{-{10}} & \mboxtwo{\mathbf{-{2}}} & \mboxtwo{6} & \mboxtwo{13} & \mboxtwo{\mathbf{-{3}}} & \mboxtwo{\mathbf{-{11}}} & \mboxtwo{7} \\[0.6mm]
    \text{{\small ${v}^{\stand\;\bneg}$}}      &  =  & \mboxtwo{\mathbf{11}} & \mboxtwo{ -{8}} & \mboxtwo{\mathbf{4}} & \mboxtwo{-{9}} & \mboxtwo{12} & \mboxtwo{5} & \mboxtwo{-{10}} & \mboxtwo{\mathbf{-{1}}} & \mboxtwo{6} & \mboxtwo{13} & \mboxtwo{\mathbf{-{2}}} & \mboxtwo{\mathbf{-{3}}} & \mboxtwo{7}
\end{array}
}
\vspace{3mm}
\]
There are a total of four decreasing hook subwords of $w^\stand$ of length 5 (these are in bijection with decreasing hook subwords of $w$ and decreasing subwords of  $w^{\stand \; \bneg}$): $4\ \crc{1}\ \crc{2}\ \crc{3}\ \crc{11}$, $4\ \crc{1}\ \crc{9}\ \crc{10}\ \crc{11}$, $4\ \crc{8}\ \crc{9}\ \crc{10}\ \crc{11}$, and $12\ 5\ \crc{2}\ \crc{3}\ \crc{11}$; the first three are special and the fourth is not.
There are a total of three decreasing hook subwords of $v^\stand$ of length 5: $11\ 4\ \crc{1}\ \crc{2}\ \crc{3}$, $11\ 5\ \crc{1}\ \crc{2}\ \crc{3}$, and $12\ 5\ \crc{1}\ \crc{2}\ \crc{3}$; the first two are special and the third is not.

It will be shown in Theorem \ref{t main thm words} that the color lowering operator $(C_-)$ is compatible with the color lowering operator on words $(\pi_-)$ in the following sense:
\vspace{1mm}
\[
\begin{array}{c@{\hspace{.8cm}}c@{\hspace{.7cm}}c}
     \substack{{\tiny \tableau{\crc{1}&1&1&1&\crc{2}\\\crc{1}&\crc{2}&2\\\crc{1}&\crc{2}\\1&2\\\crc{2}\\}} \vspace{3mm} \\ P_\mix(w)} &
     \substack{{\tiny \tableau{\crc{1}&1&1&1&\crc{2}\\\crc{1}&\crc{2}&2\\\crc{1}&\crc{2}\\1&2\\2\\}} \vspace{3mm} \\ P_\mix(\pi_-(w))  =  C_-(P_\mix(w)) } &
     \substack{{\tiny \tableau{1&3&5&9&10\\2&6&13\\4&8\\7&11\\12\\}} \vspace{3mm} \\  Q_\mix(\pi_-(w)) = Q_\mix(w) }
\end{array}
\]
\end{example}

\begin{proposition}
\label{p standardization hook subwords}
Standardization respects decreasing hook subwords and commutes with the color lowering and raising operators:
\begin{list}{\emph{(\alph{ctr})}}{\usecounter{ctr} \setlength{\itemsep}{2pt} \setlength{\topsep}{3pt}}
\item ${C_-(T)}^{\stand} = C_-({T}^{\stand})$,
\item ${C_+(T)}^{\stand} = C_+({T}^{\stand})$,
\item $w_\mathbf{j}$ is a decreasing hook subword of  $w$ if and only if $({w}^{\stand})_\mathbf{j}$ is a decreasing hook subword of ${w}^{\stand}$,
\item same as (c), for special subwords, if  $\eta$ is barred,
\item same as (c), for the leftmost special subword, if  $\eta$ is unbarred,
\item ${\pi_-(w)}^{\stand} = \pi_-({w}^{\stand})$,
\item ${\pi_+(w)}^{\stand} = \pi_+({w}^{\stand})$,
\end{list}
for any colored word $w$ and colored tableau $T$.
\end{proposition}
\begin{proof}
Statements (a)--(c) are immediate from the definitions.  By Proposition \ref{p standardization commutes} and Proposition \ref{p special subwords basics} (ii),
the rightmost (resp.  leftmost) occurrence of $\eta := \sw(w)$ is relabeled by $\sw(w^\stand)$ in the standardization  $w^\stand$ if  $\eta$ is barred (resp.   unbarred).
This, together with (c) and Proposition \ref{p special subwords basics} (iii),(iv), yield (d) and (e).
Finally, (f) follows from (d) and Proposition \ref{p special subwords basics} (iii),(v), and (g) follows from (e) and Proposition \ref{p special subwords basics} (iv),(vi).
\end{proof}

\subsection{Compatibility of the color lowering operators $C_-$ and $\pi_-$}
We now prove the relationship between $C_-$ and $\pi_-$ alluded to in Example \ref{ex pi-}.

We will need the following extension of Proposition \ref{p remove largest letter}.
\begin{lemma}
\label{l remove two largest letters}
Let  $w = w_1 \cdots w_n$ be a colored permutation with largest letter $w_n = \crc{n}$ and second-largest letter $w_b$. Set $w' = w_1 \cdots w_{b-1} w_{b+1} \cdots w_n$ and  $\beta = (w^\inv)_{n-1}$ (thus $\beta = b$ if  $w_b$ is unbarred and $\beta = \crc{b}$ if  $w_b$ is barred).   Let $Q'$ be the tableau obtained from $Q_\mix(w')$ by replacing  $n-1$ with  $n$,  $n-2$ with  $n-1$,  $\ldots$, $b$ with  $b+1$.  If $\tau(w') = \tau(w)$, then
\be \label{el for step 4 wn}
\text{$P_\mix(w) = P_\mix(w') \sqcup  {\tiny \tableau{w_b}}_{(r,c)}$\ \ and\ \ $Q_\mix(w) = Q' \xleftarrow{\lr} \beta$,}
\ee
where  $(r,c)$ is the position of the cell $\sh(Q_\mix(w))/\sh(Q')$.

Similarly, suppose $v$ is a colored permutation with largest letter $v_1 = n$ and second-largest letter $v_b$.  Let  $v'$, $Q'$, and  $(r,c)$ be defined just as $w'$, $Q'$, and $(r,c)$ are above. If $\tau(v') = \tau(v)$, then
\be \label{el for step 4 w1}
\text{$P_\mix(v) = P_\mix(v') \sqcup  {\tiny \tableau{v_b}}_{(r,c)}$\ \ and\ \ $Q_\mix(v) = Q' \xleftarrow{\lr} \beta$.}
\ee
\end{lemma}
\begin{proof}
As in the proof of Proposition \ref{p remove largest letter}, we work with left-right insertion of $w^\inv$ instead of mixed insertion of $w$.
Set $(w^\inv)_L = (w^\inv)_1 (w^\inv)_2 \cdots (w^\inv)_{n-2}$. Note that $(w^\inv)_n = \crc{n}$.  We first prove
\be \label{e two insertion orders}
P_\lr((w^\inv)_L) \xleftarrow{\lr} \beta \xleftarrow{\lr} \crc{n} = P_\lr((w^\inv)_L) \xleftarrow{\lr} \crc{n} \xleftarrow{\lr} \beta.
\ee
Set  $\tau = \tau(w)$.
By the assumption  $\tau(w')=\tau(w)$, the number of rows of  $P_\lr(( w^\inv)_L)$, $P_\lr(( w^\inv)_L \, \beta)$, $Q' = P_\lr(( w^\inv)_L \, \crc{n})$, and $P_\lr(w^\inv)$ are  $\tau-1$,  $\tau-1$,  $\tau$, and $\tau$, respectively.  Hence the left-right insertion of $ \crc{n}$ on either side of \eqref{e two insertion orders} simply adds the letter $n$ in a new cell at position $(\tau,1)$ and the insertion path of the left-right insertion of $\beta$ on either side of \eqref{e two insertion orders} does not involve position  $(\tau,1)$.  This proves \eqref{e two insertion orders} and, keeping track of recording tableaux of these left-right insertions, gives
\hoogte=13pt
\breedte=20pt
\Yboxdim20pt
\be \label{e two recording orders}
\begin{array}{lcl}
Q_\lr(w^\inv) &=& Q_\lr((w^\inv)_L) \sqcup {\tiny\tableau{w_b}}_{(r,c)} \sqcup {\tiny\tableau{\crc{n}}}_{(\tau,1)}, \\[2mm]
Q_\lr((w^\inv)_L \ \crc{n}) &=& Q_\lr((w^\inv)_L) \sqcup {\tiny\begin{Young}$\crc{n\!-\!1}$ \cr \end{Young}}_{\, (\tau,1)},
\end{array}
\ee
Noting that  $P_\mix(w')$ is obtained from $Q_\lr((w^\inv)_L\, \crc{n})$ by replacing $\crc{n-1}$ with $\crc{n}$, the desired \eqref{el for step 4 wn} now follows from computations similar to those in the proof of Proposition \ref{p remove largest letter}.

The second statement of the lemma follows from the first applied to $w := v^{\rev \; *}$: to avoid confusion, let  $b_w$, $\beta_w$, $Q'_w$ (resp.   $b_v$, $\beta_v$, $Q'_v$) be  $b$, $\beta$, $Q'$ for \eqref{el for step 4 wn} (resp. \eqref{el for step 4 w1}).  The desired result about $P_\mix(v)$ is immediate from \eqref{el for step 4 wn} and Proposition \ref{p rev star mixed} (i),(iii).  For the desired result about mixed recording tableaux, we first assume $\beta_w$ is unbarred and compute
\[ Q_\mix(v) = {Q_\mix(w)}^{\evac} = \big(Q'_w \xleftarrow{\lr} \beta_w\big)^{\evac} \ke (\reading(Q'_w) b_w)^{\ud \; \rev} \ke b_v \, \reading(Q'_v) \ke  \big(Q'_v \xleftarrow{\lr} \beta_v\big), \]
where the first equality is by Proposition \ref{p rev star mixed} (ii), (iv) and the first  plactic equivalence is by \eqref{e ud rev}; the second plactic equivalence follows from \eqref{e ud rev}, $Q_\mix(w')= Q_\mix((v')^{\rev \; *}) = Q_\mix(v')^\evac$, and  $b_v = n+1-b_w$.  The case $\beta_w$ is barred is similar.
\end{proof}

\begin{theorem}
 \label{t main thm words}
For a color lowerable word $w$,
\be \label{et rotate words-}
\text{$P_\mix(\pi_-(w)) = C_-(P_\mix(w))$\ \ and\ \ $Q_\mix(\pi_-(w)) = Q_\mix(w)$.}
\ee
Similarly, for a color raisable word $v$,
\be \label{et rotate words+}
\text{$P_\mix(\pi_+(v)) = C_+(P_\mix(v))$\ \ and\ \ $Q_\mix(\pi_+(v)) = Q_\mix(v)$.}
\ee
\end{theorem}
\begin{proof}
We first show that \eqref{et rotate words+} follows from \eqref{et rotate words-} by applying \eqref{et rotate words-} to $w := v^{\rev \; *}$. The operators ${}^{\rev}$ and  ${}^*$ do not commute with  standardization, so we need to assume that $v$ is a colored permutation (this implies the general case by Step 1, below).  The automorphism $u \mapsto u^{\rev \; *}$ of colored permutations identifies leftmost special subwords with rightmost special subwords, so  $\pi_+(v)^{\rev \; *} = \pi_-(w)$.  This gives the second to last equality of
\[ C_+(P_\mix(v)) = C_+(P_\mix(w)^{*\;\transpose}) = C_-( P_\mix(w))^{*\;\transpose} = P_\mix(\pi_-(w))^{*\;\transpose} = P_\mix({\pi_+(v)}^{\rev\;*})^{*\;\transpose} = P_\mix(\pi_+(v));\]
the first and last equalities are by Proposition \ref{p rev star mixed} (i), (iii), the middle equality is by \eqref{et rotate words-}, and the second equality is clear.
A similar computation using Proposition \ref{p rev star mixed} (ii), (iv) yields $Q_\mix(\pi_+(v)) = Q_\mix(v)$.

We now prove \eqref{et rotate words-}. Let  $w$ be a color lowerable word and set $v = \pi_-(w)$.
Let $\tau = \tau(w)$ be the maximum length of a decreasing hook subword of $w$.  Let  $\eta = \sw(w)$ be the southwest entry of $P_\mix(w)$; we are assuming that this entry is barred, so set $\crc{x} = \eta$.
Let $\mathbf{k}$ be the place word of the rightmost special subword of $w$; thus $w_{k_\tau} = \eta$ by Proposition \ref{p special subwords basics} (iii).
Let  $n$ be the length of $w$.
The proof is by induction on $n$. The base case $n=1$ is clear. The proof begins with three straightforward reductions (Steps 1--3), followed by a substantial reduction (Step 4) and two cases (Steps 5a and 5b). Step 4 is particularly interesting because it explains why it is the rightmost special subword that needs to be rotated (and not some other subword, for instance).

\textbf{{Step 1}} It is convenient to assume that $w$ is a colored permutation, and this is accomplished by replacing $w$ with ${w}^{\stand}$.  The theorem for  ${w}^{\stand}$ proves it for $w$ by Propositions \ref{p standardization commutes} and \ref{p standardization hook subwords}.

\textbf{{Step 2}}
We may assume that $\crc{x}$ is the largest letter in $w$ (for $<$). If not, let $\alpha > \crc{x}$ be the largest letter in $w$ and let $w'$ (resp. $v'$) be  $w$ (resp. $v$) with  $\alpha$ removed.  Then $\pi_-(w') = v'$ because  $\alpha$  does not belong to the rightmost special subword of  $w$ (by Proposition \ref{p special subwords basics} (ii)). By induction, $C_-(P_\mix(w')) =  P_\mix(v')$ and $Q_\mix(w') = Q_\mix(v')$.
Now Proposition \ref{p remove largest letter} says that $Q_\mix(w)$ and $Q_\mix(v)$ are obtained from $ Q_\mix(w') = Q_\mix(v')$ by the same procedure, hence $Q_\mix(w) = Q_\mix(v)$.  Proposition \ref{p remove largest letter} also proves $C_-(P_\mix(w))= C_-(P_\mix(w')) \sqcup {\tiny \tableau{\alpha}}_{(r,c)} = P_\mix(v') \sqcup {\tiny \tableau{\alpha}}_{(r,c)}=P_\mix(v)$; here we are using that $(r,c)$ is not the position of the southwest cell of  $P_\mix(w)$, which follows from the fact that  $\alpha$ does not belong to the rightmost special subword of  $w$.

Note that once we assume $\crc{x}$ is the largest letter of  $w$, this implies that the bottom
($\tau$-th)  row of  $P_\mix(w)$ consists of a single cell containing $\crc{x}$.

\textbf{{Step 3}}
We may assume that  $\crc{x}$ is the last letter of  $w$, i.e.  $k_\tau = n$, and that $x$ is the first letter of $v$, i.e.  $k_1 =1$.
We will only show that the case  $k_\tau < n$ can be reduced to the case $k_\tau =n$, the reduction from $k_1 > 1$ to $k_1 =1$ being similar.
Suppose $k_\tau < n$ and set $w' = w_1 w_2 \cdots w_{n-1}$ and $v' = v_1 v_2 \cdots v_{n-1}$.
Deleting $w_n$ does not change the rightmost special subword, that is, $w'_\mathbf{k}$ is the rightmost special subword of  $w'$, hence $\pi_-(w') = v'$.
This implies $\sw(w') = \sw(w) = \crc{x}$ and  $\tau(w')=\tau(w)$ (by Proposition \ref{p special subwords basics} (i), (ii)).
Now we claim that the insertion paths of $P_\mix(w') \xleftarrow{\mix} w_{n}$ and $P_\mix(v') \xleftarrow{\mix} v_{n}$ are identical and do not involve positions $(\tau,1)$ and $(\tau+1,1)$.
If the insertion path of $P_\mix(w') \xleftarrow{\mix} w_{n}$ involved $(\tau,1)$ or $(\tau+1,1)$, then $\sw(w') \neq \sw(w)$ or $\tau(w') \neq \tau(w)$, which is impossible. Then since $P_\mix(w')$ and $P_\mix(v')$ differ only in their southwest cell, and $\crc{x}$ and $x$ are the largest letters of $P_\mix(w')$ and $P_\mix(v')$, respectively, the claim follows (this uses Step 2, which is not strictly necessary, but makes this argument slightly easier to say).
This claim and induction give the desired equalities in \eqref{et rotate words-}.

\textbf{{Step 4}}
By Steps 2 and 3, we may assume that $w_n = \crc{x}$ is the largest letter in $w$ and $v_1 = v_{k_1} = x$.
Set $w_L = w_1 w_2 \cdots w_{n-1}$ and $v_R = v_2 v_3 \cdots v_{n}$. Let $\eta'$ be the entry in position $(\tau-1,1)$ of  $P_\mix(w)$ (we are assuming $n>1$, so by the note at the end of Step 2, $P_\mix(w)$ has at least two rows).

The goal of Step 4 is to reduce to the case that $\eta'$ is the largest letter of $w_L$; the argument is similar to, but more involved than Step 2.
Suppose that the largest letter  $\alpha$ of $w_L$ is not  $\eta'$.  Let  $w'$ (resp. $v'$) be  $w$ (resp. $v$) with  $\alpha$ removed.
Let us first establish some basic facts.  Note that decreasing hook subwords of $w$ of length  $\tau$ must use  $ \crc{x}$, hence
\be \label{e w wL bijection}
\parbox{14cm}{the map  $(w_L)_\mathbf{j} \mapsto (w_L)_\mathbf{j}\, \crc{x}$ is a bijection between decreasing hook subwords of $w_L$ of length  $\tau-1$ and decreasing hook subwords of $w$ of length $\tau $.}
\ee
Because $w_n = \crc{x}$ is the largest letter of $w$ and by the note at the end of Step 2,  $P_\mix(w) = P_\mix(w_L) \sqcup {\tiny\tableau{\crc{x}}}_{(\tau,1)}$.  Hence
\be
\sw(w_L) = \eta' < \alpha. \label{e2 step 4 basics}
\ee

We now prove
\be
\label{e pi 4}
\pi_-(w') = v'.
\ee
To prove this, we must show that $w_\mathbf{k}$ is the rightmost special subword of $w'$.  In particular, we must show that
\be  \label{e alpha does not belong}
\text{$\alpha$ does not belong to $w_{\mathbf{k}}$.}
\ee
In fact, this is sufficient as any special subword of  $w'$ further right than  $w_\mathbf{k}$ would yield a special subword of  $w$ further right than  $w_\mathbf{k}$.
We now suppose that  $\alpha$ does belong to $w_\mathbf{k}$ and will obtain a contradiction. There are two cases depending on whether or not $\alpha$ is barred.

\emph{The case $\alpha$ is unbarred:}
In this case, we must have $w_1 = w_{k_1} = \alpha$.  Then by \eqref{e2 step 4 basics},  $w_1$ does not belong to the rightmost special subword of  $w_L$.
By Proposition \ref{p lexicographic eq componentwise} applied to  $w^{\bneg}$, every decreasing hook subword  $w_\mathbf{j}$ of  $w$ of length $\tau$ must satisfy $j_1 \leq k_1 = 1$, i.e.  $w_\mathbf{j}$ must contain  $w_1$.
Then by \eqref{e w wL bijection}, any decreasing hook subword of $w_L$ of length  $\tau -1$ must contain  $w_1$, contradiction.

\emph{The case $\alpha$ is barred:}
In this case,  $w_{k_{\tau-1}} = \alpha$.  Let  $\eta'' $ be the rightmost letter in the rightmost special subword of  $w_L$.
First note that \eqref{e2 step 4 basics} implies $\eta'' \leq \eta' < \alpha$.
Consider the subword of $w$ obtained by adding  $\crc{x}$ to the end of the rightmost special subword $w_L$.  Comparing this to  $w_\mathbf{k}$ using Proposition \ref{p lexicographic eq componentwise} shows that $\eta''$ lies to the left of $\alpha$.
But this implies $\alpha$ can be added to the end of the rightmost special subword of $w_L$ to obtain a decreasing hook subword of  $w_L$ of length $\tau$, contradiction.

Now that \eqref{e pi 4} has been established, induction yields $C_-(P_\mix(w')) =  P_\mix(v')$ and $Q_\mix(w') = Q_\mix(v')$.
The desired result \eqref{et rotate words-} now follows from Lemma \ref{l remove two largest letters} (with $w_b = v_b = \alpha$).


\medskip

We may now assume that  $\eta' < \crc{x}$ are the two largest letters in  $w$.
Note that this implies that the last two rows of  $P_\mix(w)$ look like  ${\tiny \tableau{\eta' \\ \crc{x}}}$ with no cells to their right.
Let  $P_0^w$ be the result of  removing the last two rows of $P_\mix(w)$.
Now there are two cases: either $\eta'$ is barred (Step 5a) or it is unbarred (Step 5b).


\textbf{{Step 5a}}
Set  $\crc{y} = \eta'$,
\begin{align*}
w' &= w_L, \\
v_{LR} &= v_2 v_3 \cdots v_{n-1}, \ \  v_L = v_1 v_2 \cdots v_{n-1},\ \  v' = y \, v_{LR},
\end{align*}
and  $\mathbf{k}' = k_1\, k_2  \cdots k_{\tau-1}$ (see Example \ref{ex step 5}).
By the definitions, $\mathbf{k}'$ is the place word of a decreasing hook subword of $w'$ and of $v'$ and  $\pi_{\mathbf{k}'}(w') = v'$.
By Step 4, every decreasing hook subword of $w_L$ (resp.  $w$) of length  $\tau -1$ (resp.  $\tau$) is a special subword.  Together with \eqref{e w wL bijection}, this establishes that  $(w_L)_{\mathbf{k}'}$ is the right most special subword of $w_L$.
Therefore $\pi_-(w') = v'$.
By induction,
\be \label{e step 5a induction}
\text{$C_-(P_\mix(w')) =  P_\mix(v')$\ \ and\ \ $Q_\mix(w') = Q_\mix(v')$.}
\ee

From our work in Step 4, the tableaux $P_\mix(w)$ and $Q_\mix(w)$ are easily computed in terms of $P_\mix(w')$ and $Q_\mix(w')$:
\be \label{e step5a w tableaux}
\begin{array}{lllll}
P_\mix(w') & = & P_0^w \sqcup {\tiny \tableau{\crc{y}}}_{(\tau-1,1)}, \\[1.3mm]
P_\mix(w) & = & P_0^w \sqcup {\tiny \tableau{\crc{y}}}_{(\tau-1,1)} \sqcup {\tiny \tableau{\crc{x}}}_{(\tau,1)}, \\[1.3mm]
Q_\mix(w) &=& Q_\mix(w') \sqcup {\tiny \tableau{n}}_{(\tau,1)}.
\end{array}
\ee

The tableaux for $v$ and $v'$ require slightly more care to compute:
\be \label{e step5a v tableaux}
\begin{array}{lllll}
P_\mix(v') & = & P_0^w \sqcup {\tiny \tableau{y}}_{(\tau-1,1)}, \\[1.3mm]
P_\mix(v') & = & P_\mix(v_{LR}) \xleftarrow{\dualmix} y, \\[1.3mm]
P_\mix(v_L) & = & P_\mix(v_{LR}) \xleftarrow{\dualmix} x &=& P_0^w \sqcup {\tiny \tableau{x}}_{(\tau-1,1)}, \\[1.3mm]
P_\mix(v) & = & P_\mix(v_L) \xleftarrow{\mix} \crc{y} & = & P_0^w \sqcup {\tiny \tableau{\crc{y}}}_{(\tau-1,1)} \sqcup {\tiny \tableau{x}}_{(\tau,1)}, \\[1.3mm]
Q_\mix(v) &=& Q_\mix(v_L) \sqcup {\tiny \tableau{n}}_{(\tau,1)} &=& Q_\mix(v') \sqcup {\tiny \tableau{n}}_{(\tau,1)}.
\end{array}
\ee
The first line is immediate from \eqref{e step 5a induction} and \eqref{e step5a w tableaux}.  The second line is clear given Proposition \ref{p dual mixed}.  The third line follows from first two and the fact that $(v')^{\stand} = (v_L)^{\stand}$.
For the fourth line, the mixed insertion of $\crc{y}$ bumps the $x$ in position $(\tau-1,1)$ and then places $x$ in a new cell at position $(\tau,1)$.
This mixed insertion computation together with $(v')^{\stand} = (v_L)^{\stand}$ gives the last line.
The desired result \eqref{et rotate words-} now follows from \eqref{e step 5a induction}, \eqref{e step5a w tableaux}, and \eqref{e step5a v tableaux}.

\textbf{{Step 5b}}
Set  $y = \eta'$,
\begin{align*}
w_{LR} &= w_2 w_3 \cdots w_{n-1},\ \ w_R = w_2 w_3 \cdots w_{n},\ \ w' = w_{LR}\, \crc{y},\\
v' &= v_R,
\end{align*}
and  $\mathbf{k}' = k_2-1\ k_3-1 \, \cdots \, k_{\tau}-1$ (see Example \ref{ex step 5}).
By the definitions and Steps 2 and 3, $w_1 = y$, $\mathbf{k}'$ is the place word of a decreasing hook subword of $w'$ and of $v'$, and $\pi_{\mathbf{k}'}(w') = v'$.
Since $w_1 = y$ is the largest unbarred letter in $w$, decreasing hook subwords of $w$ of length $\tau$ must use $y$ and this yields a bijection between decreasing hook subwords of $w$ of length $\tau$ and decreasing hook subwords of $w_R$ of length $\tau-1$.
Since $w_n = \crc{x}$ is the largest letter of  $w$, every decreasing hook subword of $w_R$ (resp.  $w$) of length  $\tau -1$ (resp.  $\tau$) is a special subword.
These facts, together with $(w')^{\stand} = (w_R)^{\stand}$, imply that $w'_{\mathbf{k}'}$ is the rightmost special subword of $w'$. Hence $\pi_-(w') = v'$.
By induction,
\be \label{e step 5b induction}
\text{$C_-(P_\mix(w')) =  P_\mix(v')$\ \ and\ \ $Q_\mix(w') = Q_\mix(v')$.}
\ee

Next, we prove
\be \label{e step5b w tableaux}
\begin{array}{lllll}
P_\mix(w_{LR}) &=& P_0^w, \\[1.3mm]
P_\mix(w_R) & = & P_\mix(w_{LR}) \xleftarrow{\mix} \crc{x} &=& P_0^w \sqcup {\tiny \tableau{\crc{x}}}_{(\tau-1,1)}, \\[1.3mm]
P_\mix(w') & = & P_0^w \sqcup {\tiny \tableau{\crc{y}}}_{(\tau-1,1)}, \\[1.3mm]
P_\mix(w) & = & P_\mix(w_R) \xleftarrow{\dualmix} y & = & P_0^w \sqcup {\tiny \tableau{y}}_{(\tau-1,1)} \sqcup {\tiny \tableau{\crc{x}}}_{(\tau,1)}, \\[1.3mm]
Q_\mix(w) &=& Q_\mix(w_R) \sqcup {\tiny \tableau{n}}_{(\tau,1)} &=& Q_\mix(w') \sqcup {\tiny \tableau{n}}_{(\tau,1)}.
\end{array}
\ee
The first line follows from the note at the end of Step 4 and the fact that $w_1 = y$ and $w_n = \crc{x}$ are the two largest letters of $w$. The second line is an easy consequence of the first. The third line follows from the second as $(w')^{\stand} = (w_R)^{\stand}$. For the fourth line, the dual mixed insertion of $y$ bumps the $\crc{x}$ in position $(\tau-1,1)$ and then places $\crc{x}$ in a new cell at position $(\tau,1)$. This dual mixed insertion computation, together with $(w')^{\stand} = (w_R)^{\stand}$, gives the last line.

We also have
\be \label{e step5b v tableaux}
\begin{array}{lllll}
P_\mix(v') & = & P_0^w \sqcup {\tiny \tableau{y}}_{(\tau-1,1)}, \\[1.3mm]
P_\mix(v) & = & P_\mix(v') \xleftarrow{\dualmix} x &=& P_0^w \sqcup {\tiny \tableau{y}}_{(\tau-1,1)} \sqcup {\tiny \tableau{x}}_{(\tau,1)}, \\[1.3mm]
Q_\mix(v) &=& Q_\mix(v') \sqcup {\tiny \tableau{n}}_{(\tau,1)},
\end{array}
\ee
where the first line follows from \eqref{e step 5b induction} and \eqref{e step5b w tableaux}; the second and third lines are then clear as the dual mixed insertion $P_\mix(v') \xleftarrow{\dualmix} x$ simply adds a new cell containing $x$ in position $(\tau,1)$.
The desired result \eqref{et rotate words-} now follows from \eqref{e step 5b induction}, \eqref{e step5b w tableaux}, and \eqref{e step5b v tableaux}.
\end{proof}

Theorem \ref{t main thm words} has the following corollary, which does not seem easy to prove directly.
\begin{corollary}
\label{t pi bijection}
The operators $\pi_-$ and $\pi_+$ are inverses of each other and define a bijection between color lowerable words and color raisable words.
\end{corollary}
%
%
%

\begin{example}
 \label{ex step 5} A possibility for Step 5a of the proof of Theorem \ref{t main thm words} is
\[
\begin{array}{rcl}
w  &=& \mathbf{5}\ \mathbf{3}\ \crc{2}\ \mathbf{1}\ \mathbf{\crc{6}}\ 4\ \mathbf{\crc{7}} \\
v  &=& \mathbf{7}\ \mathbf{5}\ \crc{2}\ \mathbf{3}\ \mathbf{1}\       4\ \mathbf{\crc{6}} \\
w' &=& \mathbf{5}\ \mathbf{3}\ \crc{2}\ \mathbf{1}\ \mathbf{\crc{6}}\ 4\ \\
v' &=& \mathbf{6}\ \mathbf{5}\ \crc{2}\ \mathbf{3}\ \mathbf{1}\       4,\
\end{array}
\]
where the bold letters indicate the rightmost special subwords of $w$ and $w'$ and the leftmost special subwords of $v$ and $v'$.
For this example, $\tau = \tau(w) = 5$, $\eta = \crc{7}$, $x = 7$, $\eta' = \crc{6}$, $y = 6$, and
\[
P_0^w =
{\tiny\tableau{1&\crc{2}&4\\3\\5}}\ \
\quad
P_\mix(w') =
{\tiny\tableau{1&\crc{2}&4\\3\\5\\\crc{6}}} \
\quad
P_\mix(v') =
{\tiny\tableau{1&\crc{2}&4\\3\\5\\6}}\ \ \
\quad
P_\mix(w) =
{\tiny\tableau{1&\crc{2}&4\\3\\5\\\crc{6}\\\crc{7}}}\
\quad
P_\mix(v) =
{\tiny\tableau{1&\crc{2}&4\\3\\5\\\crc{6}\\7}}.
\]

A possibility for Step 5b is
\[
\begin{array}{rcl}
w  &=& \mathbf{6}\ \mathbf{3}\ \crc{2}\ \mathbf{1}\ \mathbf{\crc{5}}\ 4\ \mathbf{\crc{7}} \\
v  &=& \mathbf{7}\ \mathbf{6}\ \crc{2}\ \mathbf{3}\ \mathbf{1}\       4\ \mathbf{\crc{5}} \\
w' &=& \hphantom{\mathbf{6}}\ \mathbf{3}\ \crc{2}\ \mathbf{1}\ \mathbf{\crc{5}}\ 4\ \mathbf{\crc{6}}\ \\
v' &=& \hphantom{\mathbf{6}}\ \mathbf{6}\ \crc{2}\ \mathbf{3}\ \mathbf{1}\ 4\      \mathbf{\crc{5}}.\
\end{array}
\]
For this example, $\tau = \tau(w) = 5$, $\eta = \crc{7}$, $x = 7$, $\eta' = 6$, $y = 6$, and
\[
P_0^w =
{\tiny\tableau{1&\crc{2}&4\\3\\\crc{5}}}\ \
\quad
P_\mix(w') =
{\tiny\tableau{1&\crc{2}&4\\3\\\crc{5}\\\crc{6}}}\
\quad
P_\mix(v') =
{\tiny\tableau{1&\crc{2}&4\\3\\\crc{5}\\6}}\ \ \
\quad
P_\mix(w) =
{\tiny\tableau{1&\crc{2}&4\\3\\\crc{5}\\6\\\crc{7}}}\
\quad
P_\mix(v) =
{\tiny\tableau{1&\crc{2}&4\\3\\\crc{5}\\6\\7}}.
\]
\end{example}

\subsection{Completing the proof of Hook Kronecker Rule I}
\label{ss completing the proof}

Recall that  $\oplus$ denotes concatenation of tableaux and  $\ke$ denotes plactic equivalence (see \textsection\ref{ss plactic equivalence}).
In this subsection we prove
\begin{theorem} \label{t plain commutes with pi-}
For any color lowerable word $w$, $w^\plain \ke \pi_-(w)^\plain$.
\end{theorem}
This, together with Theorem \ref{t main thm words} and Proposition \ref{p plain basics} (i), proves Theorem \ref{t plain commutes with color lowering}.

We give a lemma and then proceed with the proof of Theorem \ref{t plain commutes with pi-}. These are heavy in notation, so it is helpful to follow along with Example \ref{ex plain commutes with pi-}.
\begin{lemma} \label{l plain commutes with pi-}
Suppose $w$ is a color lowerable word and $\mathbf{k} = k_1 \cdots k_\tau$ is the place word of its rightmost special subword. Set $v = \pi_-(w)$ and  $\crc{x} = \sw(w)$. Let $i$ be such that $w_{k_i}$ is the leftmost barred letter of $w_\mathbf{k}$. Set $w_L = w_1 w_2 \cdots w_{k_i-1}$, $w_R = w_{k_i} w_{k_i+1} \cdots w_n$, $v_L = v_1 v_2 \cdots v_{k_i}$, and $v_R = v_{k_i+1} v_{k_i+2} \cdots v_n$. Then
\begin{align}
\label{el pi- left half}
\pi_-\big(\sub_\varnothing(w_L) \ \crc{x}\big) &= \sub_\varnothing(v_L),  \\
\label{el pi+ right half}
\sub_\crcempty(w_R) &= \pi_+\big(x \ \sub_\crcempty(v_R)\big).
\end{align}
Moreover,
\begin{align}
 \label{el ptab equality1}
{\tiny\tableau{x}} \oplus P\big(\sub_\varnothing(w_L)\big) &\ke P\big(\sub_\varnothing(v_L)\big), \\
 \label{el ptab equality2}
P\big(\sub_\crcempty(w_R)^*\big) &\ke P\big(\sub_\crcempty(v_R)^*\big) \oplus {\tiny\tableau{x}}.
\end{align}
\end{lemma}
\begin{proof}
By Propositions \ref{p standardization commutes} and \ref{p standardization hook subwords}, we can assume that $w$ is a colored permutation.
Set  $w' = \sub_\varnothing(w_L) \ \crc{x}$ and let $\mathbf{k}'$ be the place word of  $w'$ such that  $(w')_{\mathbf{k}'} = w_{k_1} \cdots w_{k_{i-1}} \crc{x}$ (this determines $\mathbf{k}'$ uniquely since we are assuming $w$ is a colored permutation).  One checks directly from the definitions that  $\pi_{\mathbf{k}'}(w') = \sub_\varnothing(v_L)$.

Note that every decreasing hook subword of  $w'$ of maximum possible length contains $ \crc{x}$.  It is then not hard to show that $w_\mathbf{k}$ being the rightmost special subword of $w$ implies $(w')_{\mathbf{k}'}$ is the rightmost special subword of $w'$.  This proves \eqref{el pi- left half}.  The proof of \eqref{el pi+ right half} is similar.

Theorem \ref{t main thm words} and \eqref{el pi- left half} imply $C_-(P_\mix(w')) = P_\mix(\sub_\varnothing(v_L))$.
It follows from the note of the previous paragraph that the last row of $P_\mix(w')$ consists of a single cell containing $ \crc{x}$.  Moreover,
$P_\mix\big(\sub_\varnothing(w_L)\big) = P\big(\sub_\varnothing(w_L)\big)$ and $P_\mix\big(\sub_\varnothing(v_L)\big) = P\big(\sub_\varnothing(v_L)\big)$ since these words consist of only unbarred letters.  These facts yield \eqref{el ptab equality1}.

A similar argument to the previous paragraph using \eqref{el pi+ right half} in place of \eqref{el pi- left half} yields
\[
P_\mix\big(\sub_\crcempty(w_R)\big) = P_\mix\big(\sub_\crcempty(v_R)\big) \sqcup {\tiny\tableau{\crc{x}}}_{(\tau-i+1,1)}.
\]
The plactic equivalence \eqref{el ptab equality2} then follows from $P_\mix\big(\sub_\crcempty(w_R)\big)^* = P_\mix\big(\sub_\crcempty(w_R)^*\big) = P\big(\sub_\crcempty(w_R)^*\big)$ and $\left(P_\mix\big(\sub_\crcempty(v_R)\big) \sqcup\, {\tiny\tableau{\crc{x}}}_{(\tau-i+1,1)}\right)^* = P_\mix\big(\sub_\crcempty(v_R)^*\big) \sqcup\, {\tiny\tableau{x}}_{(1,\tau-i+1)} \ke P\big(\sub_\crcempty(v_R)^*\big) \oplus {\tiny\tableau{x}}$ (here we have used Proposition \ref{p rev star mixed} (i)).
\end{proof}

\begin{proof}[Proof of Theorem \ref{t plain commutes with pi-}]
Maintain the notation of Lemma \ref{l plain commutes with pi-}. We compute
\[
\begin{array}{lcl}
w^\plain &=& \sub_\crcempty(w_L)^* \,\sub_\crcempty(w_R)^* \,\sub_\varnothing(w_L) \,\sub_\varnothing(w_R), \\
v^\plain &=& \sub_\crcempty(v_L)^* \,\sub_\crcempty(v_R)^* \,\sub_\varnothing(v_L) \,\sub_\varnothing(v_R).
\end{array}
\]
By Lemma \ref{l plain commutes with pi-}, we have
\begin{align}
\sub_\crcempty(w_R)^* \,\sub_\varnothing(w_L)
&\ke P\big(\sub_\crcempty(w_R)^*\big) \oplus P\big(\sub_\varnothing(w_L)\big) \notag \\
&\ke P\big(\sub_\crcempty(v_R)^*\big) \oplus {\tiny\tableau{x}} \oplus P\big(\sub_\varnothing(w_L)\big)
\label{e plain ptab} \\
&\ke P\big(\sub_\crcempty(v_R)^*\big) \oplus P\big(\sub_\varnothing(v_L)\big) \notag \\
&\ke \sub_\crcempty(v_R)^* \,\sub_\varnothing(v_L). \notag
\end{align}
This proves the theorem since $\sub_\crcempty(w_L)^* = \sub_\crcempty(v_L)^*$ and $\sub_\varnothing(w_R) = \sub_\varnothing(v_R)$.
\end{proof}

\begin{example}
\label{ex plain commutes with pi-}
Let us illustrate the proofs of Lemma \ref{l plain commutes with pi-} and Theorem \ref{t plain commutes with pi-} for the following choice of $w$:
\[
\begin{array}{cc@{\ \ \ \ }l}
w & = &\underbracket[0.3pt][1pt]{\mathbf{4}\ 1\ \crc{2}\ \mathbf{3}\ 6\ \mathbf{2}\ 3\ \crc{2}}_{w_L}\ \underbracket[0.3pt][1pt]{\mathbf{\crc{1}}\ \mathbf{\crc{1}}\ 1\ 3\ \crc{2}\ 3\ \mathbf{\crc{1}}\ \mathbf{\crc{4}}\ \mathbf{\crc{5}}\ 1\ \crc{1}\ 2}_{w_R}\ \\[5mm]
v & = &\underbracket[0.3pt][1pt]{\mathbf{5}\ 1\ \crc{2}\ \mathbf{4}\ 6\ \mathbf{3}\ 3\ \crc{2}\ \mathbf{2}}_{v_L}\ \underbracket[0.3pt][1pt]{ \mathbf{\crc{1}}\ 1\ 3\ \crc{2}\ 3\ \mathbf{\crc{1}}\ \mathbf{\crc{1}}\ \mathbf{\crc{4}}\ 1\ \crc{1}\ 2}_{v_R}\ \\[5mm]
w^\plain & = & \underbracket[0.3pt][1pt]{2\ 2}_{\mathclap{\sub_\crcempty(w_L)^*}}\ \underbracket[0.3pt][1pt]{\mathbf{1}\ \mathbf{1}\ 2\ \mathbf{1}\ \mathbf{4}\ \mathbf{5}\ 1}_{\sub_\crcempty(w_R)^*}\ \underbracket[0.3pt][1pt]{\mathbf{4}\ 1\ \mathbf{3}\ 6\ \mathbf{2}\ 3}_{\sub_\varnothing(w_L)}\ \hphantom{\mathbf{2}}\underbracket[0.3pt][1pt]{1\ 3\ 3\ 1\ 2}_{\sub_\varnothing(w_R)}\ \\[2mm]
v^\plain & = & \underbracket[0.3pt][1pt]{2\ 2}_{\mathclap{\sub_\crcempty(v_L)^*}}\ \ \ \,\underbracket[0.3pt][1pt]{\mathbf{1}\ 2\ \mathbf{1}\ \mathbf{1}\ \mathbf{4}\ 1}_{\sub_\crcempty(v_R)^*}\ \underbracket[0.3pt][1pt]{\mathbf{5}\ 1\ \mathbf{4}\ 6\ \mathbf{3}\ 3\ \mathbf{2}}_{\sub_\varnothing(v_L)}\ \underbracket[0.3pt][1pt]{1\ 3\ 3\ 1\ 2}_{\sub_\varnothing(v_R)}. \\
\end{array}
\]
The rightmost (resp. leftmost) special subword of $w$ (resp.  $v$) and the corresponding letters of $w^\plain$ (resp. $v^\plain$) are in bold.

We have
\[
\begin{array}{lcl}
\sub_\varnothing(w_L) \ \crc{x}& = & \mathbf{4}\ 1\ \mathbf{3}\ 6\ \mathbf{2}\ 3\ \mathbf{\crc{5}}\\
\sub_\varnothing(v_L)& = & \mathbf{5}\ 1\ \mathbf{4}\ 6\ \mathbf{3}\ 3\ \mathbf{2} \\[2mm]
\sub_\crcempty(w_R)        & = & \mathbf{\crc{1}}\ \mathbf{\crc{1}}\ \crc{2}\ \mathbf{\crc{1}}\ \mathbf{\crc{4}}\ \mathbf{\crc{5}}\ \crc{1}\\
x \ \sub_\crcempty(v_R) & = & \mathbf{5}\ \mathbf{\crc{1}}\ \crc{2}\ \mathbf{\crc{1}}\ \mathbf{\crc{1}}\ \mathbf{\crc{4}}\ \crc{1}.
\end{array}
\]
The rightmost (resp.  leftmost) special subwords are shown in bold in the first and third (resp. second and fourth) lines,
so \eqref{el pi- left half} and \eqref{el pi+ right half} are evident for this example.  The plactic equivalences \eqref{el ptab equality1} and \eqref{el ptab equality2} become
\[
{\tiny\tableau{5}} \oplus P\big({4}\ 1\ {3}\ 6\ {2}\ 3\big) \ke {\tiny\tableau{1&2&3\\3&6\\4\\5}} = P\big({5}\ 1\ {4}\ 6\ {3}\ 3\ {2} \big)
\]
\[
P\big(1\ 1\ 2\ 1\ 4\ 5 \ 1\big) = {\tiny\tableau{1&1&1&1&5\\2&4}} \ke P\big(1 \ 2\ 1 \ 1\ 4\ 1\big) \oplus {\tiny\tableau{5}}.
\]
Finally, \eqref{e plain ptab} becomes
\[
\sub_\crcempty(w_R)^* \,\sub_\varnothing(w_L) \ke
{\tiny\tableau{1&1&1&1&5\\2&4}}
\,\oplus\,
{\tiny\tableau{1&2&3\\3&6\\4}}
\ke
{\tiny\tableau{1&1&1&1\\2&4}}
\,\oplus\,
{\tiny\tableau{1&2&3\\3&6\\4\\5}}
\ke
\sub_\crcempty(v_R)^* \,\sub_\varnothing(v_L).
\]
%
\end{example}

\section{More hook Kronecker rules and their symmetries} \label{s more versions and symmetries of the Hook Kronecker Rule}
Here we give two variants of Hook Kronecker Rule I (\textsection\ref{ss more Kronecker rules}) and also show that this rule holds when $\nu$ is a skew shape (\textsection\ref{ss generalization to skew shapes}).
We show that the ``symmetry'' $\gcoef= g_{\lambda' \, \mu(d)' \, \nu}$ of Kronecker coefficients is evident from the hook Kronecker rules, while the symmetry $\gcoef= g_{\nu\, \mu(d)\, \lambda}$ does not seem to be (\textsection\ref{ss symmetries of the hook Kronecker rules}). Finally, we compare the hook Kronecker rules to the experiment in the introduction and to Lascoux's Kronecker Rule (\textsection\ref{ss comparison with Lascoux}).

\subsection{Hook Kronecker Rules I--III}
\label{ss more Kronecker rules}
Let $\lambda$ and $\nu$ be partitions of $n$ and let $A_\lambda$, $B_\nu$ be SYT of shapes $\lambda$, $\nu$, respectively. Define the following subsets of standard colored tableaux of size $n$:
\[
\begin{array}{lcl}
\text{CT}_{A_\lambda} &:=& \big\{T: T^\plain = A_\lambda\big\}, \\[1mm]
\text{CT}_{A_\lambda, d} &:=& \big\{T: T^\plain = A_\lambda,\ \tc(T) = d\big\}, \\[1mm]
\text{CT}_{A_\lambda, d}(\nu) &:=& \big\{T: T^\plain = A_\lambda,\ \tc(T) = d, \ \sh(T) = \nu \big\}.
\end{array}
\]
Define the following subsets of colored permutations of length $n$:
\[
\begin{array}{lcl}
\text{CW}_{A_\lambda} &:=& \big\{w: P(w^\plain) = A_\lambda\big\}, \\[1mm]
\text{CW}_{A_\lambda, d} &:=& \big\{w: P(w^\plain) = A_\lambda,\ \tc(w) = d\big\}, \\[1mm]
\text{CW}_{A_\lambda, d, B_\nu} &:=& \big\{w: P(w^\plain) = A_\lambda,\ \tc(w) = d,\ Q\nmix(w) = B_\nu\big\}.
\end{array}
\]
Further, define $\text{CT}^-_{A_\lambda}$ (resp. $\text{CT}^+_{A_\lambda}$) to be the subset of $\text{CT}^-_{A_\lambda}$ (resp. $\text{CT}^+_{A_\lambda}$) consisting of color raisable  (resp. lowerable) tableaux.  Define $\text{CT}^-_{A_\lambda, d}$, $\text{CW}^-_{A_\lambda}$, etc. similarly (for the sets of words, intersect with color raisable or lowerable words instead of tableaux).


\begin{corollary}[Hook Kronecker Rules I--III]\label{c kronecker words}
Let  $\lambda$ and $\nu$ be partitions of  $n$ and recall $\mu(d) = (n-d,1^d)$. Let $A_\lambda$ and $B_\nu$ be any SYT of shapes $\lambda$ and  $\nu$, respectively, as above.
The following sets of combinatorial objects have cardinality equal to the Kronecker coefficient $\gcoef$:
\begin{list}{\emph{(\Roman{ctr})}}{\usecounter{ctr} \setlength{\itemsep}{2pt} \setlength{\topsep}{3pt}}
\item $\text{CYT}^{\,-}_{\lambda, d}(\nu)$
\item  $\text{CT}^{\,-}_{A_\lambda, d}(\nu)$
\item  $\text{CW}^{\,-}_{A_\lambda, d, B_\nu}$.
\end{list}
\end{corollary}
\begin{proof}
We have already shown that the cardinality of (I) is $\gcoef$.
By Theorem \ref{t plain commutes with color lowering} and Proposition \ref{p plain basics} (i), the color lowering operator $C_-$ restricted to $\text{CT}_{A_\lambda}$ gives a  bijection from $\text{CT}^-_{A_\lambda, d}(\nu)$ to $\text{CT}^+_{A_\lambda, d+1}(\nu)$ for  $d \in \{0,1,\ldots,n-1\}$.
The proof of Hook Kronecker Rule I carries over to this setting with little change; to adapt the proof of Proposition \ref{p arm leg}, all that is required is to note that the number of Littlewood-Richardson tableaux of content $\lambda$ and shape $\alpha \oplus (\nu/\alpha')$ is the same as the number of standard tableaux of shape $\alpha \oplus (\nu/\alpha')$ that are plactic equivalent to $A_\lambda$. Hence $\gcoef = |\text{CT}^-_{A_\lambda, d}(\nu)|$.
Finally, (II) and (III) have the same cardinality as $P\nmix\big(\text{CW}^-_{A_\lambda, d, B_\nu}\big) = \text{CT}^-_{A_\lambda, d}(\nu)$.
\end{proof}
The set of colored words $\text{CYW}_{\lambda,d}$ defined in the introduction is related to the $\text{CW}_{Z_\lambda^\stand,d}$ defined above by standardizing: $(\text{CYW}_{\lambda,d})^\stand = \text{CW}_{Z_\lambda^\stand,d}$.  Figure \ref{f cw} (after standardizing) illustrates Hook Kronecker Rule III for  $\lambda = (3,1,1)$, $d= 2$, and all  $B_\nu$.

In addition to the three descriptions above, we also point out that the tableaux $A_\lambda$ and $B_\nu$ in the definition of $\text{CW}_{A_\lambda, d, B_\nu}$ have many descriptions:
\be \label{e A lambda B nu}
\begin{array}{lcl}
A_\lambda\ =\ P(w^\plain) = P_\lr(w^{\barrev}) = Q(w^{\barrev\;\inv\;\bneg}) = Q\nmix(w^{\barrev\;\inv}) = Q\nmix(w^{\rev\;\nobarrev\;\inv}) \\[2mm]
B_\nu \ = \ Q\nmix(w) = Q(w^{\bneg})= P(w^{\inv\; \barrev \; \plain}).
\end{array}
\ee
These equalities are by Propositions \ref{p neg and mixed insertion}, \ref{p plain basics}, and \ref{p plain and neg}.

\begin{example}\label{ex CW}
Let $A_\lambda = {\tiny\tableau{1&4&5&6\\2\\3\\}}$ and $B_\nu = {\tiny\tableau{1&3&6\\2&5\\4\\}}$. The first line below gives the nonempty sets $\text{CW}^-_{A_\lambda, d, B_\nu}$ for all $d$; the second line gives the sets $\text{CT}^-_{A_\lambda, d}(\nu)$, i.e., the mixed insertion tableaux of the words on the first line; the third line gives the tableaux $P\nmix(w^{\barrev\;\inv})$ for the $w$ on the first line (these are the subject of Proposition \ref{p CT d b nu}, below):
\[
{\tiny
\begin{array}{lllll}
\text{\small$\{4\ 2\ 5\ \crc{3}\ 1\ 6\}$}&
\text{\small$\{4\ 1\ 5\ \crc{3}\ \crc{2}\ 6,\ 5\ \crc{3}\ 2\ \crc{4}\ 1\ 6\}$}&
\text{\small$\{5\ \crc{3}\ 6\ \crc{4}\ \crc{2}\ 1,\ 5\ \crc{3}\ 1\ \crc{4}\ \crc{2}\ 6\}$}&
\text{\small$\{5\ \crc{3}\ 6\ \crc{4}\ \crc{2}\ \crc{1}\}$}\\[2mm]
\left\{\tableau{
1&\crc{3}&6\\2&5\\4\\}\right\}&
\left\{\tableau{
1&\crc{3}&6\\\crc{2}&5\\4\\}\quad\tableau{
1&\crc{3}&6\\2&\crc{4}\\5\\}\right\}&
\left\{\tableau{
1&\crc{2}&\crc{3}\\\crc{4}&6\\5\\}\quad\tableau{
1&\crc{3}&6\\\crc{2}&\crc{4}\\5\\}\right\}&
\left\{\tableau{
\crc{1}&\crc{2}&\crc{3}\\\crc{4}&6\\5\\}\right\}\\[7mm]
\left\{\tableau{
1&3&\crc{4}&6\\2\\5\\}\right\}&
\left\{\tableau{
1&3&\crc{4}&6\\2\\\crc{5}\\}\quad\tableau{
1&\crc{2}&\crc{4}&6\\3\\5\\}\right\}&
\left\{\tableau{
1&\crc{2}&3&\crc{5}\\\crc{4}\\6\\}\quad\tableau{
1&\crc{2}&\crc{5}&6\\3\\\crc{4}\\}\right\}&
\left\{\tableau{
1&\crc{2}&3&\crc{6}\\\crc{4}\\\crc{5}\\}\right\}.
\end{array}}
\]

Here are the corresponding color lowerable sets of words and tableaux ($\text{CW}^+_{A_\lambda, d, B_\nu}$, $\text{CT}^+_{A_\lambda, d}(\nu)$, and $\{P\nmix(w^{\barrev\;\inv}) : w \in \text{CW}^+_{A_\lambda, d, B_\nu}\}$):
\[
{\tiny
\begin{array}{lllll}
\text{\small$\{2\ \crc{3}\ 5\ \crc{4}\ 1\ 6\}$}&
\text{\small$\{1\ \crc{3}\ 5\ \crc{4}\ \crc{2}\ 6,\ \crc{3}\ \crc{4}\ 2\ \crc{5}\ 1\ 6\}$}&
\text{\small$\{\crc{3}\ \crc{4}\ 6\ \crc{5}\ \crc{2}\ 1,\ \crc{3}\ \crc{4}\ 1\ \crc{5}\ \crc{2}\ 6\}$}&
\text{\small$\{\crc{3}\ \crc{4}\ 6\ \crc{5}\ \crc{2}\ \crc{1}\}$}\\[2mm]
\left\{\tableau{
1&\crc{3}&6\\2&5\\\crc{4}\\}\right\}&
\left\{\tableau{
1&\crc{3}&6\\\crc{2}&5\\\crc{4}\\}\quad\tableau{
1&\crc{3}&6\\2&\crc{4}\\\crc{5}\\}\right\}&
\left\{\tableau{
1&\crc{2}&\crc{3}\\\crc{4}&6\\\crc{5}\\}\quad\tableau{
1&\crc{3}&6\\\crc{2}&\crc{4}\\\crc{5}\\}\right\}&
\left\{\tableau{
\crc{1}&\crc{2}&\crc{3}\\\crc{4}&6\\\crc{5}\\}\right\}\\[7mm]
\left\{\tableau{
1&3&\crc{4}&6\\\crc{2}\\5\\}\right\}&
\left\{\tableau{
1&3&\crc{5}&6\\\crc{2}\\\crc{4}\\}\quad\tableau{
\crc{1}&\crc{2}&\crc{4}&6\\3\\5\\}\right\}&
\left\{\tableau{
\crc{1}&3&\crc{4}&\crc{5}\\\crc{2}\\6\\}\quad\tableau{
\crc{1}&\crc{4}&\crc{5}&6\\\crc{2}\\3\\}\right\}&
\left\{\tableau{
\crc{1}&3&\crc{5}&\crc{6}\\\crc{2}\\\crc{4}\\}\right\}.
\end{array}}\]
\end{example}

\subsection{A generalization to skew shapes} \label{ss generalization to skew shapes}
Here we show that Hook Kronecker Rule I generalizes in a straightforward way to the case $\nu$ is a skew shape.
For $\beta$ a skew shape of size $n$ and $\lambda, \mu \vdash n$, the Kronecker coefficient $g_{\lambda \mu \beta}$ is defined by
\[
g_{\lambda \mu \beta} = \langle s_\lambda * s_\mu, s_\beta \rangle.
\]
The definitions of ${}^\plain$, colored Yamanouchi tableaux, color lowerable, color raisable, and $C_-$ all carry over to colored tableaux of skew shape without change.

For a tableau $B$ and (skew) shape $\theta$ contained in the shape of $B$, $B_\theta$ denotes the (skew) subtableau of $B$ obtained by restricting to the shape $\theta$.
\begin{lemma} \label{l skew plain}
Let $T$ be a $\text{CT}^{\,\prec}$ of shape $\nu/ \kappa$. Let $B_\kappa$ be a $\text{CT}^{\,\prec}$ of shape $\kappa$ which contains only barred letters $<$ all letters of $T$ and define $B$ to be the union of $B_\kappa$ and $T$ (hence $B$ is a $\text{CT}^{\,\prec}$ of shape $\nu$). Then $T^\plain \ke (B^\plain)_{\lambda/\kappa'}$, where $\lambda =\sh(B^\plain)$.
\end{lemma}
\begin{proof}
Since $B^\plain$ can be computed by inserting $\reading(\sub_\varnothing(B))$ into $\sub_{\crcempty}(B)^*$, and the insertion of a letter $x$ into an ordinary tableau does not affect letters $<x$,
\be \label{e B plain kappa}
(B^\plain)_{\kappa'} = (B_\kappa)^*.
\ee
Set $w = \reading(\sub_{\crcempty}(B)^*) \, \reading(\sub_\varnothing(B))$, so by definition, $B^\plain \ke w$. Then, letting $m$ be the smallest letter of $T^\plain$, we have
\[
(B^\plain)_{\lambda/\kappa'} = \sub_{\geq m}(B^\plain) \ke \sub_{\geq m}(w) = \reading(\sub_{\crcempty}(T)^*) \, \reading(\sub_\varnothing(T)) \ke T^\plain,
\]
where the first equality is by \eqref{e B plain kappa}, the first plactic equivalence is a well-known property of the plactic monoid, and the second equality is straightforward from definitions.
\end{proof}

\begin{proposition} \label{p skew main theorem}
Theorem \ref{t plain commutes with color lowering} generalizes to skew tableaux: for any color lowerable tableau $T$ of skew shape $\nu / \kappa$, $T^\plain = C_-(T)^\plain$.
\end{proposition}
\begin{proof}
Using Lemma \ref{l skew plain} and with $B$ and  $\lambda$ as in the lemma, we compute
\[
T^\plain \ke (B^\plain)_{\lambda/\kappa'} = (C_-(B)^\plain)_{\lambda/\kappa'} \ke C_-(T)^\plain,
\]
where the equality is by Theorem \ref{t plain commutes with color lowering}.
The result follows since $T^\plain$ and $C_-(T)^\plain$ have straight-shape.
\end{proof}

The proof of Proposition \ref{p arm leg} generalizes to the case $\nu$ is skew shape with little change (the generalized Littlewood-Richardson coefficient  $c_{\alpha \beta}^\gamma = \langle s_\alpha s_\beta , s_\gamma \rangle$ makes sense for any skew shapes $\alpha, \beta, \gamma$). Then, in view of Proposition \ref{p skew main theorem}, the proof of Hook Kronecker Rule I carries over to this case as well. Hence
\begin{corollary}[Hook Kronecker Rule IV]
The Kronecker coefficient $g_{\lambda\,\mu(d)\,\beta}$ is equal to the number of color raisable Yamanouchi tableaux of content $\lambda$, total color $d$, and shape $\beta$.
\end{corollary}

\subsection{Symmetries of the hook Kronecker rules} \label{ss symmetries of the hook Kronecker rules}
The Weyl group $D_3$ (which is isomorphic to $\S_4$) acts on triples of partitions of  $n$ by permuting them and transposing an even number of them.  Kronecker coefficients are invariant under this action, i.e., $g_{\lambda \mu \nu} = g_{\theta(\lambda, \mu, \nu)}$ for any  $\theta \in D_3$.  What we actually want to consider here is this action restricted to the subset of triples for which our rules apply, i.e. those with $\mu$ a hook shape:
the subgroup of $D_3$ taking this subset to itself is isomorphic to the dihedral group of order 8.

As far as we can tell, only 2 of these 8 symmetries can be seen from the hook Kronecker rules: $(\lambda, \mu, \nu) \mapsto (\lambda', \mu', \nu)$ and (of course) the identity.
\begin{proposition}
We have the following bijections of sets of colored permutations:
\begin{align}
\text{CW}_{A, d, B} &\xrightarrow{{}^{\rev\;*}} \text{CW}_{{A}^{\transpose}, n-d, {B}^{\evac}} \label{ep1 rev star}\\
\text{CW}^{\,-}_{A, d, B} &\xrightarrow{{}^{\rev\;*}} \text{CW}^{\,+}_{{A}^{\transpose}, n-d, {B}^{\evac}} \label{ep2 rev star}
\end{align}
\end{proposition}
\begin{proof}
The bijection \eqref{ep1 rev star} follows from Proposition \ref{p rev star mixed} and Proposition \ref{p plain and neg} (iv).  Since the automorphism $w \mapsto w^{\rev\;*}$ of colored permutations identifies leftmost special subwords with rightmost special subwords, \eqref{ep2 rev star} follows from \eqref{ep1 rev star}.
\end{proof}

Regarding the symmetry $(\lambda, \mu, \nu) \mapsto (\nu, \mu, \lambda)$, we have
\begin{proposition}\label{p CT d b nu}
The subset of standard colored tableaux
\[
\begin{array}{lcl}
\text{CT}_{d, B_\nu}(\lambda) &:=& \left\{P\nmix(w^{\barrev\; \inv}): w \in \text{CW}_{A_\lambda, d, B_\nu} \right\}  \\
\end{array}
\]
does not depend on the choice of $A_\lambda$.  Therefore $\text{CT}_{d, B_\nu}(\lambda)$ is a set of standard colored tableaux of shape $\lambda$ with cardinality $\gcoef + g_{\lambda\, \mu(d-1)\, \nu}$.
\end{proposition}
\begin{proof}
Let $w \in \text{CW}_{A_\lambda, d, B_\nu}$ and set $v = w^{\barrev\; \inv}$.   \vspace{1mm}
By Proposition \ref{p plain basics} (iii), $B_\nu = Q\nmix(w) = Q\nmix(v^{\inv \; \barrev}) = P(v^{\barud \; \barrev\; \plain})$. Then  $B_\nu$ can be computed in terms of $P\nmix(v)$ as described in Proposition \ref{p plain basics} (iv).  This gives a definition of $\text{CT}_{d, B_\nu}(\lambda)$ that depends on $d, B_\nu, \lambda$, but not on $A_\lambda$.
\end{proof}

This proposition given, we now obtain a bijection between $\text{CT}_{A_\lambda, d}(\nu)$ and $\text{CT}_{d, B_\nu}(\lambda)$ via $\text{CT}_{A_\lambda, d}(\nu) \xleftarrow{\cong} \text{CW}_{A_\lambda, d, B_\nu} \xrightarrow{\cong} \text{CT}_{d, B_\nu}(\lambda)$.  See Example \ref{ex CW}.
It may be possible to describe this bijection directly, but we do not know how to do this and, in view of this example, it will not be easy.
A related difficult problem is to give a direct definition of the partition $\text{CT}_{d, B_\nu}(\lambda) = \text{CT}^-_{d, B_\nu}(\lambda) \sqcup \text{CT}^+_{d, B_\nu}(\lambda)$ induced from the partition
$\text{CT}_{A_\lambda, d}(\nu) = \text{CT}^-_{A_\lambda, d}(\nu) \sqcup \text{CT}^+_{A_\lambda, d}(\nu)$
via this bijection.
Example \ref{ex CW} shows that the subset of $\text{CT}_{d, B_\nu}$ consisting of color raisable tableaux does not, in general, have cardinality  $\gcoef$.
We have therefore convinced ourselves that the equality  $\gcoef = g_{\nu \mu(d) \lambda}$ is difficult to see from our rules.

\subsection{Comparison of the hook Kronecker rules with Lascoux's Kronecker Rule}
\label{ss comparison with Lascoux}
We now compare Hook Kronecker Rules II and III to the experiment in the introduction and to Lascoux's Kronecker Rule \cite{Lascoux}.
This comparison is better made with the ``reverse'' of our rules, which we now compute.
Define ${}^\plainr$ by $w^\plainr := w^{\rev\;\plain\;\rev}$ (this shuffles barred letters right instead of left).
Let $\lambda$, $\nu$, $A_\lambda$, $B_\nu$ be as in  \textsection\ref{ss more Kronecker rules}.
\begin{align}
\text{CW}^\rev_{A_\lambda, d, B_\nu} :=&\ \big(\text{CW}_{{A_\lambda}^\transpose, d, {B_\nu}^{\evac\;\transpose}}\big)^{\rev} \notag\\
 =&\ \big\{w^\rev: P(w^{\plain}) = {A_\lambda}^\transpose,\ \tc(w) = d,\ Q\nmix(w) = {B_\nu}^{\evac\;\transpose}\big\}\notag \\
 =&\ \big\{w: P(w^{\rev\;\plain}) = {A_\lambda}^\transpose,\ \tc(w^{\rev}) = d,\ Q\nmix(w^{\rev}) = {B_\nu}^{\evac\;\transpose}\big\}\notag \\
 =&\ \big\{w: P(w^{\plainr}) = A_\lambda,\ \tc(w) = d,\ Q\nmix(w) = {B_\nu}\big\}  \label{e CW rev1}\\
 =&\ \big\{w: P(w^{\plainr}) = A_\lambda,\ \tc(w) = d,\ Q(w^{\bneg}) = {B_\nu}\big\} \label{e CW rev2}.
\end{align}
The second to last equality is by \eqref{e rev} and Proposition \ref{p rev star mixed} (iv), and the last equality is by Proposition \ref{p neg and mixed insertion}.
Increasing hook subwords and special increasing subwords can be defined in a similar way to their decreasing counterparts.
Then the set $\text{CW}^{\rev \; -}_{A_\lambda, d, B_\nu} := \big(\text{CW}^-_{A_\lambda^\transpose, d, {B_\nu}^{\evac\;\transpose}}\big)^{\rev}$ (which has the desired cardinality  $\gcoef$) can be defined directly as
\be \label{e rev color raisable}
\parbox{14cm}{the subset of $\text{CW}^\rev_{A_\lambda, d, B_\nu}$ consisting of those
words  $w$ such that the largest letter of any special increasing  subword of $w$ is unbarred.}
\ee

Define the following subsets of colored permutations (L stands for Lascoux)
\begin{align}
\text{CWL}_{A_\lambda, d} &:= \big\{w: P(w^{\barrev\;\plainr}) = A_\lambda,\ \tc(w) = d \big\}, \notag \\
\text{CWL}_{A_\lambda, d, B_\nu} &:= \big\{w: P(w^{\barrev\;\plainr}) = A_\lambda,\ \tc(w) = d,\ Q(w) = B_\nu\big\}, \notag \\
\text{CWL}^-_{A_\lambda, d} &:= \big\{w: P(w^{\barrev\;\plainr}) = A_\lambda,\ \tc(w) = d,\ \text{$w_n$ is unbarred}
\big\}, \notag \\
\text{CWL}^-_{A_\lambda, d, B_\nu} &:= \big\{w: P(w^{\barrev\;\plainr}) = A_\lambda,\ \tc(w) = d,\ Q(w) = B_\nu, \ \text{$w_n$ is unbarred}
\big\}. \label{e CWLB}
\end{align}
For an object $w$ in the alphabet $\mathcal{A}$, define $w^\varnothing$ to be the object in the alphabet of ordinary letters obtained from $w$ by removing all bars; also, for a set $W$ of colored objects, define $W^\varnothing$ to be the multiset $\lms w^\varnothing : w \in W \rms$.
We claim that when $\mu$ is the hook shape $\mu(d)$, the multiset \eqref{e lambda circ mu} from the introduction is related to the $\text{CWL}$ by
\be \label{e u circ v CWL}
\Gamma_\lambda \circ \Gamma_\mu = \big(\text{CWL}^-_{A_\lambda, d}\big)^\varnothing.
\ee
Right
multiplying\footnote{We adopt the convention for multiplying permutations in which  $u \circ s_i$ is obtained from  $u$ by swapping letters  $u_i$ and  $u_{i+1}$, where $s_i$ is the transposition $(i\ \,i\!+\!1)$.}
a permutation  $u$ by a permutation $v$ such that $P(v) = Z_{\mu(d)}^\stand$
is the same as reversing the subword  $u_{n-d} u_{n-d+1} \cdots u_n$ of  $u$ and then shuffling $u_n,u_{n-1}, \ldots ,u_{n-d+1}$ to the left, into the rest of the word.
By placing bars on the letters $u_{n},\ldots,u_{n-d+1}$, we obtain a colored word  $w$
of total color  $d$ such that $w^{\barrev\;\plainr} = u$ and $w_n$ is unbarred. This verifies \eqref{e u circ v CWL}.

\begin{example}
If  $d=2$,
\[\phantom{\circ v} u =5\ 2\ 7\ 1\ 4\ 6\ 3 \ \ \ \text{and}\ \ \ v=7\ 1\ 2\ 6\ 3\ 4\ 5,\]
then
\[u \circ v = 3\ 5\ 2\ 6\ 7\ 1\ 4\ \ \ \text{and}\ \ \ w = \crc{3}\ 5\ 2\ \crc{6}\ 7\ 1\ 4.\]

As a further example,  observe that $\text{CWL}_{Z_{(3,1,1)}^{\stand}, 2}$  is the result of applying ${}^{\stand \; \rev \; \rev_\crcempty}$ to the words in Figure \ref{f cw};  the bottom six rows correspond to the subset  $\text{CWL}^-_{Z_{(3,1,1)}^{\stand}, 2}$.
\end{example}

Hence the half of Lascoux's Kronecker Rule concerning property (B) becomes
\[ \text{\emph{For any hook shape $\lambda$, $d \in \{0,1,\cdots,n-1\}$, and $B_\nu \in \text{SYT}(\nu)$, $\gcoef = |\text{CWL}^-_{Z_\lambda^{\stand}, d, B_\nu}|$.}}\]
The similar forms of Lascoux's Kronecker Rule and the reverse of Hook Kronecker Rule III are then apparent by comparing \eqref{e CW rev1}, \eqref{e CW rev2}, and \eqref{e rev color raisable} with \eqref{e CWLB}.

We still do not fully understand the relationship between these rules, however.
For example, the multisets of SYT $P\big(\text{CWL}^-_{Z_\lambda^{\stand}, d}\big)^\varnothing$ and $P_\mix\big(\text{CW}^{\rev\;-}_{Z_\lambda^\stand, d}\big)^\varnothing $ are equal when $\lambda$ is a hook shape.  However, we only know how to prove this by giving an explicit description of both multisets and then checking that they are the same.  Moreover, for general  $\lambda$, these multisets seem to be quite close; in fact, the tableaux in $P_\mix\big(\text{CW}^{\rev\;-}_{Z_\lambda^\stand, d}\big)$ were originally found by making slight modifications to those in $P\big(\text{CWL}^-_{Z_\lambda^{\stand}, d}\big)$.

\begin{remark}
Be aware that, although $\big(\text{CWL}^-_{Z_\lambda^{\stand}, d}\big)^\varnothing$  is a union of Knuth equivalence classes, $\text{CWL}^-_{Z_\lambda^{\stand}, d} $ is not in the following sense:  a Knuth transformation between two elements of  this multiset, say $\cdots x z y \cdots  \tto  \cdots z x y \cdots$,  may correspond to  a transformation of the form $\cdots \crc{x} z y \cdots  \tto  \cdots \crc{z} x y \cdots$  rather than $\cdots \crc{x} z y \cdots  \tto  \cdots z  \crc{x} y \cdots$ in $\text{CWL}^-_{Z_\lambda^{\stand}, d}$.   This is part of the difficulty in Problem \ref{p Lascouxs rule}, below.
\end{remark}

We believe that the tableaux in $P_\mix\big(\text{CW}^{\rev\;-}_{Z_\lambda^\stand, d}\big)$ are really the correct combinatorial objects for Kronecker coefficients for one hook shape, but we are not entirely sure that the words $\text{CWL}^-_{Z_\lambda^{\stand}, d}$ should be given up in favor of $\text{CW}^{\rev\;-}_{Z_\lambda^\stand, d}$. We therefore suggest the following problem, which  may help uncover a deeper relationship between Lascoux's Kronecker Rule and the hook Kronecker rules.
\begin{problem}  \label{p Lascouxs rule}
Find a nice proof of the fact that $P\big(\text{CWL}^-_{Z_\lambda^{\stand}, d}\big)^\varnothing = P_\mix\big(\text{CW}^{\;\rev\;-}_{Z_\lambda^\stand, d}\big)^\varnothing$ when $\lambda$ is a hook shape. For general $\lambda$, find an explicit bijection between $\text{CWL}^-_{Z_\lambda^{\stand}, d}$ and $P_\mix\big(\text{CW}^{\;\rev\;-}_{Z_\lambda^\stand, d}\big)$. For instance, such a bijection might modify these words in a simple way and then apply Schensted or mixed insertion, or might apply a new kind of insertion algorithm.
\end{problem}
\section*{Acknowledgments}
I am extremely grateful to John Stembridge for his generous advice and many detailed discussions.  I thank Kunyu Chen and Michael Bennett for their help typing and typesetting figures.

\bibliographystyle{plain}
\bibliography{mycitations}
\end{document}